\documentclass{amsart}

\usepackage{amssymb, amsmath,mathtools}
\usepackage{mathrsfs}
\usepackage{amscd}
\usepackage{verbatim}
\usepackage{stmaryrd}
\usepackage[dvipsnames]{xcolor}
\usepackage{a4wide}
\usepackage{subfig}

\usepackage{enumerate}

\usepackage[colorlinks,linkcolor={blue},citecolor={blue},urlcolor={purple},]{hyperref}

\usepackage[colorinlistoftodos,prependcaption,textsize=tiny]{todonotes}

\usepackage{forest}
\usepackage{tikz}
\tikzset{>=latex}

\usepackage{tabularx}
\newcolumntype{L}{>{\arraybackslash}X}
\usepackage{multirow}

\DeclareMathOperator*{\argmax}{arg\,max}

\theoremstyle{plain}
\newtheorem{theorem}{Theorem}[section]
\theoremstyle{remark}
\newtheorem{remark}[theorem]{Remark}

\theoremstyle{plain}
\newtheorem{corollary}[theorem]{Corollary}
\newtheorem{lemma}[theorem]{Lemma}
\newtheorem{proposition}[theorem]{Proposition}
\newtheorem{definition}[theorem]{Definition}

\newtheorem{assumption}[theorem]{Assumption}

\numberwithin{equation}{section}

\def\Z{{\mathbb Z}}

\def\R{{\mathbb R}}

\newcommand{\E}{{\mathbb E}}
\renewcommand{\P}{{\mathbb P}}
\newcommand{\F}{{\mathscr F}}

\newcommand{\g}{\gamma}

\newcommand{\om}{\omega}
\renewcommand{\O}{\Omega}

\renewcommand{\a}{\kappa}

\newcommand{\Mart}{\mathcal{M}}
\newcommand{\Det}{\mathcal{D}}

\newcommand{\loc}{{\rm loc}}

\newcommand{\Tor}{\mathbb{T}}
\newcommand{\Dom}{\mathcal{O}}

\newcommand{\A}{\mathcal{A}}

\usepackage{stmaryrd}

\newcommand{\one}{{{\bf 1}}}
\newcommand{\embed}{\hookrightarrow}
\newcommand{\s}{\delta}

\renewcommand{\div}{{\mathrm{div}}}

\newcommand{\supp}{\mathrm{supp}\,}
\renewcommand{\l}{\langle}
\renewcommand{\r}{\rangle}

\newcommand{\norm}[1]{{\left\vert\kern-0.25ex\left\vert\kern-0.25ex\left\vert #1
    \right\vert\kern-0.25ex\right\vert\kern-0.25ex\right\vert}}

\renewcommand{\emptyset}{\varnothing}

\newcommand{\m}{a}

\newcommand{\wh}{\widehat}
\newcommand{\dd}{\mathrm{d}}

\newcommand{\Borel}{\mathscr{B}}

\newcommand{\V}{\mathcal{V}}

\newcommand{\ellip}{\nu}

\newcommand{\X}{\mathcal{X}}
\newcommand{\Sp}{Z_{r,\eta}}
\newcommand{\Y}{\mathcal{Y}}
\newcommand{\T}{\mathcal{J}}
\newcommand{\fun}{\pi}
\newcommand{\KN}{\mathcal{B}_N}
\newcommand{\LN}{\mathcal{L}_N}
\newcommand{\MRD}{\mathrm{MR}_{q,p}}
\newcommand{\vd}{v_{{\rm det}}}
\newcommand{\vdi}{v_{{\rm det},i}}
\newcommand{\vcn}{v_{{\rm cut}}^{(n)}}
\newcommand{\vcns}{v_{{\rm cut}}^{(n_*)}}

\newcommand{\non}{\mathcal{N}}

\newcommand{\parameter}{a}
\newcommand{\W}{\mathcal{O}}
\newcommand{\vdiff}{V}
\newcommand{\x}{\mathsf{X}}
\newcommand{\Br}{W}

\allowdisplaybreaks

\begin{document}

\author{Antonio Agresti}
\address{Institute of Science and Technology Austria (ISTA), Am Campus 1, 3400 Klosterneuburg, Austria} \email{antonio.agresti92@gmail.com}
\curraddr{Delft Institute of Applied Mathematics, Delft University of Technology, P.O.\ Box 5031, 2600 GA Delft, The Netherlands}

\thanks{The author has received funding from the European Research Council (ERC) under the Eu\-ropean Union’s Horizon 2020 research and innovation programme (grant agreement No 948819) \includegraphics[height=0.4cm]{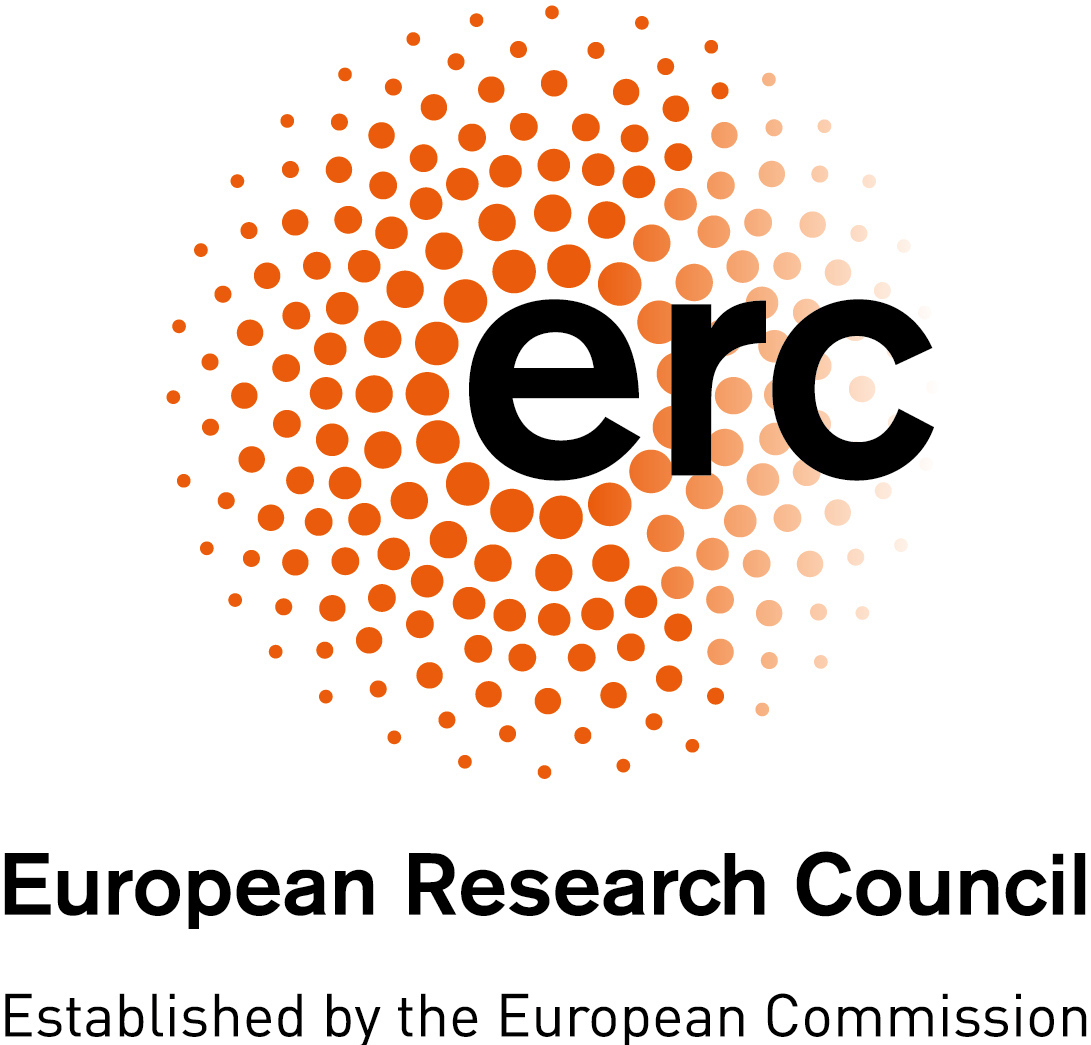}\,\includegraphics[height=0.4cm]{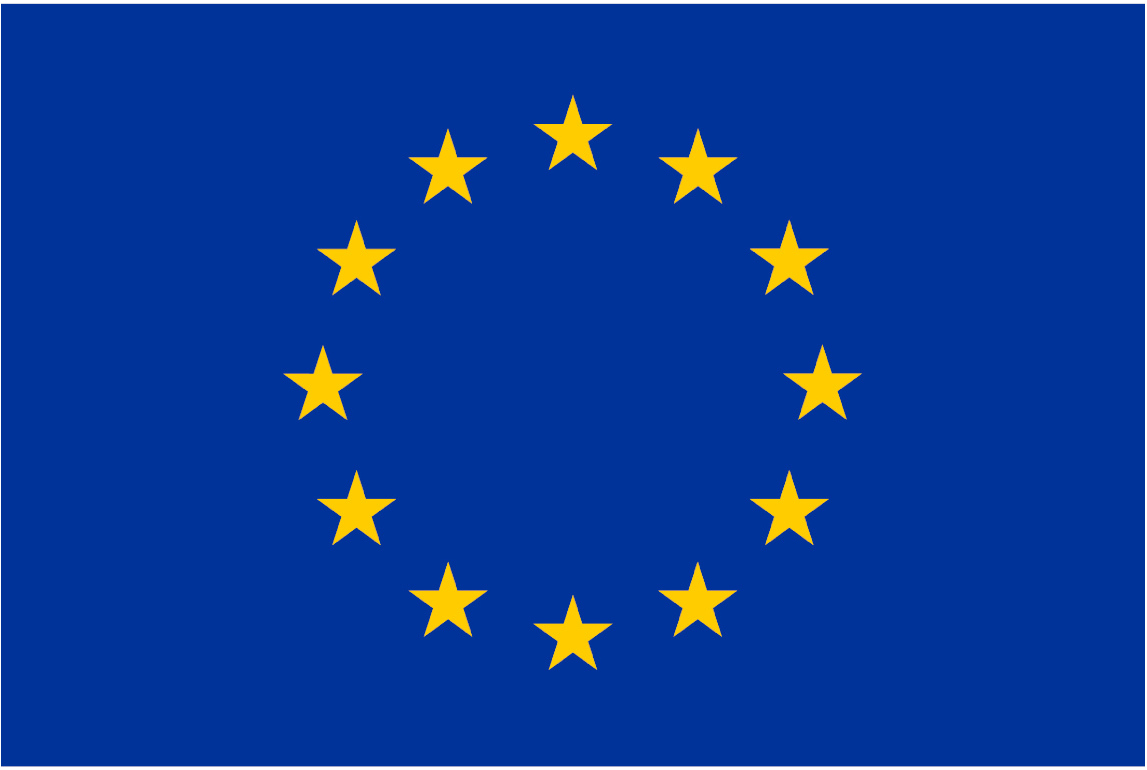}.}

\date\today

\title[Delayed blow--up for system of reaction-diffusion equations]{Delayed blow--up and enhanced diffusion by transport noise for systems of reaction--diffusion equations}



\begin{abstract}
This paper is concerned with the problem of regularization by noise of systems of reaction--diffusion equations with mass control. It is known that \emph{strong} solutions to such systems of PDEs may blow--up in finite time. Moreover, for many systems of practical interest, establishing whether the blow--up occurs or not is an open question.
Here we prove that a suitable multiplicative noise of transport type has a regularizing effect. More precisely, for both a sufficiently noise intensity and a high spectrum, 
the blow--up of strong solutions is delayed up to an arbitrary large time. Global existence is shown for the case of exponentially decreasing mass. 
The proofs combine and extend recent developments in regularization by noise and in the $L^p(L^q)$--approach to stochastic PDEs,  highlighting new connections between the two areas.
\end{abstract}

\maketitle

\tableofcontents

\section{Introduction}
\label{s:intro}

Reaction--diffusion equations arise in many branches of applied science such as biology and chemistry (see e.g.\ \cite{R84_global,P10_survey} and the references therein). 
A major challenge in the study of such equations is the presence of commonly superlinear source terms. 
Even in presence of dissipation of mass, which is sufficient to show global existence in the ODE case, blow--up in finite time of \emph{strong} solutions may occur, see \cite{PS97_blow_up} or \cite[Theorems 4.1 and 4.2]{P10_survey}. 
In addition, for many problems of practical interests, such as reversible chemical reactions (see Subsection \ref{ss:chemical_reactions} below), existence of global unique \emph{strong} solutions is still an open problem. 
However, this is only a first example, see also \cite[Section 7, Problem 1]{P10_survey} for further comments.

In this paper we show that suitable stochastic perturbations of reaction--diffusion equations improve this situation considerably. More precisely, we show delayed blow-up phenomena for reaction-diffusion equations with transport noise and periodic boundary condition:
\begin{equation}
\label{eq:reaction_diffusion_system_intro}
\left\{
\begin{aligned}
\dd v_i -\ellip_i\Delta v_i \,\dd t&=f_{i}(\cdot, v)\,\dd t
+ \sqrt{c_d\ellip} \sum_{k\in \Z^{d}_0} \sum_{1\leq \alpha\leq d-1} \theta_k (\sigma_{k,\alpha}\cdot \nabla) v_i\circ \dd w_t^{k,\alpha}, &\text{ on }&\Tor^d,\\
v_i(0)&=v_{i,0},  &\text{ on }&\Tor^d,
\end{aligned}\right.
\end{equation}
where $i\in \{1,\dots,\ell\}$ for some integer $\ell\geq 1$.
Here we denote by
$v=(v_i)_{i=1}^{\ell}:[0,\infty)\times \O\times \Tor^d\to \R^{\ell}$ the unknown process, $d\geq 2$ the dimension, 
$c_d=\frac{d}{d-1}$, $\Z_0^d =\Z^d\setminus \{0\}$ 
and $\ellip_i$ the diffusivity of $v_i$.
Finally,
$(w^{k,\alpha})_{ k,\alpha}$ 
is a sequence of complex Brownian motions on a filtered probability space and
$$
(\sigma_{k,\alpha}\cdot\nabla) v_i\stackrel{{\rm def}}{=}\sum_{1\leq j\leq d} \sigma_{k,\alpha}^j \partial_j v_i.
$$
The vector fields $\sigma_{k,\alpha}$ are smooth, divergence free and $
\theta=(\theta_k)_{k}\in \ell^2$. A precise description of the noise will be given in Subsection \ref{ss:noise_description}.
The nonlinearities $f_i$ depend on $v=(v_i)_{i=1}^{\ell}$ and are assumed to be of polynomial growth and with mass control. A prototype example is given by \eqref{eq:chemical_reactions} which appears in the study of reversible chemical reactions, see Subsection \ref{ss:chemical_reactions}. The term $\ellip_i \Delta v_i\,\dd t$ in \eqref{eq:reaction_diffusion_system_intro} can be replaced by a general second order operator. For exposition convenience, we do not pursue this here and we only provide some comments in Remark \ref{r:general_second_order_operator}.

In this work we prove that for all $T\in (0,\infty)$, there exists a choice of $(\theta,\ellip)$ such that the strong solutions to \eqref{eq:reaction_diffusion_system_intro} does \emph{not} blow up before time $T$ with high probability. 
Under additional assumptions we are also able to handle the case $T=\infty$.
Since blow--up in finite time occurs for specific instances of \eqref{eq:reaction_diffusion_system_intro} with $\theta\equiv 0$, the presence of the noise is essential. 

Transport noise is often used to study the evolution of passive scalars in turbulent flows, see e.g.\  \cite{F15_book,MK99_simplified}. Such noise is often referred as \emph{Kraichnan model} due to his pioneering works \cite{K68,K94}. Roughly speaking, transport noise can be thought of as an idealization of the effect of ``small scale'' of an underlying turbulent fluid advecting the reaction.
Heuristically, one can assume that the same type of contribution is also present in reaction--diffusion type systems reacting in a turbulent flow, see Subsection \ref{ss:physical_motivations}. 
In this scenario, as experiments with chemical reactions suggest (see e.g.\ \cite{GE63_turbulence,LW76,SH91_turbulence,MC98,KD86_turbulence_chemical_reactions,ZCB20_fluid}), turbulent flows ``effectively'' increase the diffusivity of reactants. This eventually leads to an increased efficiency of the corresponding chemical reaction. 
In practice, the chemical reaction occurs as if the reactants have an increased diffusion compared to the one measured in standard conditions. This phenomenon is usually called \emph{enhanced diffusion}. One of the aim of this paper is to provide a (possible) mathematical description of this fact by showing that, in ``relatively weak'' norms, the solution to \eqref{eq:reaction_diffusion_system_intro} is close to the solution of the corresponding deterministic problem with increased diffusivity (see Theorem \ref{t:delay_enhanced_dissipation_intro} below).
This fact can be thought of as a ``weak'' enhanced diffusion result. 

\subsection{Delayed blow--up and enhanced diffusion: Simplified version}
\label{ss:intro_delay}
To give a flavor of the results in the paper, here we state a simplified version of Theorem \ref{t:delayed_blow_up}. 
To apply it, one fixes three parameters: $T\in (0,\infty)$ the time horizon where one wants the solution to exist, $\varepsilon\in (0,1)$ the size of    
the event where the blow--up may occur and $r\in (1,\infty)$ the time integrability for the norm in which we measure the weak enhanced diffusion.

\begin{theorem}[Simplified version of Theorem \ref{t:delayed_blow_up}]
\label{t:delay_enhanced_dissipation_intro}
Let $\frac{d(h-1)}{2}\vee 2<q<\infty$. Fix $T\in (0,\infty)$, $\varepsilon\in (0,1)$ and $r\in (1,\infty)$.
Assume that $f$ is of polynomial growth with exponent $h>1$ and with mass control (see Assumption \ref{ass:f_polynomial_growth}\eqref{it:f_polynomial_growth_1}--\eqref{it:positivity} below). Let $v_0\in L^q(\Tor^d;\R^{\ell})$ be such that $v_0\geq 0$ (component-wise). Then there exist $\ellip>0$ and 
$
\theta \in \ell^2
$ such that $\#\{k\,:\, \theta_k\neq 0\}<\infty$ for which the unique strong solution $v$ to \eqref{eq:reaction_diffusion_system_intro} exists up to time $T$ with high probability:
$$
\P(\tau\geq T)>1-\varepsilon  \text{ where $\tau$ is the blow--up time of $v$}.
$$
Moreover
\begin{equation}
\label{eq:enhanced_dissipation_intro_statement}
\P\big(\tau\geq T, \, \|v-\vd\|_{L^r(0,T;L^q(\Tor^d;\R^{\ell}))}\leq \varepsilon\big)>1-\varepsilon
\end{equation}
where $\vd=(\vdi)_{i=1}^{\ell}$ is the unique strong solution to the \emph{deterministic} reaction--diffusion equation with increased diffusion on $[0,T]$:
\begin{equation}
\label{eq:enhanced_dissipation_intro_statement_2}
\partial_t\vdi -(\ellip_i +\ellip)\Delta \vdi =f_{i}(\cdot, \vd)\text{ on }\Tor^d, \qquad
\vdi(0)=v_{0,i}\text{ on }\Tor^d.
\end{equation}
\end{theorem}

In the above result one can even choose $(\theta,\ellip)$ uniformly with respect to $v_0$ such that $\|v_0\|_{L^q}\leq N$, where $N\geq 1$ is fixed. 
In such case $(\theta,\ellip)$ does not depend on $v_0$, but only on $N$. 
Actually, one can always enlarge $\ellip$ still keeping Theorem \ref{t:delay_enhanced_dissipation_intro} true. 
Theorem \ref{t:delay_enhanced_dissipation_intro} shows that the solution to \eqref{eq:reaction_diffusion_system_intro} is close to the solution of a deterministic reaction--diffusion equations with \emph{increased diffusivity}.
The existence of a unique strong solution $\vd$ to \eqref{eq:enhanced_dissipation_intro_statement_2} on $[0,T]$ is also part of the proofs. 
The complete result is given in Theorem \ref{t:delayed_blow_up}. 
In Theorem \ref{t:delayed_blow_up_t_infty} we also allow $T=\infty$, in the case of exponentially decreasing mass.

Theorems \ref{t:delay_enhanced_dissipation_intro} and \ref{t:delayed_blow_up} essentially follow from the scaling limit argument of Theorem \ref{t:weak_convergence}. 
Following the heuristic derivation of Subsection \ref{ss:physical_motivations}, where we introduce the transport noise in \eqref{eq:reaction_diffusion_system_intro} as a model for small scales of the driven turbulent dynamic, one may think of Theorem \ref{t:weak_convergence}  as an ``homogenization'' result for the SPDE \eqref{eq:reaction_diffusion_system_intro} where the role of the scale parameter is played by the ratio $\|\theta\|_{\ell^{\infty}}/\|\theta\|_{\ell^2}$ (cf.\ also Subsection \ref{ss:enhanced_dissipation}). 
%
%
Looking at  Theorem \ref{t:weak_convergence} in this perspective, Theorem \ref{t:delay_enhanced_dissipation_intro} (and the main results of the paper) can be seen as a ``large scale regularity'' result (in the homogenization sense, see e.g.\  \cite{S18_book,SKM19_book,GNO20_regularity}) for the SPDEs \eqref{eq:reaction_diffusion_system_intro}. 
Moreover, the ``homogenized'' system \eqref{eq:enhanced_dissipation_intro_statement_2} can be thought of as the ``effective problem'' for \eqref{eq:reaction_diffusion_system_intro} where the additional diffusive contribution $\ellip\Delta $ in \eqref{eq:enhanced_dissipation_intro_statement_2} takes into account the effect of the underlying turbulent flow.
The homogenization view--point is also interesting for mathematical reasons. Indeed, as it is standard in homogenization theory, even in presence of smooth diffusive matrix, one cannot prove estimates uniformly in the scale parameter. In practice, one cannot use further information on the diffusivity matrix besides ellipticity and boundedness. 
Therefore one is forced to use tools from PDEs with (rough) $L^{\infty}$--coefficients, such as Moser iterations and DeGiorgi--Nash--Moser estimates. 
A similar situation appears here, where, to run the scaling limit argument of Theorem \ref{t:weak_convergence}, one need an estimate coming from a Moser type iteration, see Theorem \ref{t:global_cut_off}\eqref{it:global_cut_off_2}. 

In light of the results in \cite{AV22} (recalled here in Theorem \ref{t:local}), the solution $v$ of Theorem \ref{t:delay_enhanced_dissipation_intro} is not only \emph{strong}, but it is also positive and instantaneously gains regularity:
\begin{align}
\label{eq:positivity_intro}
v&\geq 0 \ \text{(component-wise)} & &\text{ a.e.\ on }[0,\tau)\times \O\times \Tor^d,\\
\label{eq:regularity_intro}
v&\in C^{\gamma_1,\gamma_2}_{\loc}((0,\tau)\times \Tor^d;\R^{\ell}) & &\text{ a.s.\ for all $\g_1\in (0,\tfrac{1}{2})$ and $\gamma_2\in (0,1)$.}
\end{align}
The positivity of solutions to \eqref{eq:reaction_diffusion_system_intro} is very important from an application point of view, as $v_i$ typically models concentrations. Let us stress that an additive noise would destroy the positivity of the initial data. Thus, in the context of reaction--diffusion equations, additive noise seems not appropriate to work with. 
Another interesting feature of transport noise is that it does \emph{not} alter mass conservations, energy balance and, more generally, $L^q$--estimates. 
Here we mean that, when computing $\|v_i\|_{L^q}^q$, one obtains an equality in which the noise does not contribute. Moreover, such equality is the one obtained in absence of noise, see Subsection \ref{ss:enhanced_dissipation}. This shows in particular that the stochastic perturbation does \emph{not} help in proving $L^q$--bounds. The diffusive behavior of the noise can only be seen in norms which are ``below'' the $L^q$--energy level (e.g.\ $L^r(0,T;L^q)$ with $r<\infty$), cf.\  \eqref{eq:enhanced_dissipation_intro_statement} in Theorem \ref{t:delay_enhanced_dissipation_intro}.

Compared to standard deterministic theory of reaction--diffusion equations (e.g.\ \cite{P10_survey}), the strength of the results of Theorem \ref{t:delay_enhanced_dissipation_intro} is that the presence of noise allows us to obtain \emph{strong unique solutions} to \eqref{eq:reaction_diffusion_system_intro} with arbitrary large life (at expense of enforcing the noise). Under some additional assumptions (e.g.\ entropy-dissipation relation), in the deterministic setting, existence of global \emph{weak} solutions to \eqref{eq:reaction_diffusion_system_intro} is shown in \cite{F15_global_renormalized,FKKM22,LW22}. Determining whether or not such solutions are unique and/or smooth is an open problem \cite[Section 4]{F15_global_renormalized}. For the weaker notion of weak--strong uniqueness see 
\cite{F17_weak_strong_uniqueness}. 

Theorem \ref{t:delay_enhanced_dissipation_intro} is a \emph{regularization by noise} result since solutions to the deterministic version of \eqref{eq:reaction_diffusion_system_intro} blow--up in finite time for appropriate choices of $f_i$ satisfying Assumption \ref{ass:f_polynomial_growth}\eqref{it:f_polynomial_growth_1}--\eqref{it:positivity} (\cite{PS97_blow_up} or \cite[Theorems 4.1 and 4.2]{P10_survey}). Regularization by noise started with the seminal work of Veretennikov \cite{V81_reg_by_noise}, where he proved that noise restores existence and uniqueness in ODEs.
This basic result has been later extended in many directions and  in particular to PDEs. It is not possible to provide a complete overview on such results and we content ourself to the case of regularization by \emph{transport} noise. In such area, a first breakthrough result has been established by Flandoli, Gubinelli and Priola \cite{FGP10} where they proved that transport noise improves the well--posedness theory for the transport equation (see also \cite{GM18_conservation_laws} for scalar conservation laws). A second breakthrough has been recently obtained by Flandoli and Luo in \cite{FL19} where they prove that a sufficiently intense noise prevents the blow--up of the Navier--Stokes equations in \emph{three dimensions} and in vorticity formulation.  
Related results can be found in \cite{G20_convergence,FGL21,FHLN20,L22_avereged_NS,FGL21_mixing,GL22,L21} and in the references therein. 

The results of this paper fall into this last line of research, providing new results and highlighting new points of view on the works \cite{FL19,FGL21}. One of the main contribution of the current paper is the connection with the theory of critical spaces for SPDEs developed in \cite{AV19_QSEE_1,AV19_QSEE_2} which relies on the $L^p(L^q)$--theory for SPDEs, pioneered by Krylov \cite{Kry94a,Kry} and later by Van Neerven, Veraar and Weis \cite{NVW1,MaximalLpregularity}. 
We will see that the assumption $q>\frac{d(h-1)}{2}$ in Theorem \ref{t:delay_enhanced_dissipation_intro} is related to the criticality of the space $L^{\frac{d}{2}(h-1)}$ for the SPDEs \eqref{eq:reaction_diffusion_system_intro}, see Subsection \ref{ss:role_criticality} for more details. 
To the best of our knowledge, the current paper is the first regularization by noise result exploiting  $L^p(L^q)$--estimates. Let us stress that the $L^p(L^q)$--setting is \emph{necessary} for proving the results of the current paper in order to match the underlined homogenization argument (see the text below Theorem \ref{t:delay_enhanced_dissipation_intro}) and the polynomial growth of $f$.

The use of the $L^p(L^q)$--theory for SPDEs requires an important re-elaboration of the works  \cite{FL19,FGL21}. Indeed, on the one hand the $L^p(L^q)$--setting requires some smoothness of the coefficients (see e.g.\ \cite{AV21_SMR_torus}) and on the other hand the scaling limit results as in \cite{FL19,FGL21} and Theorem \ref{t:weak_convergence} prohibit the use of such smoothness. 
Hence, one of the main difficulties faced up with in this work is the match of the two techniques which will be accomplished via a careful analysis of the nonlinearities.
Besides this fundamental obstruction, several new analytical difficulties arise, for instance related to the regularity of weak solutions provided by the scaling limit argument behind the proof of Theorem \ref{t:delay_enhanced_dissipation_intro}, see Subsection \ref{ss:uniqueness}. 

\subsection{Reversible chemical reactions}
\label{ss:chemical_reactions}
In this subsection we apply Theorem \ref{t:delay_enhanced_dissipation_intro} to a model for reversible chemical reactions.  
For an integer $\ell\geq 1$ and two collections of nonnegative integers $(q_i)_{i=1}^{\ell},(p_i)_{i=1}^{\ell}$ (note that either $q_i=0$ or $p_i=0$ for some $i$ is allowed), consider the chemical reaction:
\begin{equation}
\label{eq:chemical_reaction_chemistry_formula}
q_1 V_1 +\dots + q_{\ell} V_{\ell} 
\xrightleftharpoons[]{}  
p_1 V_1+\dots + p_{\ell} V_{\ell}
\end{equation}
where $(V_i)_{i=1}^{\ell}$ are the reactants. Let $v_i$ be the concentration of the reactant $V_i$ with diffusivity $\ellip_i>0$. Finally let $R_{\pm}>0$ be the reaction rates.
The \emph{law of mass action} postulates that the concentration $v_i$ satisfies the deterministic version of \eqref{eq:reaction_diffusion_system_intro} with 
\begin{equation}
\label{eq:chemical_reactions}
f_i(\cdot,v)=(p_i -q_i)\Big(R_+\prod_{1\leq j\leq  \ell} v_j^{q_j}-R_-\prod_{1\leq j\leq \ell} v_j^{p_j} \Big),
\end{equation}
where $  i\in \{1,\dots,\ell\}$ 
and $v=(v_i)_{i=1}^{\ell}$.
The conservation of the reactants mass is equivalent to ask for a collection of (strictly) positive constants $(\alpha_i)_{i=1}^{\ell}$ such that $\sum_{1\leq i\leq 1} \alpha_i(q_i-p_i)=0$ (below referred as mass conservation condition).  
The following result is a special case of Theorem \ref{t:delay_enhanced_dissipation_intro}. 

\begin{theorem}
\label{t:chemical_reaction_intro}
Fix $T\in (0,\infty)$, $\varepsilon\in (0,1)$ and $r\in (1,\infty)$. Let $f_i$ be as in \eqref{eq:chemical_reactions}. Assume that the mass conservation condition holds. 
Let $h\in (1,\infty)$ be such that 
$$
h\geq \big(\sum_{1\leq i\leq\ell}q_i\big)\vee \big(\sum_{1\leq i\leq \ell }p_i\big)\quad \text{ and fix }\quad q >\frac{d}{2}(h-1) \vee 2. 
$$
Let $v_0\in L^q(\Tor^d;\R^{\ell})$ be such that $v_0\geq 0$ (component--wise). Then there exist $\ellip>0$ and $\theta\in \ell^2$ such that $\#\{k\,:\, \theta_k\neq 0\}<\infty$ for which the unique strong solution $v$ to \eqref{eq:reaction_diffusion_system_intro} satisfies  
$$
\P(\tau\geq T)>1-\varepsilon \text{ where $\tau$ is the blow--up time of $v$}.
$$
Moreover the following hold:
\begin{enumerate}[{\rm(1)}]
\item {\rm (Weak enhanced diffusion)} \eqref{eq:enhanced_dissipation_intro_statement} holds with $f_i$ as in \eqref{eq:chemical_reactions} and $r$ as above.
\item\label{it:positivity_intro_chemical} {\rm (Positivity)} $v\geq 0$ component--wise a.e.\ on $[0,\tau)\times \O\times \Tor^d$.
\item\label{it:regularity_intro_chemical}  {\rm (Instantaneous regularization)}
$v\in C^{\g,\infty}_{\loc} ((0,\tau)\times \Tor^d;\R^{\ell})$ a.s.\ for all $\gamma\in [0,\tfrac{1}{2})$.
\end{enumerate}
\end{theorem}

Item \eqref{it:positivity_intro_chemical} follows from \eqref{eq:positivity_intro}. Item \eqref{it:regularity_intro_chemical} is stronger than \eqref{eq:regularity_intro} and still follows from the result of \cite{AV22} where one also uses the fact that $f_i$ are smooth (see Remark \ref{r:regularity_paths}\eqref{it:regularity_paths_1}). Interestingly, item \eqref{it:regularity_intro_chemical} shows that $v$ is not only a strong solution to \eqref{eq:reaction_diffusion_system_intro} but it is also \emph{classical} in space.

As before, we remark that the transport noise does \emph{not} interact with the mass, energy and $L^q$--balances. For instance, under the mass conservation condition, by integrating \eqref{eq:reaction_diffusion_system_intro} with \eqref{eq:chemical_reactions}, one can show the \emph{pathwise} conservation of mass:  
\begin{equation}
\sum_{1\leq i\leq \ell} \alpha_i\int_{\Tor^d} v_i(t,x)\,\dd x=\sum_{1\leq i\leq \ell}  \alpha_i\int_{\Tor^d} v_{0,i}(x) \,\dd x\ \ \ \text{ a.s.\ for all $t\in [0,\tau).$}
\end{equation}
In absence of noise, existence for large times $T\gg 1$ of unique \emph{strong} solutions to \eqref{eq:reaction_diffusion_system_intro} with \eqref{eq:chemical_reactions} is generally \emph{not} known even for the (apparently) simple situation of \eqref{eq:chemical_reaction_chemistry_formula} with  $\ell=2$ (cf.\ \cite[Remark 3.2]{P10_survey}). Let us mention that already the case $q_1=p_2=1$ and $q_2=p_1=2$ appears problematic.  
Indeed, in the deterministic setting, existence of global unique smooth solutions is only known in case $(\sum_{1\leq i\leq \ell}q_i)\vee  (\sum_{1\leq i\leq \ell} p_i)\leq 2$. The reader is referred to \cite{FMT20}  for the general situation, and to \cite{CGV19_global} for the four species case, i.e.\ 
$V_1+V_2
\xrightleftharpoons[]{}  
V_3 +V_{4}$. Existence of global unique smooth solutions for reversible chemical-reactions \eqref{eq:chemical_reaction_chemistry_formula} in case $(\sum_{1\leq i\leq \ell}q_i)\vee  (\sum_{1\leq i\leq \ell} p_i)\geq 3$ is still open. 
%
%
In particular, if $(\sum_{1\leq i\leq \ell}q_i)\vee  (\sum_{1\leq i\leq \ell} p_i)\geq 3$, then Theorem \ref{t:chemical_reaction_intro} has no deterministic counterpart. Finally, let us note that, if $h\geq 3$ and $d\geq 2$, then $q>\frac{d(h-1)}{2}\geq 2$. In particular, $L^q$-theory with $q>2$ is necessary to apply Theorem \ref{t:chemical_reaction_intro} in the relevant situations in which $(\sum_{1\leq i\leq \ell}q_i)\vee  (\sum_{1\leq i\leq \ell} p_i)\geq 3$.

It seems that the case $T=\infty$ of Theorem \ref{t:chemical_reaction_intro} does not hold in general and one needs additional assumption on the reaction \eqref{eq:chemical_reaction_chemistry_formula}. This fact goes beyond the scope of the current paper. 

\subsection{Further comments on the literature}
\label{ss:comments_literature}
We collect here further references to the related literature. To the best of our knowledge, 
in the deterministic case, the investigation of the effect of a velocity field on the dynamics of passive scalars was first studied by Constantin, Kiselev, Ryzhik and Zlato\v{s} \cite{CKRZ08_mixing}. For some results in a nonlinear deterministic $L^2$--setting see \cite{IXZ21_TAMS,KX_16_keller}. 
One interesting feature of the stochastic setting is that, in contrast to deterministic results, (stochastic) delayed blow--up type results are always accompanied with a homogenization one which describes the ``effective'' contribution of the driven turbulent dynamics on the system and this effective representation is consistent with physical experiments \cite{GE63_turbulence,LW76,SH91_turbulence,MC98,KD86_turbulence_chemical_reactions,ZCB20_fluid}. 

The case of a linear dynamic in a turbulent fluid, modeled by a transport noise, was also studied by Gess and Yaroslavtsev in \cite{GY21}. There the authors proved stabilization and enhanced dissipation by noise for passive scalars, and they also provide a detailed overview of previous results. A related interesting situation has been studied by Bedrossian, Blumenthal and Punshon--Smith in \cite{BBPS21_PTRF,BBPS22,BBP22_AP}, where they investigate the dynamics of deterministic passive scalars driven by a flow solving a \emph{stochastic} Navier--Stokes type system with additive noise. 
%

\subsection{Notation}
Here we collect the notation which will be used throughout the paper. We write $A\lesssim_{P_1,\dots,P_N} B$ (resp.\ $A\gtrsim_{P_1,\dots,P_N} B$) whenever there exists a positive constant $C$ depending only on the parameters $P_1,\dots,P_N$ such that $A\leq C B$ (resp.\ $A\geq CB$). Furthermore we write $A \eqsim B$ if $A \lesssim B$ and $A \gtrsim B$.     
Similarly, we write $C_{P_1,\dots,P_N}$ or $C(P_1,\dots,P_N)$ if the constant $C$ depends only on $P_1,\dots,P_N$. Moreover, $\R$ is the set of real numbers, $\R_+=(0,\infty)$, $\Z$ is the set of integers, $\Z^d_0=\Z^d\setminus \{0\}$. We also employ the notation
$a \vee b=\max\{a,b\}$ and $a\wedge b=\min\{a,b\}$.
In the following, for an integer $\ell\geq 1$, $s\in\R$ and $q\in (1,\infty)$, we denote by $H^{s,q}(\Tor^d;\R^{\ell})$, $B^{s}_{q,p}(\Tor^d;\R^{\ell})$ the set of $\R^\ell$--valued maps in the Bessel potential and in the Besov classes, respectively (see e.g.\ \cite{Moder_Fourier_Analysis,Tri95,BevosBook}). Often, below we write
$L^q,H^{s,q}$ etc.\ instead of $ L^{q}(\Tor^d;\R^{\ell}),H^{s,q}(\Tor^d;\R^{\ell})$ etc., if no confusion seems likely. For $p\in (1,\infty)$, we denote by $(\cdot,\cdot)_{\theta,p}$ and $[\cdot,\cdot]_{\theta}$ the real and the complex interpolation functor, respectively. The reader is referred to 
\cite{BeLo,Analysis1,Tri95} 
for definitions and basic properties.
Below we collect some further notation which may be non standard. In the following $X$ is a Banach space, $p\in (1,\infty)$ and $I=(a,b)\subseteq \R$ is an open interval.
\begin{itemize}
\item $w_{\a}(t)=|t|^{\a}$ for $t\in \R_+$ and $\a\in \R$ (power weight).
\item $L^p(a,b,w_{\a};X)$ is the set of all strongly measurable maps $f:I\to X$ satisfying 
$$
\|f\|_{L^p(a,b,w_{\a};X)}=\Big(\int_a^b \|f(t)\|^p_X\,w_{\a}(t)\, \dd t\Big)^{1/p}. 
$$
If $\a=0$, then we write $L^p(a,b;X)$ instead of $L^p(a,b,w_{0};X)$.
\item $W^{1,p}(a,b,w_{\a};X)$ or $W^{1,p}(I,w_{\a};X)$  denotes the space of all $f\in 
 L^p(a,b,w_{\a};X)$ such that $f'\in L^p(a,b,w_{\a};X)$ endowed with the natural norm.
\item $H^{\vartheta,p}(a,b,w_{\a};X)=[L^{p}(a,b,w_{\a};X),W^{1,p}(a,b,w_{\a};X)]_{\vartheta}$ for $\vartheta\in (0,1)$.
\item For $\g_1,\g_2>0$, $C^{\g_1,\g_2}((a,b)\times \Tor^d)$ denotes the set of all bounded maps $u$ such that
$$
|u(t,x)-u(s,y)|\lesssim |t-s|^{\g_1} + |x-y|^{\g_2} \ \ \text{ for all } s,t\in (a,b), \ x,y\in \Tor^d.
$$ 
\item For a function space $\A$, we sometimes write $\A(I,w_{\a};X)$ instead of $\A(a,b,w_{\a};X)$.  
Moreover, we write $f\in \A_{\loc}(\Dom,w_{\a};X)$ provided $f\in \A(\Dom',w_{\a};X)$ for all compact set $\Dom'\subseteq \Dom$. 

\end{itemize}
Finally we collect the probabilistic notation. Further notation will be fixed in Subsection \ref{ss:noise_description}.
Throughout the paper, $(\O,\A,(\F_t)_{t\geq 0},\P)$ denotes a filtered probability space. A measurable map $\tau:\O\to [0,\infty]$ is a stopping time if $\{\tau\leq t\}\in\F_t$ for all $t\geq 0$. For a stopping time $\tau$, $\F_{\tau}$ denotes the $\sigma$-algebra of the $\tau$-past, i.e.\ $A\in \F_{\tau}$ provided
$
A\cap \{\tau\leq t\}\in \F_t 
$
for all $t\geq 0$  (see e.g.\ \cite{Kal}). For a Banach space $X$, a stochastic process $\phi:[0,\infty)\times \O\to X$ 
is said to be progressive measurable if $\phi|_{[0,t]\times \O}$ is strongly $\Borel([0,t])\times \F_t$--measurable for all $t\geq 0$.

Finally, we write $\sum_{k,\alpha}$ instead of $\sum_{k\in \Z^d_0}\sum_{1\leq \alpha\leq d-1}$, if no confusion seems likely.

\section{Derivation, enhanced diffusion and criticality}
In this section we illustrate some basic ideas leading to the proof of our main results. However, before going into the mathematical details, we first provide an heuristic derivation of \eqref{eq:reaction_diffusion_system_intro}.

\subsection{Heuristic derivation}
\label{ss:physical_motivations}
Inspired by  \cite[Subsection 1.2]{FL19}, we motivate transport noise by the idea of separating large and small scales and to model the small scale by noise. This corresponds to some intuition of turbulence. With this in mind, we formally derive \eqref{eq:reaction_diffusion_system_intro} by considering its deterministic version in which $ v_i $ is transported by a velocity field $ u $ of a fluid in which $ v_i $ lies:
\begin{equation}
\label{eq:deterministic_intro}
\partial_t v_i + (u\cdot\nabla )v_i = \ellip_i \Delta v_i +f_i(\cdot,v),\quad \text{ on }\Tor^d.
\end{equation}
Following \cite{FL19}, we decompose $u$ as $u_{L}+ u_{S}$, where $u_{L}$ and $u_{S}$ denote the large and the small scale part, respectively. Roughly speaking, in a turbulent regime, the $u_S$ varies very rapidly in time compared to $u_{L}$. In this case, one may replace $u_S$ by an approximation of white noise, i.e.\ $-\sum_{k,\alpha}\theta_k (\sigma_{k,\alpha}\cdot\nabla) v_i\,\circ \dot{w}^{k,\alpha}_t$, and therefore \eqref{eq:deterministic_intro} coincide \eqref{eq:reaction_diffusion_system_intro} with an additional deterministic transport noise. The deterministic transport term $(u_{L}\cdot\nabla) v_i$ does not play any role in the analysis, and therefore we drop it from the results below (see Remark \ref{r:general_second_order_operator} for some comments).
For various fluid dynamics models, the approximation of small scales by a transport term can be made rigorous, see \cite{FP21,DP22_two_scale}. 
Let us also remark that the noise \eqref{eq:reaction_diffusion_system_intro} is in the Stratonovich formulation, which, from a modeling point of view, seems the correct one due to its connections with Wong-Zakai type results. 
Moreover, as we will see in Subsection \ref{ss:enhanced_dissipation} below, the Stratonovich noise does not alter the mass and energy balances. This is consistent with the intuition of the stochastic perturbation in \eqref{eq:reaction_diffusion_system_intro} as a transport term.

\subsection{Enhanced diffusion and the homogenization view-point}
\label{ss:enhanced_dissipation}
The issue of global well-posedness of parabolic PDEs is usually addressed by showing energy estimates. In practice, one derives uniform in time bounds on suitable $L^q_x$--norms of the solutions to the corresponding PDEs. Blow--up criteria for SPDEs (c.f.\ Theorem \ref{t:local}\eqref{it:local_3}) shows that a pathwise $L^{\infty}_t(L^q_x)$--estimate with $q>\frac{d(h-1)}{2}\vee 2$ is sufficient to prove global existence for system of reaction diffusion equations like \eqref{eq:reaction_diffusion_system_intro}. Thus, one is tempted to apply the It\^{o} formula to compute $\|v_i\|_{L^q}^q$ and to derive such bounds. However, due to the divergence free of $\sigma_{k,\alpha}$, one has 
$$
 \int_{\Tor^d} |v_i|^{q-2}[(\sigma_{k,\alpha}\cdot \nabla ) v_i] v_i \,\dd x=0 \quad \text{ for all }q\in [2,\infty).
$$
In particular, the martingale part in the It\^o formula vanishes and one obtains, a.s.\ for all $t\in[0,\tau)$,
\begin{equation}
\begin{aligned}
\label{eq:energy_Lq_balance}
\|v_i(t)\|_{L^q}^q 
&+ \ellip_iq(q-1)\int_{0}^t \int_{\Tor^d}|v_i|^{q-2}|\nabla v_i|^2\,\dd x\, \dd s \\
&
=\|v_{0,i}\|_{L^q}^q
+ q \int_{0}^t\int_{\Tor^d} |v_i|^{q-2} f_i(\cdot,v) v_i\,\dd x\, \dd s.
\end{aligned}
\end{equation}
The above equality coincides with the $L^q$--balance in absence of stochastic perturbation in \eqref{eq:reaction_diffusion_system_intro}. Therefore it is clear that the noise \emph{cannot} help to improve such estimates. To capture the weak enhanced diffusion induced by the transport noise one has to look at \emph{weaker} norms compared to the one appearing in the energy--type balance \eqref{eq:energy_Lq_balance}, e.g.\ $L^r(0,T,L^q)$ with $r\in (1,\infty)$. 

From a mathematical perspective, the key step to understand the weak enhanced diffusive effect of the noise is the scaling limit result of Theorem \ref{t:weak_convergence}. In that result, we consider a sequence of $(\theta^{(n)})_{n\geq 1}$ and the sequence of corresponding solutions $(v^{(n)})_{n\geq 1}$ to \eqref{eq:reaction_diffusion_system_intro} and we show convergence (in ``relatively weak norms'') of the solutions to a deterministic system of reaction--diffusion equations with increased diffusivity provided $\lim_{n\to \infty} \|\theta^{(n)}\|_{\ell^{\infty}}/\|\theta^{(n)}\|_{\ell^2}=0$. 
Here we exploit the fact that the vector fields $\sigma_{k,\alpha}$ are objects with high oscillations and in the limit as $n\to \infty$ they average. The limiting contribution of the noise is the diffusive term $\ellip\Delta v_i$ in \eqref{eq:enhanced_dissipation_intro_statement_2}. 
Now Theorem \ref{t:delay_enhanced_dissipation_intro} follows by choosing $n$ so large that the solution to \eqref{eq:reaction_diffusion_system_intro} is not far from \eqref{eq:enhanced_dissipation_intro_statement_2}.
Since the transport noise in \eqref{eq:reaction_diffusion_system_intro} models small scale effects, the above argument shares the same philosophy of \emph{large scale regularity} theory in the theory of homogenization, see e.g.\ \cite{S18_book,SKM19_book,GNO20_regularity}.
As commented in the Susbection \ref{ss:intro_delay}, this interpretation naturally brings us to the use of tools from the theory of PDEs with $L^{\infty}$--coefficients, such as Moser iterations.

In a way, this view--point allows us to give an heuristic motivation for the failure of the scaling limit argument in \cite{FL19} for the \emph{full} advective noise (see \cite[Appendix 2]{FL19}). 
Recall that the vorticity formulation in \cite{FL19} is obtained by applying $\nabla \times$ to the Navier--Stokes equations with transport noise. Due to Leibniz rule, this creates a (lower order) term which cannot be controlled via the $L^{\infty}$--norm of the coefficients itself and therefore the scaling argument is doomed to fail.

\subsection{The role of criticality}
\label{ss:role_criticality}
Several choices of the spaces done in this paper are motivated by the (local) invariance of the SPDEs \eqref{eq:reaction_diffusion_system_intro} under parabolic scaling. Recall that $f$ is of polynomial growth with exponent $h>1$ (see Assumption \ref{ass:f_polynomial_growth}\eqref{it:f_polynomial_growth_1}). As discussed in \cite[Subsection 1.4]{AV22}, the Lebesgue space $L^{\frac{d}{2}(h-1)}$ is critical for \eqref{eq:reaction_diffusion_system_intro}.
Here we do not discuss the case of critical Besov spaces, as the Lebesgue ones are the natural to deal with when working with $L^{\infty}$--coefficients. 
With an eye towards the main scaling argument of Theorem \ref{t:weak_convergence}, where one needs to use compactness, we work within the subcritical regime $L^q$ with $q>\frac{d(h-1)}{2}$. Indeed, within this range, one lose regularity to obtain compactness, still being in a spaces where \eqref{eq:reaction_diffusion_system_intro} is well--posed. The subcriticality also plays an important role in the main estimates. Indeed, a fairly straightforward consequence of it is the existence of $\varepsilon>0$ such that (cf.\ Lemma \ref{lem:interpolation_strong_setting})
\begin{equation}
\label{eq:interpolation_inequality_estimate}
\|f_i(\cdot,v)\|_{H^{-1,q}} \lesssim 1+ \|v\|_{L^q}^{h-1+\varepsilon} \|v\|_{H^{1,q}}^{1-\varepsilon}.
\end{equation}
The criticality of the $L^{q}$ is equivalent to ask for which $q$ the inequality \eqref{eq:interpolation_inequality_estimate} holds with $\varepsilon=0$. The sub-criticality gives us the play parameter $\varepsilon>0$ which can be used to show 
$L^{\infty}_t(L^q_x)$--estimates via a simple buckling argument. Indeed, the Young inequality shows that, for all $p\in (2,\infty)$,  
\begin{equation}
\label{eq:inequality_interpolation_varepsilon}
\begin{aligned}
\|f_i(\cdot,v)\|_{L^p(0,T;H^{-1,q})}
&\lesssim 1+ \| v\|_{L^{r}(0,T;L^q)}^{h-1+\varepsilon}\|v\|_{L^p (0,T;H^{1,q})}^{1-\varepsilon}\\
&\lesssim 1+C_{\delta} \| v\|_{L^{r}(0,T;L^q)}^{(h-1+\varepsilon)/\varepsilon} + \delta \|v\|_{L^p(0,T;H^{1,q})}
\end{aligned}
\end{equation}
where $r(h,\varepsilon,q)\in (1,\infty)$ is large.
Choosing $\delta>0$ small enough, one can use maximal $L^p$-regularity estimates to close a bound for $\|v\|_{L^p(0,T;H^{1,q})}$ in terms of $ \| v\|_{L^{r}(0,T;L^q)}$. 
However, there is no general way estimate the latter term. Following \cite{FL19,FGL21}, we introduce a cut-off in the equation \eqref{eq:reaction_diffusion_system_intro}. We design the cut-off $\phi_{R,r}(\cdot,v)$ in a way that 
$
\|\phi_{R,r}(\cdot,v) v\|_{L^r(0,T;L^q)}\lesssim_R 1,
$
see \eqref{eq:def_cut_off} below.
Thus, for the cut--off version of \eqref{eq:reaction_diffusion_system_intro}, the inequality \eqref{eq:inequality_interpolation_varepsilon} readily proves an estimate, cf.\ Theorem \ref{t:global_cut_off}\eqref{it:global_cut_off_1}. The cut-off can later be removed by using the (weak) enhanced diffusive effect of the noise.
 
The same sort of argument also enters in the Moser type iteration used in  
Theorem \ref{t:global_cut_off}\eqref{it:global_cut_off_2}.
More precisely, looking at the $L^q$--balance of \eqref{eq:energy_Lq_balance}, the condition $q>\frac{d(h-1)}{2}$ yields the existence of $\beta\in (0,1)$ such that the RHS\eqref{eq:energy_Lq_balance} can be estimated as (cf.\ Lemma \ref{l:interpolation_L_eta})
%
\begin{align*}
\Big|\int_0^T \int_{\Tor^d} |v|^{q-2} f_i(\cdot,v) v_i \,\dd x\, \dd s\Big|
&\lesssim \|v\|_{L^{r}(0,T;L^{q})}^{q+h-1}\\
&+
\|v\|_{L^{r}(0,T;L^{q})}^{\alpha}
\Big(\max_{1\leq i\leq \ell}\int_0^T \int_{\Tor^d} |v_i|^{q-2}|\nabla v_i|^2\,\dd x\, \dd s\Big)^{\beta}.
\end{align*}
Again, by balancing the contribution of $\|v\|_{L^r(0,T;L^q)}$ with the cut--off $\phi_{R,r}$, one sees that the energy term $\max_{1\leq i\leq \ell}\int_0^T \int_{\Tor^d} |v_i|^{q-2}|\nabla v_i|^2\,\dd x\,\dd s$ can be absorbed on the LHS\eqref{eq:energy_Lq_balance} with the same buckling argument via Young inequality.


\section{Statement of the main results}
In this section we state our main result concerning reaction diffusion equations \eqref{eq:reaction_diffusion_system_intro}. Here we actually consider the following generalization of \eqref{eq:reaction_diffusion_system_intro} where we also include a conservative term:
\begin{equation}
\label{eq:reaction_diffusion}
\left\{
\begin{aligned}
\dd v_i -\ellip_i\Delta v_i \,\dd t
&= \Big[\div(F_i(\cdot,v)) +f_{i}(\cdot, v)\Big]\,\dd t \\
&+ \sqrt{c_d \ellip}\sum_{k\in \Z^{d}_0} \sum_{1\leq \alpha\leq d-1}\theta_k (\sigma_{k,\alpha}\cdot \nabla) v_i\circ \dd w_t^{k,\alpha},   & \text{ on }&\Tor^d,\\
v_i(0)&=v_{i,0},  &  \text{ on }&\Tor^d.
\end{aligned}\right.
\end{equation}
As above, $i\in \{1,\dots,\ell\}$ for some integer $\ell\geq 1$. As before $c_d\stackrel{{\rm def}}{=}\frac{d}{d-1}$ and $\ellip,\ellip_i>0$. The unexplained parameters appearing in the stochastic perturbation of \eqref{eq:reaction_diffusion} will be described in Subsection \ref{ss:noise_description}. The nonlinearities $(f,F)$ will be assumed to be of polynomial growth, see Assumption \ref{ass:f_polynomial_growth} for the precise conditions.
This section is organized as follows. In Subsection \ref{ss:noise_description} we describe the noise and its basic properties, in Subsection \ref{ss:assumptions_definitions} we collect the main assumptions, definition and a local existence result taken from \cite{AV22}. Finally in Subsection \ref{ss:main_results} we state our main results whose proofs will be commented in Subsection \ref{ss:strategy_intro}. 

\subsection{Description of the noise}
\label{ss:noise_description}
Here we specify the quantities $(\theta_k,\sigma_{k,\alpha},w^{k,\alpha})$ appearing in the stochastic perturbation in \eqref{eq:reaction_diffusion}. Here we follow \cite{FL19,FGL21}. 
Recall that $\Z_0^d=\Z^d\setminus\{0\}$. Throughout this paper we consider $\theta=(\theta_k)_{k\in \Z^d_0}\in \ell^2(\Z^d_0)$. Moreover, we assume that $\theta$ is normalized and it is radially symmetric, i.e.\
\begin{equation}
\label{eq:theta_normalized_symmetric}
\|\theta\|_{\ell^2(\Z^d_0)}=1 \quad \text{ and } \quad \theta_{j}=\theta_k \  \text{ for all $j,k\in\Z_0^d$ \  such that }|j|=|k|.
\end{equation}
Finally, we assume that $\#\{k\,:\, \theta_k\neq 0\}<\infty$. However this can be weakened, see Remark \ref{r:choice_theta}.

Next we define the family of vector fields $(\sigma_{k,\alpha})_{k,\alpha}$. Let $\Z_{+}^d$ and $\Z_-^d$ be a  partition of $\Z_0^d$ such that $-\Z_+^d=\Z_-^d$. For any $k\in \Z_+^d$, select an complete orthonormal basis  $\{a_{k,\alpha}\}_{\alpha\in \{1,\dots,d-1\}}$ of the hyperplane $k^{\bot}=\{x\in \R^d\,:\, k\cdot x=0\}$, 
and set $a_{k,\alpha}\stackrel{{\rm def}}{=}a_{-k,\alpha} $ for $k\in \Z^d_-$. Then, let
\begin{equation*}
\sigma_{k,\alpha}\stackrel{{\rm def}}{=} a_{k,\alpha} e^{2\pi i k\cdot x}  \ \  \text{ for all } \ \ x\in \Tor^d,\ k\in \Z^d_0,\ \alpha\in \{1,\dots,d-1\}.
\end{equation*}
By construction we have that $\sigma_{k,\alpha}$ are smooth and divergence free vector fields.

Finally, $(w^{k,\alpha}_t\,:\,t\geq 0)_{k,\alpha}$ denotes a family of complex Brownian motions on a filtered probability space $(\O,\A,(\F_t)_{t\geq 0},\P)$ such that (below $\overline{\cdot}$ denotes the complex conjugate)
\begin{equation}
\label{eq:complex_conjucation_preserves_w}
\overline{w^{k,\alpha}_t}= w^{-k,\alpha}_t \ \  \text{ for all } \ \ t\geq 0,\ k\in \Z^d_0,\ \alpha\in \{1,\dots,d-1\}.
\end{equation}
Moreover $w^{k,\alpha}$ and $w^{j,\beta}$ are independent whenever either $k\neq -j $ or $\alpha\neq \beta$.
The above conditions can be summarized as:
$$
[w^{k,\alpha}, w^{j,\beta}]_t = 2t \delta_{k,-j} \delta_{\alpha,\beta} \ \ \text{ for all $t\geq 0$, }k,j\in \Z^d_0, \text{ and } \alpha,\beta\in \{1,\dots,d-1\},
$$
where $[\cdot,\cdot]_t$ denotes the covariation. As in \cite[Section 2.3]{FL19} or \cite[Remark 1.1]{FGL21}, by \eqref{eq:theta_normalized_symmetric} and the definition of the vector fields $\sigma_{k,\alpha}$, at least formally, one has
\begin{equation}
\begin{aligned}
\label{eq:Ito_stratonovich_change}
\sqrt{c_d \ellip}\sum_{k,\alpha}\theta_k (\sigma_{k,\alpha}\cdot \nabla) v_i\circ \dd w_t^{k,\alpha}
&=
\ellip \Delta v_i+  \sqrt{c_d \ellip}\sum_{k,\alpha}\theta_k (\sigma_{k,\alpha}\cdot \nabla) v_i\,  \dd w_t^{k,\alpha}.
\end{aligned}
\end{equation}
To prove \eqref{eq:Ito_stratonovich_change}, one uses that $\div\,\sigma_{k,\alpha}=0$ and the elementary identity (cf.\ \cite[eq.\ (2.3)]{FL19} or \cite[eq.\ (3.2)]{FGL21})
\begin{equation}
\label{eq:ellipticity_noise}
\sum_{k,\alpha} \theta_k^2 \sigma_{k,\alpha}^{n} \overline{\sigma_{k,\alpha}^{m}}
=
\sum_{k,\alpha} \theta_k^2 a_{k,\alpha}^{n} a_{k,\alpha}^{m}
 =\frac{1}{c_d}\delta_{n,m}\ \   \text{ on }\Tor^d  \text{ for all $1\leq n,m\leq d$}.
\end{equation}
Let us remark that the stochastic integration on the RHS\eqref{eq:Ito_stratonovich_change} is understood in the It\^{o}--sense.
In the paper we will always understood the Stratonovich noise on the LHS\eqref{eq:Ito_stratonovich_change} as the RHS\eqref{eq:Ito_stratonovich_change}, namely an It\^o noise plus a diffusion term. 
However, note that the diffusion term $\ellip\Delta v_i$ does not provide any additional diffusion, as in the usual energy estimates, it is balanced by the It\^{o} correction coming from the  It\^{o}--noise. In particular \eqref{eq:Ito_stratonovich_change} is consistent with Subsection \ref{ss:enhanced_dissipation}.

\subsection{Main assumptions, definitions and local existence}
\label{ss:assumptions_definitions}
In this subsection we collect our main definitions and assumptions. 
The following will be in force throughout this paper.

\begin{assumption} 
\label{ass:f_polynomial_growth}
Suppose that $d\geq 2$ and $\min_{1\leq i\leq \ell}\ellip_i>0$. 
Let the following be satisfied:
\begin{enumerate}[{\rm(1)}]
\item\label{it:integrability_exponent_main_assumption} \emph{(Smoothness and integrability exponents) } $\delta\in [1,2)$, $q,p\in (2,\infty)$ and $\a\in [0,\frac{p}{2}-1)$.
\item\label{it:f_polynomial_growth_1} \emph{ (Polynomial growth) } 
For all $i\in \{1,\dots,\ell\}$, the following mappings are Borel measurable:
\begin{align*}
f_i:\R_+\times \Tor^d\times \R^{\ell}\to \R \ \ \text{ and } \ \
F_i=(F_{i,j})_{j=1}^d:\R_+\times \Tor^d\times \R^{\ell}\to \R^{d}.
\end{align*}
Moreover, there exists $h>1$ such that, a.e.\ on $\R_+\times \Tor^d$ and for all $i\in \{1,\dots,\ell\}$, $y,y'\in \R^{\ell}$,
\begin{align*}
|f_i(t,x,y)|\lesssim 1+ |y|^{h},\ \ \quad \qquad |F_i(t,x,y)|\lesssim 1+ |y|^{\frac{h+1}{2}},&\\
|f_i(t,x,y)-f_i(t,x,y')|\lesssim (1+|y|^{h-1}+|y'|^{h-1})|y-y'|,&\\
|F_i(t,x,y)-F_i(t,x,y')|\lesssim (1+|y|^{\frac{h-1}{2}}+|y'|^{\frac{h-1}{2}})|y-y'|.&
\end{align*}
\item\label{it:positivity} {\rm (Positivity)} For all $i\in \{1,\dots,\ell\}$, there exist a measurable function  $c_i:\R_+\to \R^{d}$ such that, a.e.\ on $\R_+\times \Tor^d$ and for all $i\in \{1,\dots,\ell\}$, $y=(y_i)_{i=1}^{\ell}\in [0,\infty)^{\ell}$,
\begin{align*}
f_i(\cdot,y_1,\dots,y_{i-1},0,y_{i+1},\dots,y_{\ell})&\geq 0,\\
F_i(\cdot,y_1,\dots,y_{i-1},0,y_{i+1},\dots,y_{\ell})&=c_{i}(\cdot).
\end{align*}
\item\label{it:mass} {\rm (Mass control)}
There exist $(\alpha_i)_{i=1}^{\ell}\subseteq (0,\infty)$ and $(\m_j)_{j=1}^2\in  \R$ such that, a.s.\ for all $t\in \R_+$, $x\in \R^d$ and  $y=(y_i)_{i=1}^{\ell}\in [0,\infty)^{\ell}$,
$$
\sum_{1\leq i\leq \ell}\alpha_i  f_i (t,x,y)\leq \m_0 + \m_1 \sum_{1\leq i \leq \ell}\alpha_i  y_i.
$$
\end{enumerate}
\end{assumption}

As we will see in Definition \ref{def:solution} and Theorem \ref{t:local} below, the parameter $\delta$ in \eqref{it:integrability_exponent_main_assumption} rule the Sobolev smoothness of $v$ with integrability $q$, while $p$ its time integrability with weight $\a$.
Conditions \eqref{it:f_polynomial_growth_1}-\eqref{it:mass} in  
Assumption \ref{ass:f_polynomial_growth}  are typically employed in the study of reaction-diffusion equations, see e.g.\ \cite{P10_survey} and the references therein. The growth of the nonlinearities $(F,f)$ in 
\eqref{it:f_polynomial_growth_1} is chosen so that the mapping $v\mapsto f(\cdot,v)$ and $v\mapsto \div(F(\cdot,v))$ has the same (local) scaling (see \cite[Subsection 1.4]{AV22}).  
As shown in \cite{AV22}, the above conditions ensure the existence of solution to \eqref{eq:reaction_diffusion}, with certain properties, under mildly regularity assumption on $v_0$. 
For the reader's convenience, we summarize the one needed in this paper in Theorem \ref{t:local} below.

To introduce the definition of solutions we use the interpretation \eqref{eq:Ito_stratonovich_change} of the Stratonovich noise.
Recall that the family $(w^{k,\alpha})_{k,\alpha}$ induces an $\ell^2$-cylindrical Brownian motion $\Br_{\ell^2}$ given by 
$$
\Br_{\ell^2}(g)\stackrel{{\rm def}}{=}\sum_{k,\alpha} \int_{\R_+} g_{k,\alpha}(t)\,\dd w^{k,\alpha}_t\ \  \ \text{ for } \ \ \ 
 g=(g_{k,\alpha})_{k,\alpha}\in L^2(\R_+;\ell^2),$$ 
where $k\in \Z^d_0$ and $\alpha\in \{1,\dots,d-1\}$.
Note that $\Br_{\ell^2}$ is real valued due to \eqref{eq:complex_conjucation_preserves_w} in case $g_{k,\alpha}\equiv g_{-k,\alpha}$. 

\begin{definition}
\label{def:solution}
Assume that Assumption \ref{ass:f_polynomial_growth} holds for some $h>1$. Suppose that $\theta$ satisfies \eqref{eq:theta_normalized_symmetric}. Let $\tau$ be a stopping time with values in $[0,\infty]$. Finally, let 
$$v=(v_i)_{i=1}^{\ell}:[0,\tau)\times \O\to H^{2-\s,q}(\Tor^d;\R^\ell) \text{ be a stochastic process.} $$
\begin{itemize}
\item 
We say that $(v,\tau)$ is a \emph{local $(p,\a,\s,q)$-solution} to \eqref{eq:reaction_diffusion} if there exists a sequence of stopping times $(\tau_j)_{j\geq 1}$ for which the following hold for all $i\in \{1,\dots,\ell\}$.
\begin{itemize}
\item $\tau_j\leq \tau$ a.s.\ for all $j\geq 1$ and $\lim_{j\to \infty} \tau_j =\tau$ a.s.
\item for all $j\geq 1$, the process $\one_{[0,\tau_j]\times \O} v_i$ is progressively measurable.
\item a.s.\ for all $j\geq 1$, we have $v_i\in L^p(0,\tau_j ,w_{\a};H^{2-\s,q}(\Tor^d))$ and
\begin{equation}
\label{eq:integrability_nonlinearity}
\div(F_i(\cdot, v)) +f_i(\cdot, v)\in L^p(0,\tau_j,w_{\a};H^{-\s,q}(\Tor^d)).
\end{equation}
\item a.s.\ for all $j\geq 1$ the following holds for all $t\in [0,\tau_j]$:
\begin{equation}
\label{eq:reaction_diffusion_global_stochastic_integrated_form}
\begin{aligned}
v_i(t)-v_{0,i}
&=\int_{0}^{t} \Big[(\ellip_i+\ellip) \Delta v_i + \div(F_i(\cdot, v)) +f_i(\cdot, v)\Big]\,\dd s\\
&+\int_{0}^t\one_{[0,\tau_j]}
\Big( \big[ \theta_k(\sigma_{k,\alpha}\cdot\nabla)v_i\big] \Big)_{k,\alpha} \dd \Br_{\ell^2}(s).
\end{aligned}
\end{equation}
\end{itemize}
A sequence of stopping times $(\tau_j)_{j\geq 1}$ satisfying the above is called a \emph{localizing sequence}.
\item $(v,\tau)$ is a \emph{unique} $(p,\a,\s,q)$-solution to \eqref{eq:reaction_diffusion} if for any other local $(p,\a,\s,q)$-solution $(v',\tau')$ to \eqref{eq:reaction_diffusion} we have $v=v'$ a.e.\ on $[0,\tau\wedge \tau')\times \O$.
\item $(v,\tau)$ is a \emph{$(p,\a,\s,q)$-solution} to \eqref{eq:reaction_diffusion} if for any other local $(p,\a,\s,q)$-solution $(v',\tau')$ to \eqref{eq:reaction_diffusion} we have $\tau'\leq \tau$ a.s.\ and $v=v'$ a.e.\ on $[0,\tau')\times \O$.
\end{itemize}
\end{definition}

Note that $(p,\a,\s,q)$-solutions are unique in the class of local $(p,\a,\s,q)$-solutions and are real valued due to \eqref{eq:theta_normalized_symmetric}--\eqref{eq:complex_conjucation_preserves_w}. 
As discussed below \cite[Definition 2.3]{AV22}, if $(v,\tau)$ is a local $(p,\a,\s,q)$-solution, then the deterministic and stochastic integrals in \eqref{eq:reaction_diffusion_global_stochastic_integrated_form} are well-defined. Indeed, the deterministic integral is defined as an $H^{-\s,q}$-valued Bochner integral due to \eqref{eq:integrability_nonlinearity} and 
$
v_i\in L^p(0,\tau_j,w_{\a};H^{2-\s,q})\subseteq L^2(0,\tau_j;H^{2-\s,q})$ a.s.\ (the inclusion follows from the H\"{o}lder inequality and $\a<\frac{p}{2}-1$). Similarly, the stochastic one is defined as an $H^{1-\s,q}$-valued It\^{o}'s integral by \cite[Theorem 4.7]{NVW13}, the previous mentioned regularity of $v_i$, the smoothness of $\sigma_{k,\alpha}$ and  $\#\{k\,:\,\theta_k\neq 0\}<\infty$.

Next we recall the following result from \cite{AV22} which will be needed below. 

\begin{theorem}[Local well-posedness and regularity]
\label{t:local}
Let Assumption \ref{ass:f_polynomial_growth} be satisfied. Assume that $\s\in (1,2)$, 
\begin{equation}
\label{eq:p_a_p_s_assumption_local}
q>\frac{d(h-1)}{2} \vee  \frac{d}{d-\s}, \ \ \  p\geq \frac{2}{2-\s}\vee q \ \ \text{ and } \ \ \a= \a_{p,\s}\stackrel{{\rm def}}{=}p(1-\frac{\s}{2})-1.
\end{equation}
Then for all $v_0\in L^q(\Tor^d;\R^{\ell})$ such that $v_0\geq 0$ (component-wise), there exists a (unique) $(p,\a_{p,\s},\s,q)$-solution to \eqref{eq:reaction_diffusion} such that for all $\g\in [0,\frac{1}{2})$ 
\begin{align}
\label{eq:positivity_local}
v&\geq 0 \ \  \text{ (component-wise)\ \  a.e.\ on $[0,\tau)\times \O\times \Tor^d$,}\\
\label{eq:regularity_v_local}
v&\in H^{\g,p}_{\loc}([0,\tau),w_{\a_{p,\s}};H^{2-\s-2\g,q}(\Tor^d;\R^{\ell}))\cap C([0,\tau);B^{0}_{q,p}(\Tor^d;\R^{\ell})) \text{ a.s.\ } 
\end{align}
Moreover, the following assertions hold.
\begin{enumerate}[{\rm(1)}]
\item\label{it:local_1} {\rm (Mass control)} a.s.\ for all $t\in [0,\tau)$,
\begin{equation*}
\int_{\Tor^d} |v(t,x)|\,\dd x \leq C_{\ell,\alpha_1,\dots\alpha_\ell} \Big( e^{\m_1 t }  \int_{\Tor^d} |v_0(x)|\,\dd x+\m_0\frac{e^{\m_1t}-1}{\m_1}  \Big).
\end{equation*}
\item\label{it:local_2} {\rm (Instantaneous regularization)} 
For all $\g_1\in (0,\frac{1}{2})$ and $\g_2\in (0,1)$
$$
v\in C^{\g_1,\g_2}_{{\rm loc}}((0,\tau)\times \Tor^d;\R^{\ell}) \ \text{ a.s.\ }
$$
\item\label{it:local_3} {\rm (Blow-up criterion)} For all $q_0>\frac{d(h-1)}{2}\vee 2$ and $0<s<T<\infty$,
$$
\P\Big(s<\tau<T,\, \sup_{t\in [s,\tau)}\|v(t)\|_{L^{q_0}}<\infty\Big)=0.
$$
\end{enumerate}
\end{theorem}

To check the condition $q>\frac{d}{d-\s}$ in \eqref{eq:p_a_p_s_assumption_local} it is enough to choose $\s$ close to $1$. Hence the first in \eqref{eq:p_a_p_s_assumption_local} is essentially equivalent to $q>\frac{d(h-1)}{2}$.
By \cite[Proposition 3.5]{AV22}, if the above result is applicable for two sets of exponents $(p,\s,q)$, then the corresponding solutions coincide. 

Eq.\ \eqref{eq:positivity_local} is of particular interest in applications as $v_i$ typically models a concentration. 
In \eqref{it:local_1}, $(\alpha_i,\m_j)$ are as in Assumption \ref{ass:f_polynomial_growth}\eqref{it:mass}.
Due to \eqref{it:local_3}, $\tau$ is called explosion or blow--up time of $v$.

\begin{proof}[Proof of Theorem \ref{t:local}]
The local existence part of Theorem \ref{t:local} and items \eqref{it:local_2} and \eqref{it:local_3} follow from \cite[Proposition 3.1 and Corollary 2.11(1)]{AV22} using that $L^q\embed B^0_{q,p}$ as $p\geq q$ (cf.\  \cite[Remark 2.8(c)]{AV22}). Note that the condition $p\geq \frac{2}{2-\s}$ is needed to ensure $\a_{p,\s}\geq 0$.

The positivity of $v$, i.e.\ \eqref{eq:positivity_local} follows from \cite[Theorem 2.13 and Proposition 3.5]{AV22} (see also \cite[Section 4]{Kry13} for the linear case).
It remains to prove \eqref{it:local_1}. Integrating \eqref{eq:reaction_diffusion} over $\Tor^d$ and using that $\int_{\Tor^d}(\sigma_{k,\alpha}\cdot\nabla) v_i\,\dd x=0$ as $\div\,\sigma_{k,\alpha}=0$, we have 
\begin{equation}
\label{eq:integral_balance_mass}
\int_{\Tor^d} v_i(t,x)\,\dd x =
\int_{\Tor^d} v_{i,0}(x)\,\dd x+\int_0^t \int_{\Tor^d} f_i(s,v)\,\dd x\,\dd s\ \ \text{ a.s.\ for all }t\in [0,\tau).
\end{equation}
Recall that $(v,\tau)$ is a $(p,\a_{p,\s},\s,q)$--solution and therefore \eqref{eq:integrability_nonlinearity} holds.  The latter and  \eqref{eq:integral_balance_mass} show that the mapping $ t \mapsto\int_{\Tor^d} v_i(t,x)\,\dd x$ is a.s.\ locally absolutely continuous on $[0,\tau)$. Let $(\alpha_i,\m_j)$ be as in Assumption \ref{ass:f_polynomial_growth}\eqref{it:mass} and set $M(t)\stackrel{{\rm def}}{=}\sum_{1\leq i\leq \ell} \alpha_i\int_{\Tor^d}  v_i(t,x)\,\dd x$. 
Hence, differentiating, 
multiplying by $\alpha_i$ \eqref{eq:integral_balance_mass}, and then summing over $i\in \{1,\dots,\ell\}$, we obtain 
$$
 \tfrac{\dd}{\dd t}M(t)\leq \m_0 + \m_1 M(t)\ \ \text{ a.s.\ for a.a.\ $t\in (0,\tau)$. }
$$
The Grownall lemma, $\min_{1\leq i\leq \ell} \alpha_i>0$ and \eqref{eq:positivity_local} readily yields \eqref{it:local_1}.
\end{proof}

Before going further let us discuss the role of $\s$ in Theorem \ref{t:local} (in practice, one chooses $\s$ close to $ 1$). 
Note that 
the case $\s=1$ is \emph{not} included in the result as it would lead to a weight $w_{\a_{p,\s}}\not\in A_{p/2}$ as $\a_{p,\s}=\a_{p,1}=\frac{p}{2}-1$ (here $A_{r}$ denotes the $r$-th Muckenhoupt class, see e.g.\ \cite{Grafakos1}). Recall that the $A_{p/2}$-setting are the natural one for SPDEs, see e.g.\ \cite[Section 7]{AV19} or \cite{LV21_singular}. 
Finally, we note that the choice of the value $\a_{p,\s}$ is optimal. Indeed the (space-time) Sobolev index of the path space $H^{\g,p}(0,T,w_{\a_{p,\s}};H^{2-\s-2\g,q})$ is equal to the one of the space of initial data $L^q$.

For later use we collect some further observations in the following

\begin{remark}[Further regularity results]\
\label{r:regularity_paths}
\begin{enumerate}[{\rm(a)}]
\item\label{it:regularity_paths_1}
Theorem \ref{t:local}\eqref{it:local_2} can be improved under additional smoothness assumptions on $(f,F)$, see \cite[Theorem 4.2]{AV22}. For instance, if $(f,F)$ are $x$--independent and smooth in $v$, then   
$$
v\in C^{\g,\infty}_{\loc}((0,\tau)\times \Tor^d;\R^{\ell}) \ \ \text{ a.s.\ for all }\g\in[0,\tfrac{1}{2}). 
$$
\item\label{it:regularity_paths_2}
The regularity near $t=0$ in \eqref{eq:regularity_v_local} can be improved under additional assumptions on $v_0$. In particular, by \cite[Proposition 3.1 and 3.5]{AV22}, if $v_0\in B^{1-2\frac{1+\a}{p}}_{q,p}$ for some $\a\in [0,\frac{p}{2}-1)$, then the $(p,\a_{p,\s},\s,q)$--solution $(v,\tau)$ of Theorem \ref{t:local} is also a $(p,\a,1,q)$--solution and it satisfies
$$
v\in H^{\g,p}_{\loc}([0,\tau),w_{\a};H^{1-2\g,p}(\Tor^d;\R^{\ell})) \cap C([0,\tau);B^{1-2\frac{1+\a}{p}}_{q,p}(\Tor^d;\R^{\ell}))\text{ a.s.\ for all }\g\in [0,\tfrac{1}{2}).
$$
Since $\a<\frac{p}{2}-1$ implies $B^{1-2\frac{1+\a}{p}}_{q,p}(\Tor^d;\R^{\ell})\embed L^q(\Tor^d;\R^{\ell})$, we also have $v\in C([0,\tau);L^q(\Tor^d;\R^{\ell}))$.
\end{enumerate}
\end{remark}

In an attempt to make this work as independent as possible from \cite{AV22}, we use Theorem \ref{t:local}\eqref{it:local_1}--\eqref{it:local_3} only to prove Theorem \ref{t:delayed_blow_up_t_infty}, while Theorem \ref{t:delayed_blow_up} only uses the local well-posedness of \eqref{eq:reaction_diffusion}.
A careful inspection of the proof of Theorem \ref{t:delayed_blow_up} shows that \eqref{eq:positivity_local} is not used (however, it will be needed for solutions of its \emph{deterministic} version, see Proposition \ref{prop:global_high_viscosity}).
Finally, Remark \ref{r:regularity_paths}\eqref{it:regularity_paths_1} (resp.\ \eqref{it:regularity_paths_2}) is used in Theorem \ref{t:chemical_reaction_intro} (resp.\ Proposition \ref{prop:local_cut_off} below).

\subsection{Main results}
\label{ss:main_results}
In this subsection we state the main results of this paper. To this end, let us introduce the following deterministic version of \eqref{eq:reaction_diffusion} with increased diffusion: 
\begin{equation}
\label{eq:reaction_diffusion_det_statements}
\left\{
\begin{aligned}
\partial_t \vdi &=(\ellip+\ellip_i)\Delta \vdi +\big[ \div(F_i(\cdot,\vd)) + f_i(\cdot,\vd)\big], & \text{ on }&\Tor^d,\\
\vdi(0)&=v_{0,i}, & \text{ on }&\Tor^d,
\end{aligned}\right.
\end{equation}
where $\ellip>0$ is as in \eqref{eq:reaction_diffusion}.
The notion of $(p,\a,\s,q)$-solution to \eqref{eq:reaction_diffusion_det_statements} is as in Definition \ref{def:solution}. Compared to Definition \ref{def:solution}, for \eqref{eq:reaction_diffusion_det_statements}, we can use the full positive $A_p$-range $\a\in [0,p-1)$ as the problem \eqref{eq:reaction_diffusion_det_statements} is deterministic. To economize the notation we say that $v$ is a $(p,q)$-solution to  \eqref{eq:reaction_diffusion_det_statements} in case is a $(p,\a,\s,q)$-solution to such problem with $\s=1$ and $\a=\a_{p,\s}=\a_{p,1}$.

To apply the next result one needs to fix five parameters $(N,T,\varepsilon,\ellip_0,r)$. Roughly speaking, $N$ bounds the size of the initial data $v_0$, $T$ is the time horizon where our solutions lives, $\varepsilon$ bounds the size of the event where the solution $v$ might explode, $\ellip_0$ is the lower bound for the increased diffusion and $r$ is the time integrability exponent in which we measure the convergence of \eqref{eq:reaction_diffusion} to the deterministic problem \eqref{eq:reaction_diffusion_det_statements} with increased diffusion.

\begin{theorem}[Delayed blow-up and weak enhanced diffusion]
\label{t:delayed_blow_up}
Let Assumption \ref{ass:f_polynomial_growth} be satisfied. 
Let $(q,p,\a_{p,\s})$ be as in  \eqref{eq:p_a_p_s_assumption_local} for some $\s\in  (1,2)$. 
Fix $
N\geq 1$, $\varepsilon\in (0,1)$,  $T,\ellip_0\in (0,\infty)$ and $r\in (1,\infty)$.  
Then  there exist 
\begin{equation*}
\ellip\geq \ellip_0  \quad  \text{ and } \quad  \theta\in \ell^2(\Z_0^d)  \ \  \text{ with } \ \  \#\{\theta_k\neq 0\}<\infty
\end{equation*}
such that, for all 
\begin{equation}
\label{eq:data_L_q_N_statement}
v_0\in L^{q}(\Tor^d;\R^{\ell})  \ \text{ satisfying } \ v_0\geq 0 \text{ on }\Tor^d\ \text{ and } \ 
\|v_0\|_{L^{q}(\Tor^d;\R^{\ell})}\leq N,
\end{equation}
the unique $(p,\a_{p,\s},\s,q)$-solution $(v,\tau)$ 
to \eqref{eq:reaction_diffusion} provided by Theorem \ref{t:local} satisfies the following.
\begin{enumerate}[{\rm(1)}]
\item\label{it:delayed_blow_up} {\em (Delayed blow-up)}
The solution $v$ exists up to time $T$ with high probability:
$$
\P(\tau\geq T)>1-\varepsilon.
$$
\item\label{it:enhanced_dissipation} {\em (Weak enhanced diffusion)} There exists a (unique) $(p,q)$-solution $\vd$ to \eqref{eq:reaction_diffusion_det_statements} on $[0,T]$, and the solutions $v$ and $\vd$ are close in the following sense:
\begin{align*}
\P\big(\tau\geq T,\, \|v-\vd\|_{L^r(0,T;L^q(\Tor^d;\R^{\ell}))}\leq\varepsilon\big)
&>1-\varepsilon.
\end{align*}
\end{enumerate}
\end{theorem}

It is interesting to note that the parameters $(\ellip,\theta)$ are independent of $v_0$ satisfying \eqref{eq:data_L_q_N_statement} (however, they may depend on $N$). The choice of $(\ellip,\theta)$ is not unique. Indeed, as the proof of Theorem \ref{t:delayed_blow_up} shows, one can always enlarge $\ellip$ still keeping the assertions \eqref{it:delayed_blow_up}--\eqref{it:enhanced_dissipation} true. The same is also valid for Theorem \ref{t:delayed_blow_up_t_infty} below.
Other possible choices of $\theta$ will be given in Remark \ref{r:choice_theta} below.
Finally, let us remark that $\vd$ in Theorem \ref{t:delayed_blow_up}\eqref{it:enhanced_dissipation} is actually a $(p,q)$--solution to \eqref{eq:reaction_diffusion_det_statements} given by Proposition \ref{prop:global_high_viscosity} below. In particular $\vd\in L^{\infty}(0,T;L^q(\Tor^d;\R^{\ell}))$ for all $T<\infty$.

In case of exponentially decreasing mass we can allow $T=\infty$ in Theorem \ref{t:delayed_blow_up}.  
By Theorem \ref{t:local}\eqref{it:local_1}, exponentially decreasing mass happens if Assumption \ref{ass:f_polynomial_growth}\eqref{it:mass} holds with $\m_0=0$ and $\m_1<0$.
To apply the following result one fixes five parameters $(N,\varepsilon,\ellip_0,r,q_0)$. Compared to Theorem \ref{t:delayed_blow_up}, the time horizon is $T=\infty$ and we have an additional parameter $q_0<q$ for the space integrability in the weak enhanced diffusion assertion. 

\begin{theorem}[Global existence and weak enhanced diffusion]
\label{t:delayed_blow_up_t_infty}
Let Assumption \ref{ass:f_polynomial_growth} be satisfied.  
Suppose that Assumption \ref{ass:f_polynomial_growth}\eqref{it:mass} with $\m_0=0$ and $\m_1<0$. 
Let $(q,p,\a_{p,\s})$ be as in  \eqref{eq:p_a_p_s_assumption_local}  for some $\s\in  (1,2)$.
Fix $
N\geq 1$, $\varepsilon\in (0,1)$,  $\ellip_0\in (0,\infty)$, $r\in (1,\infty)$ and $q_0\in (1,q)$. 
Then  there exist 
\begin{equation*}
\ellip\geq \ellip_0  \quad  \text{ and } \quad  \theta\in \ell^2(\Z_0^d)  \ \  \text{ with } \ \  \#\{\theta_k\neq 0\}<\infty
\end{equation*}
such that, for all 
\begin{equation}
\label{eq:data_L_q_N_statement_global}
v_0\in L^{q}(\Tor^d;\R^{\ell})  \ \text{ satisfying } \ v_0\geq 0 \text{ on }\Tor^d\ \text{ and } \ 
\|v_0\|_{L^{q}(\Tor^d;\R^{\ell})}\leq N,
\end{equation}
the unique $(p,\a_{p,\s},\s,q)$-solution $(v,\tau)$  
to \eqref{eq:reaction_diffusion} provided by Theorem \ref{t:local} satisfies the following.
\begin{enumerate}[{\rm(1)}]
\item\label{it:delayed_blow_up_global} {\em (Global existence)} The solution $v$ is \emph{global} in time with high-probability:
$$
\displaystyle{\P(\tau= \infty)>1-\varepsilon.}
$$
\item\label{it:enhanced_dissipation_global} {\em (Weak enhanced diffusion)} There exists a (unique) $(p,q)$-solution $\vd$ to \eqref{eq:reaction_diffusion_det_statements} on $[0,\infty)$ and
\begin{align*}
\P\big(\tau= \infty,\, \|v-\vd\|_{L^{r}(\R_+;L^{q_0}(\Tor^d;\R^{\ell}))}\leq \varepsilon\big)>1-\varepsilon.
\end{align*}
\end{enumerate}
\end{theorem}

The parameters $(\ellip,\theta)$ in Theorem \ref{t:delayed_blow_up_t_infty} are independent of $v_0$ satisfying \eqref{eq:data_L_q_N_statement_global}. Moreover, we remark that $\vd$ in item \eqref{it:enhanced_dissipation_global} is  as in Lemma \ref{l:global_high_viscosity_infty} and therefore $\vd\in L^{\zeta}(\R_+;L^{q_0}(\Tor^d;\R^{\ell}))$ for all $\zeta<\infty$.
Let us conclude this subsection with several remarks.

\begin{remark}[Refined weak enhanced diffusion]
\label{r:refined_enhanced_dissipation}
As the proof of Theorem \ref{t:delayed_blow_up}  (resp.\  Theorem \ref{t:delayed_blow_up_t_infty}) shows that the norm $L^r(0,T;L^q)$ (resp.\ $L^q(\R_+;L^{q_0})$) in Theorem \ref{t:delayed_blow_up}\eqref{it:enhanced_dissipation} (resp.\ \ref{t:delayed_blow_up_t_infty}\eqref{it:enhanced_dissipation_global}) 
can be replaced by $C([0,T];H^{-\g})\cap L^2(0,T;H^{1-\g})\cap L^r(0,T;L^q)$ (resp.\ $C([0,\infty);H^{-\g})\cap L^r(\R_+;L^{q_0})$) where $\g>0$. In such a case the parameters $(\theta,\ellip)$ also depend on $\g>0$.
\end{remark}

\begin{remark}[On the choice of $\theta$]
\label{r:choice_theta}
The proof of Theorems \ref{t:delayed_blow_up} and \ref{t:delayed_blow_up_t_infty} also reveals other possible choices of $\theta$. Indeed, for each sequence $(\theta^{(n)})_{n\geq 1}\subseteq \ell^2(\Z^d_0)$ satisfying \eqref{eq:theta_normalized_symmetric} for all $n\geq 1$ and 
\begin{equation}
\label{eq:lim_l_infty_zero}
\lim_{n\to \infty} \|\theta^{(n)}\|_{\ell^{\infty}(\Z^d_0)}=0,
\end{equation}
there exists $n_*>0$ sufficiently large such that the assertions of Theorems \ref{t:delayed_blow_up}--\ref{t:delayed_blow_up_t_infty} hold for all $\theta=\theta^{(n)}$ with $n\geq n_*$ (cf.\ Proposition \ref{t:weak_convergence} below). As in \cite{FL19}, an example is given by 
\begin{equation}
\label{eq:choice_theta}
\theta^{(n)} =\frac{\Theta^{(n)}}{\|\Theta^{(n)}\|_{\ell^2}} , \quad \text{ where }\quad 
\Theta^{(n)}(k)\stackrel{{\rm def}}{=}\one_{\{n\leq |k|\leq 2n\}} \frac{1}{|k|^{\g}}  \   \text{ for $\g>0$ and }k\in \Z_0^d.
\end{equation}
The above example also satisfies $\#\{k\,:\,\theta_k^{(n)}\neq 0\}<\infty$ for all $n$. Interestingly, the sequence \eqref{eq:choice_theta} satisfies $\supp\theta^{(n)}\subseteq \{n\leq k\leq 2n\}$ and therefore it only acts on high Fourier modes. Moreover, as we may enlarge $n$, such frequencies can be chosen as large as needed. 
We will employ such sequence later on, but of course other choices are possible, see e.g.\ \cite[Remark 5.7]{FL19} and \cite[Remark 1.8]{FGL21_mixing} for Kraichnan's type noise. 
\end{remark}

\begin{remark}[Inhomogeneous diffusion/deterministic transport]
\label{r:general_second_order_operator}
The operator $\ellip_i\Delta v$ in \eqref{eq:reaction_diffusion} can be replaced by a general second order operator $\div(a_i\cdot\nabla v_i)+ (b_i\cdot\nabla) v_i + \div(B_i v_i)+ c_i v_i$ where the (deterministic) coefficients $(a_i,b_i,B_i,c_i)$ are $\alpha$--H\"{o}lder continuous with $\alpha>0$ and $a_i$ is a bounded elliptic matrix with ellipticity constant $\ellip_i>0$. In such a case, the results of Theorems \ref{t:delayed_blow_up}--\ref{t:delayed_blow_up_t_infty} still hold if $\s<1+\alpha$ (this restriction comes from the application of \cite{AV21_SMR_torus} in Theorem \ref{t:global_cut_off} below). 
\end{remark}

\begin{remark}[The case of constant mass]
The assumptions $\m_0=0$ and $\m_1<0$ in Theorem \ref{t:delayed_blow_up_t_infty} cannot be removed in general.
 However, in case of constant mass (i.e.\ $\m_0=\m_1=0$ in Assumption \ref{ass:f_polynomial_growth}\eqref{it:mass}), we expect that Theorem \ref{t:delayed_blow_up_t_infty} still holds. Indeed, it is often true that solutions to the deterministic version of \eqref{eq:reaction_diffusion} converges exponentially to a steady state $v_{\infty}$, see e.g.\ \cite{AMT00_convergence,DF06,DiFFM08_entropy,DFT17,FT18_normalized_entropy,DJT20} for some examples.
In this scenario, Theorem \ref{t:delayed_blow_up_t_infty} concerns the case $v_{\infty}=0$. However, compared to the references before, here we do not assume any global existence a--priori and in particular any assumption on $h$. 
It would be interesting to see if entropy methods, as used in the above references, can allow us to extend 
Theorem \ref{t:delayed_blow_up_t_infty} in case $\m_0=\m_1=0$. 
\end{remark}

\begin{remark}[Navier-Stokes equations]
\label{r:NS_3D}
It is natural to ask for similar results for Navier-Stokes equations perturbed by transport noise. 
Note that the equations considered in \cite{FL19} are \emph{not} equivalent to those (see \cite[Subsection 1.2 and Appendix 2]{FL19}).  
Although the $L^p(L^q)$--setting for the Navier-Stokes equations with transport noise has been developed in \cite{AV20_NS}, at the moment an extension of Theorems \ref{t:delayed_blow_up}--\ref{t:delayed_blow_up_t_infty} to such problem seems out of reach. Among others, one of the main issue seems the extension of Theorem \ref{t:global_cut_off}\eqref{it:global_cut_off_2} below. To prove the latter we exploit the fact that the nonlinearities $(f(\cdot,v),\div(F(\cdot,v)))$ and the transport noise are \emph{local} in $v$. The latter fact is not true for the Navier-Stokes equations due to the Helmholtz projection.
\end{remark}

\subsection{Strategy of the proofs}
\label{ss:strategy_intro}
In this subsection we summarize the strategy in the proof of our main results. It consists of three main steps:
\begin{enumerate}[{\rm(1)}]
\item\label{it:global_scheme_proof} Global existence and $\theta$--uniform $L^{\infty}_t(L^q_x)$--estimates for \eqref{eq:reaction_diffusion} with cut-off.
\item\label{it:high_viscosity_scheme_proof} Global existence for the deterministic version of \eqref{eq:reaction_diffusion} for high diffusivity. 
\item\label{it:scaling_limit_scheme_proof} Scaling limit for \eqref{eq:reaction_diffusion} with cut-off. 
\end{enumerate}

Roughly, the strategy follows the one of  \cite{FL19,FGL21}. 
 However, as commented in Subsection \ref{ss:role_criticality},  to handle the arbitrary large growth of the nonlinearities in \eqref{eq:reaction_diffusion}, in \eqref{it:global_scheme_proof}--\eqref{it:high_viscosity_scheme_proof}
 we exploit the full strength of maximal $L^p(L^q)$--techniques.

\eqref{it:global_scheme_proof}: In Section \ref{s:global_cut_off}, we consider \eqref{eq:reaction_diffusion_system_intro} with cut-off on $\Tor^d$:
\begin{align}
\label{eq:cut_off_intro}
\dd v_i -\ellip_i\Delta v_i \,\dd t
&=\phi_{R,r}(\cdot,v)\big[\div (F_i(\cdot,v)) + f_{i}(\cdot, v)\big]\,\dd t \\
\nonumber
&\ + \sqrt{c_d\ellip} \sum_{k,\alpha} \theta_k (\sigma_{k,\alpha}\cdot \nabla) v_i\circ \dd w_t^{k,\alpha}.
\end{align}
Here, for $R\geq 1$ and suitable parameters $q,r\in (1,\infty)$, $\phi_{R,r}$ is a cut-off given by
$
\phi_{R,r} (t,v)\stackrel{{\rm def}}{=}\phi\big(R^{-1}\|v\|_{L^r(0,t;L^q)}\big),
$
where $\phi$ is a bump function satisfying $\phi|_{[0,1]}=1$.
As we have seen in Subsection \ref{ss:role_criticality} the choice of the cut--off is related to the subcriticality of $L^q$ with $q>\frac{d(h-1)}{2}$. In Theorem \ref{t:global_cut_off} we prove global existence of unique strong solutions to \eqref{eq:cut_off_intro} and $L^{\infty}_t(L^q_x)$--estimates with constants \emph{independent} of $\theta$ (recall that we are assuming $\|\theta\|_{\ell^2}=1$). The latter estimates are obtained by mimicking a \emph{Moser--type iteration}. Recall that, as commented in Subsection \ref{ss:enhanced_dissipation}, we cannot use the spatial smoothness of the noise to obtain estimates with constants independent of $\theta$. In the proof of Theorem \ref{t:global_cut_off} the subcriticality of $L^q$ plays a key role.

\eqref{it:high_viscosity_scheme_proof}: In Section \ref{s:deterministic_high_diffusion} we show that the deterministic reaction-diffusion equations
\begin{equation}
\label{eq:v_det_intro}
\partial_t\vdi -\mu_i\Delta \vdi \,=\div(F_i(\cdot,\vd))+ f_{i}(\cdot, \vd)\quad \text{ on }\Tor^d,
\end{equation}
has a unique strong solutions on $[0,T]$, for any given $T<\infty$, provided $\mu_i(T) \gg0 $; see Proposition \ref{prop:global_high_viscosity}. 
Moreover, we investigate certain \emph{weak--strong} uniqueness result for a class of weak solutions appearing in the scaling limit argument of \eqref{it:scaling_limit_scheme_proof}, see Corollary \ref{cor:uniqueness}.

\eqref{it:scaling_limit_scheme_proof}: For all $n\geq 1$, consider the solution $v^{(n)}$ to \eqref{eq:cut_off_intro}  with $\theta=\theta^{(n)}$ where $\theta^{(n)}$ is as in \eqref{eq:choice_theta}. Then, using the $\theta$--independence of the $L^{\infty}_t(L^q_x)$--estimates in \eqref{it:global_scheme_proof} and a compactness argument, up to a subsequence, we have that $v^{(n)}\to \vd$ in probability in $L^r(0,T;L^q)$ where $\vd$ solves \eqref{eq:v_det_intro} with $\mu_i=\ellip+\ellip_i$. Here $\ellip$ is as in the stochastic perturbation of \eqref{eq:cut_off_intro}. Theorems \ref{t:delayed_blow_up}--\ref{t:delayed_blow_up_t_infty} now follow by choosing $\ellip$ very large so that \eqref{it:high_viscosity_scheme_proof} applies with $\|\vd\|_{L^r(0,T;L^q)}\leq R-1$ and choose $n_*$ large enough so that $\|v^{(n)}-\vd\|_{L^r(0,T;L^q)}\leq 1$ for all $n\geq n_*$ with high probability. Thus, for all $n\geq n_*$, we have $\phi_{R,r}(\cdot,v^{(n)})=1$ and $v^{(n)}$ solves \eqref{eq:reaction_diffusion}  on $[0,T]$ with high-probability for $\theta=\theta^{(n)}$.

Due to technical problems related to anisotropic spaces (cf.\ the discussion below Theorem \ref{t:local}), the above argument works only if $v_0$ has positive smoothness in a Besov scale, see Proposition \ref{prop:delayed_blow_up_smooth}. To show Theorem \ref{t:delayed_blow_up} we need an approximation argument which requires to study \eqref{eq:reaction_diffusion} with a stronger cut--off compared to the one used in Section \ref{s:global_cut_off}, see Lemma \ref{l:interp_inequality_final_one}. 
Finally, to prove Theorem \ref{t:delayed_blow_up_t_infty}, we exploit that the mass is exponentially decreasing due to Theorem \ref{t:local}\eqref{it:local_1} with $\m_0=0$ and $\m_1<0$. See \cite[Theorem 1.5]{FGL21} and \cite[Theorem 1.6]{FL19} for similar situations.

\section{Stochastic reaction-diffusion equations with cut-off}
\label{s:global_cut_off}
In this section we consider the following version of \eqref{eq:reaction_diffusion_system_intro} with cut-off:
\begin{equation}
\label{eq:reaction_diffusion_system_truncation}
\left\{
\begin{aligned}
\dd v_i -\ellip_i\Delta v_i \,\dd t
&= \phi_{R,r} (\cdot,v)\Big[\div (F(\cdot,v))+f_{i}(\cdot, v)\Big]\,\dd t \\
&+ \sqrt{c_d\ellip} \sum_{k,\alpha} \theta_k (\sigma_{k,\alpha}\cdot \nabla) v_i\circ \dd w_t^{k,\alpha}, &  \text{ on }&\Tor^d,\\
v_i(0)&=v_{i,0},  &  \text{ on }&\Tor^d.
\end{aligned}\right.
\end{equation}
As before $i\in \{1,\dots,\ell\}$ for some integer $\ell\geq 1$. Moreover $\phi_{R,r}(\cdot,v)$ stands for the cut-off 
\begin{equation}
\label{eq:def_cut_off}
\phi_{R,r}(t,v)\stackrel{{\rm def}}{=}\phi\Big(\frac{1}{R}\|v\|_{L^r(0,t;L^{q}(\Tor^d;\R^{\ell}))}\Big) \ \ \text{ where }\ \  R>0, \ r\in (1,\infty),
\end{equation}
and  $\phi\in C^{\infty}(\R)$ satisfies $\phi|_{[0,1]}=1$ and $\phi|_{[2,\infty)}=0$. The notion of  
 $(p,\a,\s,q)$-solutions to \eqref{eq:reaction_diffusion_system_truncation} can be given as in Definition \ref{def:solution}. 
The aim of this section is to prove the following result.

\begin{theorem}[Global existence and uniform estimates for \eqref{eq:reaction_diffusion_system_truncation}]
\label{t:global_cut_off}
Let Assumption \ref{ass:f_polynomial_growth} be satisfied. Assume that $q> \frac{d(h-1)}{2}$ and $
v_0\in B^{1-2\frac{1+\a}{p}}_{q,p}(\Tor^d;\R^{\ell})$.
Suppose that $
\theta=(\theta_k)_{k}$ satisfies \eqref{eq:theta_normalized_symmetric} and $\#\{k\,:\, \theta_k\neq 0\}<\infty$.
Then there exists
$
r_0\in (1,\infty)
$
for which the following hold for all $r\in [r_0,\infty)$.
\begin{enumerate}[{\rm(1)}]
\item\label{it:global_cut_off_1} There exists a (unique) global $(p,\a,1,q)$--solution $v$ to \eqref{eq:reaction_diffusion_system_truncation} such that a.s.\ 
\begin{equation*}
v\in H^{\g,p}_{\loc}( [0,\infty),w_{\a};H^{1-2\g,q}(\Tor^d;\R^{\ell})) \cap C([0,\infty); B^{1-2\frac{1+\a}{p}}_{q,p}(\Tor^d;\R^{\ell})) \ \text{ for all }\g\in [0,\tfrac{1}{2}).
\end{equation*}
\item\label{it:global_cut_off_2} 
For all $T\in (0,\infty)$ there exists $C_T>0$ independent of $(\theta,v_0)$ such that  a.s.\ 
\begin{equation*}
\sup_{t\in [0,T]}\|v(t)\|_{L^{q}(\Tor^d;\R^{\ell})}^{q}
+\max_{1\leq i\leq \ell}\int_0^t \int_{\Tor^d}(1+ |v_i|^{q-2}) |\nabla v_i|^2\,\dd x\, \dd s\leq C_T(1+
\|v_{0}\|_{L^{q}(\Tor^d;\R^{\ell})}^{q}).
\end{equation*}
\end{enumerate}
\end{theorem}

The proof of Theorem \ref{t:global_cut_off} shows that $r_0\in (1,\infty)$ depends only on $(p,q,\a,h,d)$. Recall that 
\begin{equation}
\label{eq:elementary_embedding_Besov_in_Lq}
B^{1-2\frac{1+\a}{p}}_{q,p}\stackrel{(i)}{\embed} B^0_{q,1} \stackrel{(ii)}{\embed} L^q
\end{equation}
where $(i)$ follows from $2\frac{1+\a}{p}<1$ and $(ii)$ from elementary embeddings (see e.g.\ \cite[Proposition 2.1]{BevosBook}). Hence $v_0\in L^q$ and the RHS in the estimate of \eqref{it:global_cut_off_2} is finite. The crucial point in Theorem \ref{t:global_cut_off}\eqref{it:global_cut_off_2} is the independence of $C_T$ on $\theta$. Note that 
\begin{equation}
\label{eq:composition_rules_sobolev_regularity_power}
\int_{\Tor^d} |v_i|^{q-2} |\nabla v_i|^2 \,\dd x \eqsim \int_{\Tor^d} \Big| \nabla  \big[|v_i|^{q/2}\big]\Big|^2 \,\dd x. 
\end{equation}
Thus Theorem \ref{t:global_cut_off}\eqref{it:global_cut_off_2} and Sobolev embeddings yield, for all $T\in (0,\infty)$,
\begin{equation}
\label{eq:conseguence_of_estimate}
\|v\|_{L^q(0,T;L^{\xi}(\Tor^d;\R^{\ell}))}\lesssim_T 1+
\|v_{0}\|_{L^{q}(\Tor^d;\R^{\ell})} \ \  \ \text{ where }\ \  \
\xi=
\left\{
\begin{aligned}
&\in (2,\infty) &\text{ if }&d=2,\\
&\frac{qd}{d-2}&\text{ if }& d\geq 3,
\end{aligned}\right.
\end{equation}
and the implicit constant is independent of $(\theta,v_0)$.

The proof of Theorem \ref{t:global_cut_off} is spread over this section. 
More precisely, the proof of Theorem \ref{t:global_cut_off}\eqref{it:global_cut_off_1} and \eqref{it:global_cut_off_2} are given in Subsections \ref{ss:global_cut_off} and \ref{ss:uniform_estimate_cut_off}, respectively. In Subsection \ref{ss:local_cut_off} we investigate local existence for \eqref{eq:reaction_diffusion_system_truncation} which is an important preparatory step for the proof of Theorem \ref{t:global_cut_off}\eqref{it:global_cut_off_1}. 

\subsection{Local existence for reaction-diffusion equations with cut-off}
\label{ss:local_cut_off}
In this subsection we begin our analysis of the problem \eqref{eq:reaction_diffusion_system_truncation} with cut-off. 
Here we prove the existence of local unique solutions to \eqref{eq:reaction_diffusion_system_truncation}. Moreover, we provide a general blow-up criterion for the local solution to \eqref{eq:reaction_diffusion_system_truncation} which will be used in Subsection \ref{ss:global_cut_off} to prove that such solutions are actually global.


\begin{proposition}[Local existence and blow-up criterion with cut-off]
\label{prop:local_cut_off}
Let the assumptions of Theorem \ref{t:global_cut_off} be satisfied.
Then there exists $
r_0(p,q,\a,h,d)\in (1,\infty)$
for which the following hold for all $r\in [r_0,\infty)$.
\begin{enumerate}[{\rm(1)}]
\item\label{it:local_cut_off_1} \emph{(Local existence and regularity)} There exists a (unique) $(p,\a,1,q)$-solution $(v,\tau)$ to \eqref{eq:reaction_diffusion_system_truncation} such that a.s.\ $\tau>0$ and 
\begin{align*}
v&\in H^{\g,p}_{\loc}( [0,\tau),w_{\a};H^{1-2\g,q}(\Tor^d;\R^{\ell})) \cap C([0,\tau); B^{1-2\frac{1+\a}{p}}_{q,p}(\Tor^d;\R^{\ell})) \ \text{ for all }\g\in [0,\tfrac{1}{2}).
\end{align*}
\item\label{it:local_cut_off_2} \emph{(Blow-up criterion)} For all $T\in (0,\infty)$,
\begin{align*}
\P\Big(\tau<T,\,\max_{1\leq i\leq \ell}\Big\|\phi_{R,r} (\cdot,v)\big[ \div(F_i(\cdot,v)) + f_i(\cdot,v)\big]\Big\|_{L^p(0,\tau,w_{\a};H^{-1,q})}<\infty \Big)=0.
\end{align*}
\end{enumerate}
\end{proposition}

Proposition \ref{prop:local_cut_off} does not follow directly from the results of \cite{AV19_QSEE_1,AV19_QSEE_2} as the setting used there does not allow for the non--local (in time) operator $v\mapsto \phi_{R,r}(\cdot,v)$. 
However, the methods of \cite{AV19_QSEE_1,AV19_QSEE_2} are still applicable with minor modifications. Below we give some indications how to extend the proofs of \cite{AV19_QSEE_1,AV19_QSEE_2} to the present situation.

\begin{proof}[Proof of Proposition \ref{prop:local_cut_off} -- Sketch]
We split the proof into three steps.

\emph{Step 1: \eqref{it:local_cut_off_1} holds}.  
Consider the system of SPDEs \eqref{eq:reaction_diffusion} without cut-off. By Theorem \ref{t:local} and Remark \ref{r:regularity_paths}\eqref{it:regularity_paths_2}, there exists a  $(p,\a,1,q)$--solution $(\wh{v},\wh{\tau})$ to \eqref{eq:reaction_diffusion} (here we use $\wh{\cdot}$ to distinguish from solutions to \eqref{eq:reaction_diffusion_system_truncation} considered in this section). Note that, by \eqref{eq:elementary_embedding_Besov_in_Lq},
$$
\wh{v}\in C([0,\tau);B^{1-2\frac{1+\a}{p}}_{q,p})\subseteq C([0,\tau);L^q)\ \text{ a.s.\ }
$$ 
Thus, the following is a stopping time
$$
\tau_*\stackrel{{\rm def}}{=}\inf\{t\in [0,\wh{\tau})\,:\, \|\wh{v}\|_{L^r(0,t;L^q)}\geq R\} \ \  \text{ where }\ \ \inf\emptyset\stackrel{{\rm def}}{=}\wh{\tau}. 
$$
Note that $\phi_{R,r}(\cdot,\wh{v})=1$ a.e.\ on $[0,\tau_*)\times \O$. Therefore $(\wh{v}|_{[0,\tau_*)\times \O},\tau_*)$ is a local unique $(p,\a,1,q)$--solution to \eqref{eq:reaction_diffusion_system_truncation}. The existence of a local unique  $(p,\a,1,q)$--solution which is \emph{maximal} in the class of  local unique  $(p,\a,1,q)$--solution now follows as in Step 5b of \cite[Theorem 4.5]{AV19_QSEE_1}. 
Here $(v,\tau)$ is maximal in the following sense: for any other local \emph{unique} $(p,\a,1,q)$--solution $(v',\tau')$, one has $\tau'\leq \tau$ a.s.\ and $v'=v$ a.e.\ on $[0,\tau')\times\O$. Note that, at this point, we do not know if $(v,\tau)$ construct above is actually a $(p,\a,1,q)$--solution. However, this will be a consequence of the blow-up criterion of Step 2 below (see \cite[Remark 5.6]{AV19_QSEE_2} for a similar situation).

To establish the blow-up criterion of \eqref{it:local_cut_off_2} we follow the arguments in \cite[Subsection 5.2]{AV19_QSEE_2} which was devoted to the proof of \cite[Theorem 4.10(2)]{AV19_QSEE_2} that is closely related to  \eqref{it:local_cut_off_2}. 
The result of Step 2 should be compared with \cite[Lemma 5.4]{AV19_QSEE_2}.

\emph{Step 2: (Intermediate blow-up criterion).\ Let $(v,\tau)$ be the unique local $(p,\a,1,q)$--solution to \eqref{eq:reaction_diffusion_system_truncation} provided in \eqref{it:local_cut_off_1} (cf.\ Step 1). Then}
\begin{align}
\label{eq:blow_up_easy_truncated}
\P\Big(\tau<T,\, \lim_{t\uparrow \tau}v \text{ exists in }B_{q,p}^{1-\frac{2}{p}},\, \mathcal{N}_{\a}(\tau,v)<\infty\Big)=0\ \  \text{\emph{ for all $T\in (0,\infty)$}},&\\
\nonumber 
\mathcal{N}_{\a}(\tau,v)\stackrel{{\rm def}}{=}\max_{1\leq i\leq \ell}\Big\|\phi_{R,r} (\cdot,v)\big[ \div(F_i(\cdot,v))+f_i(\cdot,v) \big]\Big\|_{L^p(0,\tau,w_{\a};H^{-1,q})}.&
\end{align}
\emph{In particular, $(v,\tau)$ is a $(p,\a,1,q)$--solution to \eqref{eq:reaction_diffusion_system_truncation}.}

The last claim follows as in  \cite[Remark 5.6]{AV19_QSEE_2} once \eqref{eq:blow_up_easy_truncated} is proven.  
To prove \eqref{eq:blow_up_easy_truncated}, we argue by contradiction with the maximality of $(v,\tau)$ (see the text at the end of Step 1). Hence, by contradiction, assume that 
$$
\P\Big(\tau<T,\, \lim_{t\uparrow \tau}v \text{ exists  in }B_{q,p}^{1-\frac{2}{p}},\, \mathcal{N}_{\a}(\tau,v)<\infty\Big)>0.
$$
Thus there exist $M,\eta>0$ and a set $\V\in \F_{\tau}$ such that $\P(\V)>0$, and a.s.\ on $\V$, one has $\tau>\eta$ and
\begin{equation}
\label{eq:contradiction_V_exists_on_V}
\lim_{t\uparrow \tau }v \text{ exists in }B_{q,p}^{1-\frac{2}{p}}, \qquad \ 
\sup_{t\in [0,\tau) }\|v(t)\|_{B_{q,p}^{1-\frac{2}{p}}}\leq M, \qquad \
\mathcal{N}_{\a}(\tau,v)<\infty.
\end{equation}

Let $\phi$ be as below \eqref{eq:def_cut_off}. 
For all $u\in L^r(\tau,T;L^q)$ we set
$$
\phi_{v,\tau,R,r}(t,u)\stackrel{{\rm def}}{=}\phi\big(R^{-1}\| \one_{[0,\tau]} v + \one_{(\tau,T]} u\|_{L^r(0,t;L^q)}\big) \ \ \text{ on }\V,
$$
and $
\phi_{v,\tau,R,r}(t,u)\stackrel{{\rm def}}{=} 0$ on $\O\setminus \V$.
Consider the following version of \eqref{eq:reaction_diffusion_system_truncation} with modified cut-off:
\begin{equation}
\label{eq:reaction_diffusion_system_truncation_V}
\left\{
\begin{aligned}
\dd u_i -\ellip_i\Delta u_i \,\dd t&= \phi_{v,\tau,R,r}(\cdot,u)\Big[\div (F(\cdot,u))+f_{i}(\cdot, u)\Big]\,\dd t \\
&+ \sqrt{c_d\ellip} \sum_{k,\alpha} \theta_n (\sigma_{k,\alpha}\cdot \nabla) u_i\circ \dd w_t^{k,\alpha}, &  \text{ on }&\Tor^d,\\
u_i(\tau\vee \eta)&=\one_{\V}v(\tau),  &  \text{ on }&\Tor^d.
\end{aligned}\right.
\end{equation}
Note that $\one_{\V} v(\tau)\in L^{\infty}_{\F_{\tau}}(\O;B^{1-2/p}_{q,p})$ by \eqref{eq:contradiction_V_exists_on_V}.
One can check that the proof of 
\cite[Proposition 5.1]{AV19_QSEE_2} extends to the present setting (more precisely, the estimates below \cite[(5.9)]{AV19_QSEE_2} also hold). Thus, reasoning as in \cite[Proposition 5.1]{AV19_QSEE_2}, one sees that there exists a $(p,0,\s,q)$-solution $(u,\lambda)$ to \eqref{eq:reaction_diffusion_system_truncation_V} such that $\lambda>\tau\vee \eta$ a.s.\ (note that we use the trivial weight at time $\lambda\geq \eta$).
We remark that the stochastic maximal $L^p$--regularity estimates used in \cite{AV19_QSEE_2} holds by \cite[Theorem 1.2]{AV21_SMR_torus} and \cite[Proposition 3.12]{AV19_QSEE_2}.
Set 
$$
\tau_*\stackrel{{\rm def}}{=}\one_{\V} \lambda+ \one_{\O\setminus \V} \tau \qquad \text{ and }\qquad
v_*\stackrel{{\rm def}}{=}\one_{[0,\tau) \times \O} v+ \one_{[\tau, \lambda)\times \V} u.
$$
One can check that $(v_*,\tau_*)$ is a \emph{unique} local $(p,\a,1,q)$--solution which extends $(v,\tau)$ since $\P(\tau_*>\tau)>0$. This contradicts the maximality of $(v,\tau)$. Hence the claim of Step 2 follows.

\emph{Step 3: \eqref{it:local_cut_off_2} holds}. The claim of this step follows verbatim from the proof of Theorem 4.10(1) in \cite[Subsection 5.2]{AV19_QSEE_2} (here we are using that the SPDEs \eqref{eq:reaction_diffusion_system_truncation} are semilinear). 
\end{proof}

\subsection{Proof of Theorem \ref{t:global_cut_off}\eqref{it:global_cut_off_1}}
\label{ss:global_cut_off}
We begin with the following interpolation inequalities involving the nonlinearities in \eqref{eq:reaction_diffusion}.
Here the subcritical nature of the spaces considered comes into play.

\begin{lemma}
\label{lem:interpolation_strong_setting}
Let Assumption \ref{ass:f_polynomial_growth}\eqref{it:f_polynomial_growth_1} be satisfied. Assume that $\frac{d(h-1)}{2} \vee 2 <q<\infty$.
Then there exist $\varphi\in (0,\frac{1}{h})$ and $\psi\in (0,\frac{2}{h+1})$ such that, 
 for all $u\in H^{1,q}$,
\begin{align}
\label{eq:f_estimate_sublinear}
\|f(\cdot,u)\|_{H^{-1,q}}&\lesssim 1+\|u\|_{L^q}^{(1-\varphi)h}\|u\|_{H^{1,q}}^{\varphi h},\\
\label{eq:F_estimate_sublinear}
\|\div (F(\cdot,u))\|_{H^{-1,q}}&\lesssim 1+\|u\|_{L^q}^{(1-\psi)\frac{h+1}{2}}\|u\|_{H^{1,q}}^{\psi\frac{h+1}{2} }.
\end{align}
\end{lemma}

The key point is that the RHS\eqref{eq:f_estimate_sublinear}-\eqref{eq:F_estimate_sublinear} grows sub--linearly in $\|u\|_{H^{1,q}}$.

\begin{proof}[Proof of Lemma \ref{lem:interpolation_strong_setting}]
We split the proof into two steps.

\emph{Step 1: \eqref{eq:f_estimate_sublinear} holds}. Recall that $q> 2$ and $d\geq 2$ by assumption. By Sobolev embeddings, 
$$
L^{\zeta}\embed 
H^{-1,q} \quad \text{ where }\quad \zeta\stackrel{{\rm def}}{=}\frac{dq}{q+d}\in (1,\infty). 
$$
Therefore, using Assumption \ref{ass:f_polynomial_growth}\eqref{it:f_polynomial_growth_1}, we have
\begin{align*}
\|f(\cdot,u)\|_{H^{-1,q}}
\lesssim \|f(\cdot,u)\|_{L^{\zeta}}
\lesssim\|(1+|u|^{h})\|_{L^{\zeta}}\lesssim 1+\|u\|_{L^{h \zeta}}^h.
\end{align*}
Without loss of generality we assume that $h \zeta>q$, otherwise the previous inequality already gives \eqref{eq:f_estimate_sublinear}. If $h \zeta>q$, then by Sobolev embeddings we have $H^{\varphi,q}\embed L^{h \zeta}$ for some $\varphi>0$ such that 
$$
\varphi- \frac{d}{q}=-\frac{d}{h \zeta} \qquad \Longleftrightarrow \qquad 
\varphi=\frac{d}{q}\Big(1-\frac{1}{h}\Big)-\frac{1}{h}.
$$
Since $\|u\|_{H^{\varphi,q}}\lesssim \|u\|_{L^{q}}^{1-\varphi}\|u\|_{H^{1,q}}^{\varphi}$, we have
\begin{equation*}
\|f(\cdot,u)\|_{H^{-1,q}}\lesssim 1+\|u\|_{L^{q}}^{(1-\varphi)h}\|u\|_{H^{1,q}}^{\varphi h}.
\end{equation*}
To conclude Step 1, it remains to note that the condition $\varphi h<1$ follows from $q>\frac{d}{2}(h-1)$.

\emph{Step 2: \eqref{eq:F_estimate_sublinear} holds}. Reasoning as in Step 1, and noticing that $H^{\psi,q}\embed L^{q\frac{h+1}{2}}$ for $\psi-\frac{d}{q}=-\frac{2d}{q(h+1)}$ we have, for all $u\in H^{1,q}$,
\begin{align*}
\|\div(F(\cdot,u))\|_{H^{-1,q}}
\lesssim \|F(\cdot,u)\|_{L^q}
&\lesssim 1+ \|u\|_{L^{q\frac{h+1}{2}}}^{\frac{h+1}{2}}\\
&\lesssim 1+ \|u\|_{L^q}^{\frac{h+1}{2}(1-\psi)}\|u\|_{H^{1,q}}^{\psi \frac{h+1}{2}}.
\end{align*}
Since $\psi=\frac{d}{q}(\frac{h-1}{h+1})$, the condition $\psi \frac{h+1}{2}<1$ follows from $q>\frac{d}{2}(h-1)$.
\end{proof}

We are ready to prove Theorem \ref{t:global_cut_off}\eqref{it:global_cut_off_1}:

\begin{proof}[Proof of Theorem \ref{t:global_cut_off}\eqref{it:global_cut_off_1}]
Let $(\varphi,\psi)$ be as in Lemma \ref{lem:interpolation_strong_setting}. Assume that
$$
r_0\geq  \max\Big\{\frac{(1-\varphi)hp}{(1-\varphi h)}  ,\frac{(1-\psi)p}{(1-\psi \frac{h+1}{2})}\frac{h+1}{2}\Big\}\vee 2 
$$
and recall that $r\in [r_0,\infty)$.
By Proposition \ref{prop:local_cut_off}, it is enough to show $\tau=\infty$ a.s. To prove the latter we employ the blow-up criterion of Proposition \ref{prop:local_cut_off}\eqref{it:local_cut_off_2}.

\emph{Step 1: For all $i\in \{1,\dots,\ell\}$, $t\in [0,T]$ and $u\in L^r(0,t;L^{q})\cap L^p(0,t;H^{1,q})$, }
\begin{align*}
\Big\|\phi_{R,r}(\cdot,u)f_i(\cdot,u)\Big\|_{L^p(0,t,w_{\a};H^{-1,q})}
&\lesssim_{R,T,r} 1+\|u\|_{L^p(0,t,w_{\a};H^{1,q})}^{\varphi h},\\
\Big\|\phi_{R,r}(\cdot,u)\div (F_i(\cdot,u))\Big\|_{L^p(0,t,w_{\a};H^{-1,q})}
&\lesssim_{R,T,r} 1+\|u\|_{L^p(0,t,w_{\a};H^{1,q})}^{\psi\frac{h+1}{2}},
\end{align*}
{\em where the implicit constants are independent of $(t,u)$}.
Below we only prove the first estimate as the second one follows similarly.
Fix $i\in \{1,\dots,\ell\}$, $t\in [0,T]$ and $u\in L^r(0,t;L^{q})\cap L^p(0,t,w_{\a};H^{1,q})$. Let $e_u$ be the following (deterministic) exit time:
\begin{equation}
\label{eq:Xi_u_definition_L_r_q}
e_{u}\stackrel{{\rm def}}{=}\inf\{s\in [0,t]\,:\,\|u\|_{L^r(0,s;L^q)}\geq 2R \} \quad \text{ where }\quad \inf\emptyset\stackrel{{\rm def}}{=}t.
\end{equation}
Since $\phi|_{[2,\infty)}=0$, we have
$
\phi_{R,r}(s,u)=0 $ for all $ s\geq e_u
$.
Set $r_1\stackrel{{\rm def}}{=}\frac{(1-\varphi)hp}{(1-\varphi h)}\leq r_0$ and note that
\begin{align*}
\Big\|\phi_{R,r}(\cdot,u)f_i(\cdot,u)\Big\|_{L^p(0,t,w_{\a};H^{-1,q})}^p
& =
\Big\|\phi_{R,r}(\cdot,u)f_i(\cdot,u)\Big\|_{L^p(0,e_{u},w_{\a};H^{-1,q})}^p\\
&\stackrel{(i)}{\lesssim} 1+\int_0^{e_{u}} \|u(s)\|_{L^q}^{(1-\varphi)hp }\|u(s)\|_{H^{1,q}}^{\varphi h p}\, s^{\a}\, \dd s\\
& \stackrel{(ii)}{\lesssim} 
1+ \|u\|_{L^{r_1}(0,e_{u},w_{\a};L^q)}^{(1-\varphi)h p} \|u\|_{L^{p}(0,e_{u},w_{\a};H^{1,q})}^{\varphi h p}\\
&  \stackrel{\eqref{eq:Xi_u_definition_L_r_q}}{\leq}
1+ (2c R)^{(1-\varphi)h p} \|u\|_{L^{p}(0,t,w_{\a};H^{1,q})}^{\varphi h p}
\end{align*}
where $c(T,r,r_0)>0$ and in $(i)$ we used \eqref{eq:f_estimate_sublinear} and $0\leq \phi_{R,r}(\cdot,u)\leq 1$, in $(ii)$ the H\"{o}lder inequality with exponent $(\frac{1}{\varphi h},\frac{1}{1-\varphi h})$ and $r\geq r_0\geq r_1$. 
Note that $(ii)$ is valid since $\varphi h<1$ by Lemma \ref{lem:interpolation_strong_setting}.

\emph{Step 2: (Intermediate estimate). For all $T\in (0,\infty)$, there exists $c_0(T)>0$ such that, for all  $v_0\in B^{1-2(1+\a)/p}_{q,p}$,}
\begin{equation}
\label{eq:claim_step_2_estimate_a_priori_estimate_besov_space}
\E\|v\|_{L^p(0,\tau\wedge T,w_{\a};H^{1,q})}^p\leq c_0(1+\|v_0\|_{B^{1-2(1+\a)/p}_{q,p}}^p).
\end{equation}
Let $(\tau_n)_{n\geq 1}$ be a localizing sequence for $(v,\tau)$, cf.\ Definition \ref{def:solution}. For all $n\geq 1$, let 
$$
\gamma_n\stackrel{{\rm def}}{=}\inf\{t\in [0,\tau_n)\,:\,\|v\|_{L^p(0,\tau,w_{\a};H^{1,q})}\geq n\}\wedge T\ \ \text{ where }\ \ \inf\emptyset\stackrel{{\rm def}}{=}\tau_n\wedge T. 
$$
Since $(\tau_n)_{n\geq 1}$ is a localizing sequence and $v\in L^p_{\loc}([0,\tau),w_{\a};H^{1,q})$ a.s.\ by Proposition \ref{prop:local_cut_off}\eqref{it:local_cut_off_1}, we have $\lim_{n\to \infty}\gamma_n=\tau\wedge T$ a.s. 
The stochastic maximal $L^p$-regularity estimates of \cite[Theorem 2.1]{AV21_SMR_torus} and Step 1 yield, for some $C_0(R)>0$ independent of $(v_0,n)$
\begin{align*}
&\E\|v\|_{L^p(0,\gamma_n,w_{\a};H^{1,q})}^p\\
&\leq 
C_0  (1+\|v_0\|_{B^{1-2(1+\a)/p}_{q,p}}^p+  \E\|v\|_{L^p(0,\gamma_n,w_{\a};H^{1,q})}^{p \varphi h}+ \E\|v\|_{L^p(0,\gamma_n,w_{\a};H^{1,q})}^{p\psi\frac{h+1}{2}})\\
&\stackrel{(i)}{\leq }
C_1 (1+\|v_0\|_{B^{1-2(1+\a)/p}_{q,p}}^p)+ \frac{1}{2} \E\|v\|_{L^p(0,\gamma_n,w_{\a};H^{1,q})}^{p},
\end{align*}
where $C_1(C_0, \varphi,\psi,h)>0$ and in $(i)$ we used that $\max\{\varphi h ,\psi\frac{h+1}{2}\}<1$ and the Young inequality.
Note that $\|v\|_{L^p(0,\gamma_n,w_{\a};H^{1,q})}\leq n$ a.s.\ by the definition of $\gamma_n$. Thus, the above inequality yields
$$
\E\|v\|_{L^p(0,\gamma_n,w_{\a};H^{1,q})}^p\leq 
2C_1 (1+\|v_0\|_{B^{1-2(1+\a)/p}_{q,p}}^p).
$$
Since $C_1$ is independent of $(v_0,n)$ and $\lim_{n\to\infty}\g_n=\tau\wedge T$ a.s., the claim of Step 2 with $c_0=2C_1$ follows by letting $n\to \infty$ in the above estimate.

\emph{Step 3: Conclusion}. By Step 2 we know that $v\in L^p(0,\tau\wedge T,w_{\a};H^{1,q})$ a.s.\ for all $T<\infty$.  From Step 1 we deduce that, for all $T<\infty$,
\begin{equation}
\label{eq:global_existence_consequence_steps_1_2}
\max_{1\leq i\leq \ell}\Big\|\phi_{R,r} (\cdot,v)\big[ \div(F_i(\cdot,v))+f_i(\cdot,v) \big]\Big\|_{L^p(0,\tau\wedge T,w_{\a};H^{-1,q})}<\infty\ \  \text{a.s.\ }
\end{equation}
Therefore, by Proposition \ref{prop:local_cut_off}\eqref{it:local_cut_off_2},
\begin{align*}
\P(\tau<T)\stackrel{\eqref{eq:global_existence_consequence_steps_1_2}}{=}\P\Big(\tau<T,\,\max_{1\leq i\leq \ell}\Big\|\phi_{R,r} (\cdot,v)\big[ \div(F_i(\cdot,v))+f_i(\cdot,v) \big]\Big\|_{L^p(0,\tau,w_{\a};H^{-1,q})}<\infty \Big)=0.
\end{align*}
Hence $\tau\geq T$ a.s.\ for all $T<\infty$ and therefore $\tau=\infty$ a.s. 
\end{proof}


\begin{remark}[On the use of blow-up criteria]
\label{r:blow_up_criteria_necessity}
In the works \cite{FL19,FGL21} the analogue of Theorem \ref{t:global_cut_off} is proven by showing global existence and pathwise uniqueness which, in combination with a Yamada--Watanabe type argument, yields existence of global unique solutions. Our approach is different and it based on the construction of maximal solutions and blow-up criteria, following the scheme of \cite{AV19_QSEE_1,AV19_QSEE_2}. This strategy has two basic advantages. Firstly, the role of the sub-criticality is clear from the estimates of Lemma \ref{lem:interpolation_strong_setting} which in combination of a (relatively) soft argument gives global existence for \eqref{eq:reaction_diffusion_system_truncation}. Secondly, in an $L^p(L^q)$--setting, pathwise uniqueness is more difficult to achieve as it often difficult to estimate differences  like $f(\cdot,v^{(1)})-f(\cdot,v^{(2)})$. Indeed, such estimate seems possible only if one enforces the cut-off, cf.\ Lemma \ref{l:interp_inequality_final_one} below.
\end{remark}

\subsection{Proof of Theorem \ref{t:global_cut_off}\eqref{it:global_cut_off_2}}
\label{ss:uniform_estimate_cut_off}
Here we prove Theorem \ref{t:global_cut_off}\eqref{it:global_cut_off_2} by applying the It\^{o} formula to the functionals $v=(v_i)_{i=1}^{\ell}\mapsto \|v_i\|^{q}_{L^q}$, for $i\in \{1,\dots,\ell\}$, mimicking a Moser iteration (see e.g.\ \cite{DG15_boundedness}). To handle the nonlinear terms in \eqref{eq:reaction_diffusion_system_truncation} we need the following interpolation inequality.

\begin{lemma}[Interpolation inequality]
\label{l:interpolation_L_eta}
Assume that $d\geq 2$ and $T\in (0,\infty)$. Let $h>1$ and $\frac{d(h-1)}{2}\vee 2< q<\infty$.
Then there exist $r_*\in (1,\infty)$, $\alpha\in (1,\infty)$ and $\beta\in (0,1)$ such that, 
such that for all $t\in [0,T]$ and $u\in L^{\infty}(0,t;L^q(\Tor^d))\cap L^2(0,t;W^{1,q}(\Tor^d))$,
\begin{equation}
\label{eq:interp_inequality_II}
\begin{aligned}
\| u\|_{L^{q+h-1}((0,t)\times \Tor^d)}^{q+h-1}
&\lesssim_T
\|u\|_{L^{r_*}(0,t;L^{q}(\Tor^d))}^{q+h-1}\\
&+
\|u\|_{L^{r_*}(0,t;L^{q}(\Tor^d))}^{\alpha}
\Big(\int_0^t \int_{\Tor^d} |u|^{q-2}|\nabla u|^2\,\dd x\, \dd s\Big)^{\beta},
\end{aligned}
\end{equation}
where the implicit constant is independent of $u$ and $t\in (0,T]$.
\end{lemma}

The RHS\eqref{eq:interp_inequality_II} is finite due to the regularity assumptions on $u$. Indeed, $q> 2$ and H\"{o}lder inequality with exponents $(\frac{q}{2},\frac{q}{q-2})$ ensure
\begin{equation}
\label{eq:mixed_estimate_q}
\int_0^t \int_{\Tor^d} |u|^{q-2}|\nabla u|^2\,\dd x\, \dd s \leq \|u\|_{L^{\infty}(0,t;L^q)}^{q-2} \| \nabla u\|_{L^{2}(0,t;L^q)}^2<\infty.
\end{equation}
As in Lemma \ref{lem:interpolation_strong_setting}, the crucial point of Lemma \ref{l:interpolation_L_eta} is that $\beta<1$. This is of course due to the subcriticality of $L^q$ with $q>\frac{d}{2}(h-1)$, see Subsection \ref{ss:role_criticality}.
For exposition convenience we postpone the proof of Lemma \ref{l:interpolation_L_eta} at the end of this subsection.

\begin{proof}[Proof of Theorem \ref{t:global_cut_off}\eqref{it:global_cut_off_2}]
Fix $T\in (0,\infty)$. Without loss of generality we may assume that $r_0\geq r_*$ where $r_*$ is as in Lemma \ref{l:interpolation_L_eta}. 
To prove the claim of Step 1, we compute $\int_{\Tor^d}|v_i|^q\,\dd x$ and we estimate the nonlinearities by employing Lemma \ref{l:interpolation_L_eta}. As in \cite[Lemma 2]{DG15_boundedness}, we need an approximation argument. For $N\geq 1$, set 
$$
\psi_N (y)\stackrel{{\rm def}}{=}
\left\{
\begin{aligned}
&|y|^q  &\ \ \text{ if }&|y|\leq N,\\
&\frac{q(q-1)}{2} n^{q-2}(|y|-N)^2+ q N^q-1(|y|-N)+N^q & \text{ if }& |y|> N.
\end{aligned}\right.
$$
One can check that there exists $c\geq 1$ independent of $N\geq 1$ and $y\in \R$ such that
\begin{equation}
\label{eq:estimate_psi_n}
|\psi_N(y)|\leq c |y|^q , \qquad |\psi_N' (y)|\leq c  |y|^{q-1} , \qquad |\psi_N''(y)|\leq c|y|^{q-2}.
\end{equation}
Moreover, for all $y\in \R$, as $N\to \infty$ we have
\begin{equation}
\label{eq:convergence_psi_n}
\psi_N(y)\to |y|^q, \qquad \psi_N'(y)\to q|y|^{q-2} y, \qquad \psi_N''(y)\to q(q-1)|y|^{q-2}.
\end{equation}
The generalized It\^o formula (see e.g.\ \cite[Section 3]{Kry13} or \cite[Proposition A.1]{DHV16}) yields, for all $i\in \{1,\dots,\ell\}$, $N\geq 1$, $t\in [0,T]$ and a.s.,
\begin{equation}
\label{eq:identity_ito_psi_n}
\begin{aligned}
\int_{\Tor^d} \psi_N (v_i(t))\,\dd x
&-
\int_{\Tor^d} \psi_N (v_{0,i})\,\dd x
=-
 (\ellip+\ellip_i) \int_0^t \int_{\Tor^d} \psi''_N(v_i) |\nabla v_i|^2\,\dd x\, \dd s \\
&
+
\int_0^t   \int_{\Tor^d} \phi_{R,r}(\cdot,v) \big[ \psi_N'(v_i)f_i(\cdot,v)- \psi_N''(v_i)F_i(\cdot,v)\cdot\nabla v_i\big]\,\dd x\, \dd s\\
&
+
c_d \ellip
\sum_{k,\alpha} 
\int_{0}^t \int_{\Tor^d} \theta_k^2\, \psi_N''(v_i) |(\sigma_{k,\alpha}\cdot\nabla) v_i|^2\,\dd x\, \dd s
\end{aligned}
\end{equation}
where we used that the martingale part cancels since 
$$
\int_{\Tor^d} \psi_N'(v_i)[ (\sigma_{k,\alpha}\cdot \nabla) v_i]\,\dd x
\stackrel{(i)}{=}
\int_{\Tor^d}  (\sigma_{k,\alpha}\cdot \nabla)\big[\psi_N(v_i)\big]\,\dd x 
\stackrel{(ii)}{=}0\  \text{ a.s.\ }
$$
Here $(i)$ follows from the chain rule and  $(ii)$ from integrating by parts as well as $\div\,\sigma_{k,\alpha}=0$.

For the reader's convenience, we split the remaining proof into several steps.

\emph{Step 1: For all $i\in \{1,\dots,\ell\}$ we have, a.s.\ for all $t\in [0,T]$,}
\begin{align*}
\|v_i(t)\|_{L^q}^q
&
+ q(q-1) \ellip_i \int_0^t \int_{\Tor^d} |v_i|^{q-2} |\nabla v_i|^2\,\dd x\, \dd s\\
&=  
\|v_{0,i}\|_{L^q}^q 
+
q \int_0^t   \int_{\Tor^d} \phi_{R,r}(\cdot,v)  |v_i|^{q-2} \big[ f_i(\cdot,v) v_i- (q-1)F_i(\cdot,v)\cdot \nabla v_i\big] \,\dd x\, \dd s.
\end{align*}

Fix $i\in \{1,\dots,\ell\}$.
By taking $N\to \infty$ in \eqref{eq:identity_ito_psi_n} and using \eqref{eq:estimate_psi_n} we have, a.s.\ for all $t\in [0,T]$,
\begin{align}
\label{eq:ito_intermediate_with_ito_correction}
\|v_i(t)\|_{L^q}^q
&
+ q(q-1) (\ellip+\ellip_i) \int_0^t \int_{\Tor^d} |v_i|^{q-2} |\nabla v_i|^2\,\dd x\, \dd s\\
\nonumber
&= 
\|v_{0,i}\|_{L^q}^q 
+
q \int_0^t   \int_{\Tor^d} \phi_{R,r}(\cdot,v)  |v_i|^{q-2} \big[ f_i(\cdot,v) v_i- (q-1)F_i(\cdot,v)\cdot \nabla v_i\big] \,\dd x\, \dd s\\
\nonumber
& 
+q(q-1)
c_d \ellip 
\sum_{k,\alpha} 
\int_{0}^t \int_{\Tor^d}\theta_k^2\, |v_i|^{q-2}  |(\sigma_{k,\alpha}\cdot\nabla) v_i|^2\,\dd x\, \dd s.
\end{align}
It remains to discuss the legitimacy of using the Lebesgue domination theorem to take $N \to \infty $ in \eqref{eq:identity_ito_psi_n}.
Firstly, recall that, by Theorem \ref{t:global_cut_off}\eqref{it:global_cut_off_1},  $\a<\frac{p}{2}-1$ and \eqref{eq:elementary_embedding_Besov_in_Lq},
\begin{equation}
\label{eq:regularity_v_q}
v\in L^2(0,T;W^{1,q})\cap L^{\infty}(0,T;L^q) \text{ a.s.\ }
\end{equation}
Hence \eqref{eq:mixed_estimate_q} shows $ |v_i|^{q-2}|\nabla v_i|^2\in L^1((0,T)\times \Tor^d)$ a.s.
Moreover, by Assumption \ref{ass:f_polynomial_growth}\eqref{it:f_polynomial_growth_1},  
$$
|v_i |^{q-2}\phi_{R,r}(\cdot,v)\big| f_i(\cdot,v)v_i\big|\lesssim |v_i|^{q-2}(1+|v|^{h}) \lesssim 1+ |v|^{q+h-1}.
$$ 
Combining the above with \eqref{eq:regularity_v_q}, \eqref{eq:mixed_estimate_q} and Lemma \ref{l:interpolation_L_eta} we get
$$
|v_i |^{q-2}\phi_{R,r}(\cdot ,v) f_i(\cdot,v)v_i \in L^1((0,T)\times \Tor^d) \text{ a.s.\ }
$$ 
For the $F$-term we argue similarly. By Assumption \ref{ass:f_polynomial_growth}\eqref{it:f_polynomial_growth_1} and the Cauchy-Schwartz inequality,
\begin{align*}
\big|\phi_{R}(\cdot,v)|v_i|^{q-2} F_i(\cdot,v)\cdot\nabla v_i \big|
&\lesssim  |F_i(\cdot,v)|^2 |v_i|^{q-2} + |v_i|^{q-2} |\nabla v_i|^2 \\
&\lesssim 
1+|v|^{q+h-1}+ |v_i|^{q-2} |\nabla v_i|^2 \in L^1((0,T)\times \Tor^d) \text{ a.s.\ }
\end{align*}
Thus \eqref{eq:ito_intermediate_with_ito_correction} is proved. 
To conclude the proof of Step 1, it is enough to note that
$$
c_d
\sum_{k,\alpha}\int_{0}^t \int_{\Tor^d}\theta_k^2\, |v_i|^{q-2} |(\sigma_{k,\alpha}\cdot\nabla) v_i|^2\,\dd x\, \dd s
\stackrel{\eqref{eq:ellipticity_noise}}{=}  \int_{0}^t \int_{\Tor^d} |v_i|^{q-2} |\nabla v_i|^2\,\dd x\, \dd s.
$$

\emph{Step 2: Let $\ellip_0\stackrel{{\rm def}}{=}\min_{1\leq i\leq \ell} \ellip_i$. Then there exists $K>0$, independent of $(\theta,v_0)$, such that, for all $i\in \{1,\dots,\ell\}$,}
\begin{multline*}
\Big| \int_0^T   \int_{\Tor^d}   |v_i|^{q-2}\phi_{R,r}(\cdot,v) f_i(\cdot,v) v_i \,\dd x\, \dd s \Big|
+\Big| \int_0^T   \int_{\Tor^d}   |v_i|^{q-2}\phi_{R,r}(\cdot,v) F_i(\cdot,v) \cdot\nabla v_i \,\dd x\, \dd s \Big|\\
 \leq K+ \frac{3\ellip_0}{4}\sum_{1\leq i\leq \ell }  \int_0^T   \int_{\Tor^d}   |v_i|^{q-2}|\nabla v_i|^2 \,\dd x\, \dd s.
\end{multline*}

Fix $i\in \{1,\dots,\ell\}$.
In this step we use that $r_0\geq r_*$ where $r_*$ is as in Lemma \ref{l:interpolation_L_eta}.  
We first estimate the $f$-term. Let $e_v$ be the first exit time of $t\mapsto \|v\|_{L^r(0,t;L^{q})}$ from the interval $[0,2R]$, i.e.
$$
e_v\stackrel{{\rm def}}{=} \inf\big\{t\in [0,T]\,:\,\|v\|_{L^r(0,t;L^{q})}\geq 2R\big\}  \ \ \text{ where }\ \  \inf\emptyset\stackrel{{\rm def}}{=}T.
$$
By \eqref{eq:def_cut_off} and Assumption \ref{ass:f_polynomial_growth}\eqref{it:f_polynomial_growth_1} we have
\begin{align*} 
&\Big|\int_0^T   \int_{\Tor^d}   |v_i|^{q-2}\phi_{R,r}(\cdot,v) f_i(\cdot,v) v_i \,\dd x\, \dd s \Big|
= \Big|\int_0^{e_v} \int_{\Tor^d}  |v_i|^{q-2}\phi_{R,r}(\cdot,v) f_i(\cdot,v) v_i \,\dd x\, \dd s\Big|\\
&\qquad \qquad \qquad \lesssim 1+ \int_0^{e_v} \int_{\Tor^d} |v|^{q+h-1} \,\dd x\, \dd s
\lesssim 1+ \sum_{1\leq i\leq \ell}\int_0^{e_v} \int_{\Tor^d} |v_i|^{q+h-1} \,\dd x\, \dd s
\end{align*}
where the implicit constants depend only on $(q,h,\ell,\|f(\cdot,0)\|_{L^{\infty}(\Tor^d;\R^{\ell})})$.

By Lemma \ref{l:interpolation_L_eta}, for some $\beta\in (0,1)$, 
\begin{equation}
\label{eq:e_u_trick_estimate_R_reaction_diffusion_cut_off}
\begin{aligned}
&\int_0^{e_u} \int_{\Tor^d} |v_i|^{q+h-1} \,\dd x\, \dd s\\
&\lesssim 
\|v_i \|_{L^r(0, e_v;L^{q})}^{q+h-1} + 
\|v_i \|_{L^r(0, e_v;L^{q})}^{\alpha} \Big( \int_0^{e_v}   \int_{\Tor^d}|v_i|^{q-2} |\nabla v_i|^2\,\dd x\, \dd s \Big)^{\beta}\\
&\stackrel{(i)}{\lesssim}_{q,h,R} 1 + \Big( \int_0^{e_v}   \int_{\Tor^d}|v_i|^{q-2} |\nabla v_i|^2\,\dd x\, \dd s \Big)^{\beta}\\
&\leq  C(R,q,h,\ellip_0)+ \frac{\ellip_0}{4 }\int_0^{T}   \int_{\Tor^d}|v_i|^{q-2} |\nabla v_i|^2\,\dd x\, \dd s 
\end{aligned}
\end{equation}
where in $(i)$ we used $\|v\|_{L^r(0,e_v;L^q)}\leq 2R$ by definition of $e_v $. Hence we proved that 
$$
\Big|\int_0^T   \int_{\Tor^d}  \phi_{R,r}(\cdot,v)  |v_i|^{q-2} f_i(\cdot,v) v_i \,\dd x\, \dd s \Big|
\leq K+ \frac{\ellip_0}{4}\sum_{1\leq i\leq \ell}\int_0^{T}   \int_{\Tor^d}|v_i|^{q-2} |\nabla v_i|^2\,\dd x\, \dd s .
$$
Similarly we estimate the $F$-term. By Cauchy-Schwartz inequality and Assumption \ref{ass:f_polynomial_growth}\eqref{it:f_polynomial_growth_1},
\begin{equation}
\label{eq:F_gradient_cauchy_schwartz_proof_theta_ind_estimate}
\begin{aligned}
&\Big| \int_0^T   \int_{\Tor^d} \phi_{R,r}(\cdot,v)  |v_i|^{q-2} F_i(\cdot,v) \cdot\nabla v_i \,\dd x\, \dd s \Big|\\
&\leq \frac{\ellip_0}{4}  \int_0^{e_v}  \int_{\Tor^d}   |v_i|^{q-2}|\nabla v_i|^2 \,\dd x\, \dd s 
+ C(\ellip_0)  \int_0^{e_v}   \int_{\Tor^d}    |v_i|^{q-2}|F_i(\cdot,v)|^2 \,\dd x\, \dd s\\
&\leq \frac{\ellip_0}{4}  \int_0^{e_v}   \int_{\Tor^d}   |v_i|^{q-2}|\nabla v_i|^2 \,\dd x\, \dd s 
+ C(\ellip_0,T)\Big(1+   \int_0^{e_v}  \int_{\Tor^d}  |v|^{q+h-1} \,\dd x\, \dd s\Big).
\end{aligned}
\end{equation}
Since $|v|\leq \sum_{1\leq i\leq \ell} |v_i|$, the last integral can be estimated as in \eqref{eq:e_u_trick_estimate_R_reaction_diffusion_cut_off}.
Putting together the above estimates, one obtains the claim of Step 2. 

\emph{Step 3: Conclusion}. Summing over $i\in \{1,\dots,\ell\}$ the estimate of Step 1 and using the estimate of Step 2, one gets 
\begin{equation}
\label{eq:a_priori_estimates_cut_off_conclusion_1}
\sup_{t\in [0,T]}\|v(t)\|_{L^{q}}^{q}+\max_{1\leq i\leq \ell}\int_0^T \int_{\Tor^d}|v_i|^{q-2} |\nabla v_i|^2\,\dd x\, \dd s
\leq C_T(1+
\|v_{0}\|_{L^{q}}^{q}) \ \ \text{ a.s.\ }
\end{equation}
where $C_T$ is independent of $(\theta,v_0)$.  We remark that    
$ \int_0^T   \int_{\Tor^d}|v_i|^{q-2} |\nabla v_i|^2\,\dd x\, \dd s<\infty$ a.s.\ due to \eqref{eq:mixed_estimate_q}. Therefore the term   
$q(q-1)\frac{3\ellip_0}{4} \sum_{1\leq i \leq \ell}
 \int_0^T   \int_{\Tor^d}|v_i|^{q-2} |\nabla v_i|^2\,\dd x\,\dd s$ obtained by summing the estimate of Step 1 can be absorbed on the LHS of the corresponding estimate.
 
To conclude the proof of Theorem \ref{t:global_cut_off}\eqref{it:global_cut_off_2}, it remains to show
\begin{equation}
\label{eq:L_2_energy_bounds}
\max_{1\leq i \leq \ell}\int_0^T \int_{\Tor^d} |\nabla v_i|^2\,\dd x\,\dd s\leq C_T( 1+
\|v_{0}\|_{L^{q}}^{q})\ \ \text{ a.s.\ }
\end{equation}
where $C_T$ is independent of $(\theta,v_0)$.
By Step 1 with $q=2$, it remains to show that  
\begin{equation}
\label{eq:L_2_estimate_RHS}
\max_{1\leq i\leq \ell}
\Big|\int_0^T\int_{\Tor^d}\phi_{R,r}(\cdot,v) \big[f_i(\cdot,v) v_i -F_i(\cdot,v)\cdot \nabla v_i\big]\,\dd x\, \dd s\Big|
\lesssim_T 1+\|v_0\|_{L^q}^q.
\end{equation}
To this end, recall that $q\geq 2$ and $0\leq \phi_{R,r}(\cdot,v)\leq 1$. Thus, by Assumption \ref{ass:f_polynomial_growth}\eqref{it:f_polynomial_growth_1}, for all $i\in \{1,\dots,\ell\}$,
\begin{align*}
\Big|
\int_0^t\int_{\Tor^d}\phi_{R,r}(\cdot,v)   f_i(\cdot,v)v_i\,\dd x\, \dd s\Big|
\lesssim_T 1+ 
\int_0^t\int_{\Tor^d}\phi_{R,r}(\cdot ,v) |v_i|^{q+h-1}\,\dd x\, \dd s
\lesssim_T 1+\|v_0\|_{L^q}^q
\end{align*}
where the last inequality follows from \eqref{eq:e_u_trick_estimate_R_reaction_diffusion_cut_off} and \eqref{eq:a_priori_estimates_cut_off_conclusion_1}. With a similar argument one can show $\max_{1\leq i\leq \ell}\int_0^T\int_{\Tor^d}\phi_{R,r}(\cdot,v) |F_i(\cdot,v)||\nabla v_i|\,\dd x\, \dd s\lesssim_T 1+\|v_0\|^q_{L^q}$.
Thus we have proved 
\eqref{eq:L_2_estimate_RHS}.
\end{proof}

\begin{proof}[Proof of Lemma \ref{l:interpolation_L_eta}]
As above, we break the proof into steps. Below we set $1/0\stackrel{{\rm def}}{=}\infty$.

\emph{Step 1: For all 
$
1<\psi <\frac{2(d+2)}{d},
$
there exist $\theta\in (0,\frac{d}{d+2})$, $r_1\in (2,\infty)$, $\zeta_1\in (1,2)$ and $\xi_1\in (2,\frac{2d}{d-2})$ such that 
\begin{equation}
\label{eq:interp_step_1_proof}
\|u_1\|_{L^{\psi}((0,t)\times \Tor^d)}\lesssim
\|u_1\|_{L^{r_1}(0,t;L^{\zeta_1})}^{1-\theta} \|u_1\|_{L^2(0,t;L^{\xi_1})}^{\theta}
\end{equation}
where the implicit constant is independent of $u_1\in C([0,T]\times \Tor^d)$ and $t\in (0,T]$}. 

By standard interpolation arguments, one sees that \eqref{eq:interp_step_1_proof} holds provided 
\begin{equation}
\label{eq:conditions_interpolation_step_1}
 \frac{1-\theta}{r_1}+\frac{\theta}{2}\leq \frac{1}{\psi}, \quad \text{ and } \quad 
 \frac{1-\theta}{\zeta_1}+\frac{\theta}{\xi_1}\leq \frac{1}{\psi}.
\end{equation}
Since $\psi<\frac{2(d+2)}{d}$ by assumption, the conditions in \eqref{eq:conditions_interpolation_step_1} hold with the \emph{strict} inequalities in case $(r_1,\zeta_1,\xi_1,\theta)$ are replaced by the corresponding extreme values $(\infty,2, \frac{d}{d-2},\frac{d}{d+2})$. By continuity, there exist $r_1<\infty$, $\zeta_1\in (1,2)$, $\theta<\frac{d}{d+2}$ and $\xi_1<\frac{d}{d-2}$ such that the conditions in \eqref{eq:conditions_interpolation_step_1} hold with the strict inequality. This concludes the proof of Step 1. 

\emph{Step 2: Conclusion}. 
By a standard approximation it is suffices to consider $u\in C^{1}([0,T]\times \Tor^d)$.
We begin by noticing that 
\begin{equation}
\label{eq:def_psi_transformation}
|u|^{q+h-1}=\Big[|u|^{q/2}\Big]^{\psi}
\qquad \text{ and }\qquad \psi\stackrel{{\rm def}}{=}\frac{2}{q}(q+h-1)<\frac{2(d+2)}{d}.
\end{equation} 
Here the last inequality follows from the assumption $q>\frac{d}{2}(h-1)$. Applying Step 1 to $u_1=|u|^{q/2}\in C^1([0,T]\times \Tor^d)$ and $\psi$ as above, we have
\begin{equation}
\label{eq:int_1_u_proof}
\begin{aligned}
\int_0^t \int_{\Tor^d} |u|^{q+h-1}\,\dd x\, \dd s
&\lesssim 
\Big\||u|^{q/2}	\Big\|_{L^{r_1}(0,t;L^{\zeta_1})}^{\psi(1-\theta)} \Big\||u|^{q/2} \Big\|_{L^2(0,t;L^{\xi_1})}^{\psi\theta}\\
&\stackrel{(i)}{\lesssim}
\|u\|_{L^{r_*}(0,t;L^{q})}^{\frac{q}{2}\psi(1-\theta)} \Big\||u|^{q/2} \Big\|_{L^2(0,t;L^{\xi_1})}^{\psi\theta}
\end{aligned}
\end{equation}
where $(i)$ we set $r_*\stackrel{{\rm def}}{=}\frac{qr_1}{2}\in (2,\infty)$ and used that $\zeta_1\frac{q}{2}\leq q$.
Since $\xi_1<\frac{2d}{d-2}$, the Sobolev embedding $H^1(\Tor^d)\embed L^{\xi_1}(\Tor^d)$ yields
\begin{equation}
\label{eq:int_2_u_proof_intermediate}
\begin{aligned}
\Big\||u|^{q/2}\Big\|_{L^2(0,t;L^{\xi_1})}
&\lesssim
\Big\||u|^{q/2}\Big\|_{L^2(0,t;L^{2})} + 
\Big\|\nabla [|u|^{q/2}]\Big\|_{L^2(0,t;L^{2})}\\
&=
\|u\|_{L^{q}(0,t;L^{q})}^{q/2} + 
\Big(\int_{0}^t \int_{\Tor^d} |u|^{q-2}|\nabla u|^2 \,\dd x\, \dd s\Big)^{1/2}
\end{aligned}
\end{equation}
Since $r_*>q$, the previous yields
\begin{equation}
\label{eq:int_2_u_proof}
\Big\||u|^{q/2}\Big\|_{L^2(0,t;L^{\xi_1})}
\lesssim_T
\|u\|_{L^{r_*}(0,t;L^{q})}^{q/2} + 
\Big(\int_{0}^t \int_{\Tor^d} |u|^{q-2}|\nabla u|^2 \,\dd x\, \dd s\Big)^{1/2}.
\end{equation}
Inserting \eqref{eq:int_2_u_proof} in \eqref{eq:int_1_u_proof}, one sees that the estimate \eqref{eq:interp_inequality_II} follows with $\alpha\stackrel{{\rm def}}{=}\frac{q}{2}\psi(1-\theta)$ and $\beta\stackrel{{\rm def}}{=}\frac{\psi\theta}{2}$.
To conclude the proof of Lemma \ref{l:interpolation_L_eta}, it remains to show $\beta<1$, i.e.\
$
\psi \theta<2
$.
To see the latter, recall that $\theta<\frac{d}{d+2}$ and observe that
$$
\psi \frac{d}{d+2}=\Big[\frac{2}{q}(q-1+h)\Big]\frac{d}{d+2}<2  \quad \Longleftrightarrow \quad q>\frac{d}{2}(h-1)
$$
which holds by assumption. This completes the proof of Lemma \ref{l:interpolation_L_eta}.
\end{proof}

From the proof of Lemma \ref{l:interpolation_L_eta} we can extract the following result which will be used later on.

\begin{remark}[Interpolation inequality II]
\label{r:interpolation_L_eta_2}
Assume $h>1$ and $\frac{d(h-1)}{2}\vee 2<q<\infty$. Then there exist $\alpha,\beta>0$, $r\in (2,\infty)$ and $\xi\in [q,\frac{dq}{d-2})$ such that, for all $t\in(0,T]$ and $u\in L^{r}(0,t;L^q)\cap L^q(0,t;L^{\xi})$,
\begin{equation*}
\| u\|_{L^{q+h-1}((0,t)\times \Tor^d)}
\lesssim_T
\|u\|_{L^{r}(0,t;L^{q})}^{\alpha}
\|u\|_{L^{q}(0,t;L^{\xi})}^{\beta}.
\end{equation*}
The above inequality readily  follows 
from Step 1 of Lemma \ref{l:interpolation_L_eta} and \eqref{eq:def_psi_transformation}, cf.\
\eqref{eq:int_1_u_proof}.
\end{remark}

\section{Deterministic reaction-diffusion equations with high diffusivity}
\label{s:deterministic_high_diffusion}
In this section we investigate \emph{deterministic} reaction-diffusion equations:
\begin{equation}
\label{eq:reaction_diffusion_deterministic}
\left\{
\begin{aligned}
\partial_t v_i &=\mu_i\Delta v_i +\big[ \div(F_i(\cdot,v)) + f_i(\cdot,v)\big],  &\text{ on }&\Tor^d,\\
v_i(0)&=v_{0,i},& \text{ on }&\Tor^d,
\end{aligned}\right.
\end{equation}
where $i\in \{1,\dots,\ell\}$ for some integer $\ell\geq 1$, $\mu_i>0$ and $(F,f)$ are as in Assumption \ref{ass:f_polynomial_growth}.
The results of this section will be used in combination with Theorem \ref{t:weak_convergence} below to prove the results stated in Subsection \ref{ss:main_results}.
This section is organized as follows. In Subsection \ref{ss:high_reaction_rate} we show the existence global unique solutions to \eqref{eq:reaction_diffusion_deterministic} provided the diffusivities $\mu_i$ are sufficiently large. Finally, in Subsection \ref{ss:uniqueness} we prove an uniqueness result for a class of weak solutions to \eqref{eq:reaction_diffusion_deterministic} which naturally appears when dealing with certain compactness arguments, see the proof of  Theorem \ref{t:weak_convergence}.

\subsection{Reaction-diffusion equations with high diffusivity}
\label{ss:high_reaction_rate}
Here we show the existence on large time intervals of solutions to \eqref{eq:reaction_diffusion_deterministic} with $\mu_i \gg0$. 
Recall that $(p,q)$--solutions to \eqref{eq:reaction_diffusion_deterministic} have been defined below \eqref{eq:reaction_diffusion_det_statements} and that $(p,q)$--solutions are unique by definition.

\begin{proposition}
\label{prop:global_high_viscosity}
Suppose that $(F,f)$ satisfy Assumption \ref{ass:f_polynomial_growth}\eqref{it:f_polynomial_growth_1}--\eqref{it:mass}. 
Fix $N\geq 1$ and $T\in (0,\infty)$. Let $
\frac{d(h-1)}{2}\vee 2<q<\infty$ and $p\in [q,\infty)$. Set $\a_{p}\stackrel{{\rm def}}{=}\frac{p}{2}-1$.  Suppose that 
\begin{equation}
\label{eq:data_v_0_high_diffusion}
v_0\in L^q(\Tor^d;\R^{\ell}) \ \ \text{ satisfies }\ \ v_0\geq 0 \text{ on }\Tor^d \ \text{ and }\ \|v_0\|_{L^q}\leq N.
\end{equation}
Then there exists 
$
\mu_0(N,q,p,d,T,h,\m_i)>0
$ for which the following assertions hold provided
$$
\min_{1\leq i\leq \ell} \mu_i\geq \mu_0.
$$ 
\begin{enumerate}[{\rm(1)}]
\item\label{it:global_high_viscosity_1} 
There exists a $(p,q)$--solution $v$ to \eqref{eq:reaction_diffusion_deterministic} on $[0,T]$ satisfying
\begin{equation*}
v\in W^{1,p}(0,T,w_{\a_p};W^{-1,q}(\Tor^d;\R^{\ell}))\cap L^p(0,T,w_{\a_p};W^{1,q}(\Tor^d;\R^{\ell})).
\end{equation*}
\item\label{it:sol_operator_continuous} The solution mapping $v_0\mapsto v$ is continuous from $L^q(\Tor^d;\R^{\ell})$ into 
$$
W^{1,p}(0,T,w_{\a_p};W^{-1,q}(\Tor^d;\R^{\ell}))\cap L^p(0,T,w_{\a_p};W^{1,q}(\Tor^d;\R^{\ell})).
$$
\item\label{it:global_high_viscosity_2}
For some $C_0(T,N,q,p,d,h)>0$, the $(p,q)$--solution $v$ of \eqref{it:global_high_viscosity_1} satisfies
\begin{equation*}
\sup_{t\in (0,T]}\|v(t)\|^q_{L^q}+\max_{1\leq i\leq \ell}\int_0^T\int_{\Tor^d}|v_i|^{q-2}|\nabla v_i|^2\,\dd x\,\dd s
\leq C_0.
\end{equation*}
\end{enumerate} 
\end{proposition}

Note that $(\mu_0,C_0)$ are independent of $v_0$ satisfying the conditions in \eqref{eq:data_v_0_high_diffusion}.
Before going into the proof of the above result, we collect some observations. To apply $L^p(L^q)$-techniques, it is convenient to use that $v_0\in B^0_{q,p}$ since $L^q\stackrel{(p\geq q)}{\embed} B^0_{q,p}$. Moreover, \eqref{it:global_high_viscosity_1} and the trace embedding of anisotropic maps yield (see e.g.\ \cite[Theorem 3.4.8]{pruss2016moving} or \cite[Theorem 1.2]{ALV21})
\begin{equation}
\label{eq:trace_regularity_v}
v\in C([0,T];B^0_{q,p})\cap C((0,T];B^{1-2/p}_{q,p}).
\end{equation}
The previous and $p>2$ imply that $v(t)\in B^{1-2/p}_{q,p}\subseteq L^q$ for all $t>0$ (cf.\ \eqref{eq:elementary_embedding_Besov_in_Lq} for the inclusion). In particular, the term $
\sup_{t\in (0,T]}\|v(t)\|^q_{L^q}$ in \eqref{it:global_high_viscosity_2} is well-defined. However, since $ B^0_{q,p} \not\embed L^q$, it is a part of the proof of Proposition \ref{prop:global_high_viscosity} to show its finiteness for $t$ small. A similar remark holds for the second term estimated in \eqref{it:global_high_viscosity_2} since the argument in \eqref{eq:mixed_estimate_q} holds only on the interval $[s,T]$ with $s>0$.
Finally, as in \eqref{eq:composition_rules_sobolev_regularity_power}-\eqref{eq:conseguence_of_estimate}, Proposition \ref{prop:global_high_viscosity}\eqref{it:global_high_viscosity_2} and Sobolev embeddings yield
\begin{equation}
\label{eq:conseguence_of_estimate_deterministic}
\|v\|_{L^q(0,T;L^{\xi})}\leq c_0(T,N,q,p,d,h) \ \  \text{ where }  \ \ 
\xi=
\left\{
\begin{aligned}
&\in (2,\infty) &\text{ if }&d=2,\\
&\frac{qd}{d-2}&\text{ if }& d\geq 3.
\end{aligned}\right.
\end{equation}
In the following we need another interpolation inequality. For all $t\in \R_+$ and $u\in L^{\infty}(0,t;L^2(\Tor^d))\cap L^2(0,t;H^1(\Tor^d))$ such that $\int_{\Tor^d} u\,\dd x=0$ a.e.\ on $[0,t]$,  
\begin{equation}
\label{eq:interpolation_inequality_mean_zero}
\|u\|_{L^{2/\g}((0,t)\times \Tor^d)}
\lesssim
\|u\|_{L^{\infty}(0,t;L^2(\Tor^d))}^{1-\g} \|\nabla u\|_{ L^2(0,t;L^2(\Tor^d))}^{\g}\ \  \text{ where }\ \ 
\g =\frac{d}{d+2}
\end{equation}
and the implicit constant is independent of $(t,u)$.
The estimate \eqref{eq:interpolation_inequality_mean_zero} follows from the Poincar\'{e} inequality, interpolation and the Sobolev embedding $H^{\g}(\Tor^d)\embed L^{2/\gamma}(\Tor^d)$.

\begin{proof}[Proof of Proposition \ref{prop:global_high_viscosity}]
Through the proof, we fix $T\in (0,\infty)$ and $v_0\in L^q\subseteq B^0_{q,p}$.  To economize the notation, here we denote by $c_T$ a constant which may change from line to line and depends only on $(N,q,p,d,T,h,\m_i)$, where $(h,\m_i)$ is as in Assumption \ref{ass:f_polynomial_growth}. 
Next we collect some useful facts. By \cite[Theorem 1.2]{CriticalQuasilinear}, there exists a $(p,q)$--solution $(v,\tau)$ to \eqref{eq:reaction_diffusion_deterministic} such that 
\begin{equation}
\label{eq:W_1p_loc_tau_deterministic_regularity}
v\in W^{1,p}_{\loc}([0,\tau),w_{\a_p};W^{-1,q}(\Tor^d;\R^{\ell}))\cap L^p_{\loc}([0,\tau),w_{\a_p};W^{1,q}(\Tor^d;\R^{\ell})).
\end{equation}
Moreover, \cite{CriticalQuasilinear} also shows the existence of positive constants $(T_0(v_0),\varepsilon_0(v_0))$ for which the following holds:
For all $u_0\in L^q$ such that $\|v_0-u_0\|\leq \varepsilon_0$ there exists a $(p,q)$--solution $(u,\lambda)$ to \eqref{eq:reaction_diffusion_deterministic} with initial data $u_0$ satisfying $\lambda>T_0$ and
\begin{equation}
\label{eq:continuity_dependence_deterministic}
\|v-u\|_{W^{1,p}(0,T_0,w_{\a_p};W^{-1,q})\cap L^p(0,T_0,w_{\a_p};W^{1,q})}\lesssim \|v_0-u_0\|_{L^q}
\end{equation}
where the implicit constant is independent of $u_0$ (but depends on $v_0$).

Combining a linearization argument and the maximum principle for the heat equation, one can check that 
Assumption \ref{ass:f_polynomial_growth}\eqref{it:positivity} and $v_0\geq 0$ a.e.\ on $\Tor^d$ yield (see e.g.\ \cite{P10_survey} and \cite[Theorem 2.13]{AV22} for the conservative term $\div(F(\cdot,v))$)
\begin{equation}
\label{eq:positivity_deterministic}
v(t,x)\geq 0\ \  \text{ a.e.\ on }[0,\tau)\times \Tor^d.
\end{equation}
Arguing as for Theorem \ref{t:local}\eqref{it:local_1}, by Assumption \ref{ass:f_polynomial_growth}\eqref{it:mass} and \eqref{eq:positivity_deterministic}, we have, for all $t\in [0,\tau)$,
\begin{equation}
\label{eq:mass_conservation_estimate}
\int_{\Tor^d} |v(t,x)| \,\dd x\lesssim 
 \Big[ e^{\m_1 t }  \int_{\Tor^d} |v_0(x)|\,\dd x+\m_0\frac{e^{\m_1t}-1}{\m_1} \Big]
\leq e^{\m_1 t } N+\m_0\frac{e^{\m_1t}-1}{\m_1}.
\end{equation}

Below, w.l.o.g., we assume that $\mu_i\geq 1$ for all $i\in \{1,\dots,\ell\}$. Finally, we set
$$
\mu\stackrel{{\rm def}}{=}\min_{1\leq i\leq \ell} \mu_i.
$$
Now we break the proof into several steps. 
In Steps 1-3 we prove the estimate in Proposition  \ref{prop:global_high_viscosity}\eqref{it:global_high_viscosity_2} with $T$ replaced by $\tau\wedge T$. In Step 4 we prove that $\tau\geq T$ and therefore 
Proposition  \ref{prop:global_high_viscosity}\eqref{it:global_high_viscosity_1} and \eqref{it:global_high_viscosity_2} follows from Steps 1-4. Finally, in Step 5 we Proposition  \ref{prop:global_high_viscosity}\eqref{it:sol_operator_continuous}.
This will complete the proof of Proposition  \ref{prop:global_high_viscosity}.

\emph{Step 1: There exists $c_T>0$, independent of $(v,v_0,\mu_i)$, such that, for all $0\leq  t< \tau\wedge T$,}\begin{equation}
\begin{aligned}
\label{eq:Step_1_claim_high_reaction_rate}
\sup_{r\in (0,t]}\|v(r)\|_{L^q}^q
&+\mu \max_{1\leq i\leq \ell}\int_0^t \int_{\Tor^d} |v_i|^{q-2}|\nabla v_i|^2\,\dd x\, \dd r\\
&
\leq  c_T \Big(1+\|v_0\|_{L^q}^q+\|v\|_{L^{q+h-1}(0,t;L^{d+h-1})}^{q+h-1} \Big).
\end{aligned}
\end{equation}
\emph{Finally, $c_T$ can be chosen independently of $T$ if Assumption \ref{ass:f_polynomial_growth}\eqref{it:mass} holds with $\m_0=0$ and $\m_1<0$}. 

As we remarked below the statement of Proposition \ref{prop:global_high_viscosity}, the case $s=0$ of \eqref{eq:Step_1_claim_high_reaction_rate} is not immediate as $\sup_{s\in (0,t]}\|v(s)\|_{L^q}^q$ and $\int_{0}^t \int_{\Tor^d} |v_i|^{q-2}|\nabla v_i|^2\,\dd x\,\dd s$ might not be finite for small $t$ if $v$ is as in \eqref{eq:W_1p_loc_tau_deterministic_regularity}.
To overcome this difficulty, we use an approximation argument. To this end, we first prove \eqref{eq:Step_1_claim_high_reaction_rate} on an interval separated from $t=0$. Namely, for all $0<s<t<\tau\wedge T$, we claim that
\begin{equation}
\begin{aligned}
\label{eq:Step_1_claim_high_reaction_rate_s}
\sup_{r\in (s,t]}\|v(r)\|_{L^q}^q
&+\mu \max_{1\leq i\leq \ell}\int_s^t \int_{\Tor^d} |v_i|^{q-2}|\nabla v_i|^2\,\dd x\, \dd r\\
&
\leq  c_T \Big(1+\|v(s)\|^q_{L^q}+\|v\|_{L^{q+h-1}(s,t;L^{d+h-1})}^{q+h-1} \Big).
\end{aligned}
\end{equation}
Here it is important that $c_T$ on the RHS\eqref{eq:Step_1_claim_high_reaction_rate_s} does not depend on $s$ but only on $(N,q,p,d,T,h)$.

To see \eqref{eq:Step_1_claim_high_reaction_rate_s}, we can argue as follows. Firstly, as the weight $w_{\a}$ acts only at $t=0$, we have
\begin{equation}
\label{eq:regularity_det_high_diffusion_far_from_zero}
v\in  W^{1,p}_{\loc}([s,\tau);W^{-1,q})\cap L^p_{\loc}([s,\tau);W^{1,q}) \text{ for all }s>0.
\end{equation}
In particular, the terms on the LHS are finite, cf.\ \eqref{eq:mixed_estimate_q}. Similarly, one can also show that RHS\eqref{eq:mixed_estimate_q} is finite (see \eqref{eq:embedding_L_q_h_1_detereministic} below for the more involved weighted case). Now, since $q>2$, the proof of \eqref{eq:Step_1_claim_high_reaction_rate} for $s>0$ follows as in the proof of Theorem \ref{t:global_cut_off}\eqref{it:global_cut_off_2} in Subsection \ref{ss:uniform_estimate_cut_off} by computing $\partial_t \|v_i(t)\|_{L^q}^q$ for a fixed $i\in \{1,\dots,\ell\}$ and them summing up over $i$.  
Compared to Subsection \ref{ss:uniform_estimate_cut_off}, the term
$\int_0^t\int_{\Tor^d} |v_i|^{q-2}f_i(\cdot,v)v_i\,\dd x$ can be estimated as
$$
\Big|\int_0^t \int_{\Tor^d} |v_i|^{q-2}f_i(\cdot,v)v_i\,\dd x\Big|\stackrel{(i)}{\lesssim} 
\int_{\Tor^d} |v|^{q-1}(1+|v|^h)\,\dd x 
\stackrel{(ii)}{\lesssim} \|v\|_{L^1(0,t;L^1)}+\|v\|_{L^{q+h-1}(0,t;L^{q+h-1})}^{q+h-1}
$$
where in $(i)$ we used Assumption \ref{ass:f_polynomial_growth}\eqref{it:f_polynomial_growth_1} and in $(ii)$ $q>2$. The $F$--term can be estimate similarly also using the Cauchy-Schwartz inequality, cf.\ \eqref{eq:F_gradient_cauchy_schwartz_proof_theta_ind_estimate} for a similar situation. 

Now we would like to take the limit as $s\downarrow 0$ in \eqref{eq:Step_1_claim_high_reaction_rate_s}. To this end,  let us first prove that
\begin{equation}
\label{eq:L_q_h_1_deterministic}
v\in L^{q+h-1}_{\loc}([0,\tau);L^{q+h-1}) .
\end{equation}
In particular, the last term on RHS\eqref{eq:Step_1_claim_high_reaction_rate_s} is finite also if $s=0$. To prove \eqref{eq:L_q_h_1_deterministic}, due to \eqref{eq:W_1p_loc_tau_deterministic_regularity}, it suffices to show that,  for all $t\in (0,\infty)$,
\begin{equation}
\label{eq:embedding_L_q_h_1_detereministic}
W^{1,p}(0,t;w_{\a_{p}};W^{-1,q})\cap L^p(0,t;w_{\a_p};W^{1,q})\embed L^{q+h-1}(0,t;L^{q+h-1}).
\end{equation} 
By mixed-derivative embeddings (see e.g.\ \cite[Proposition 2.8]{AV19_QSEE_1}), we have that the space on LHS\eqref{eq:embedding_L_q_h_1_detereministic}
embeds into $H^{\theta,p}(0,T,w_{\a_p};H^{1-2\theta,q})$ for all $\theta\in (0,1)$. Thus, to prove \eqref{eq:embedding_L_q_h_1_detereministic}, it is suffices to show the existence of $\theta\in (0,1)$ such that
\begin{equation}
\label{eq:emb_H_theta_L_q_h_1_det}
H^{\theta,p}(0,T,w_{\a_p};H^{1-2\theta,q})\embed L^{q+h-1}(0,T;L^{q+h-1}).
\end{equation}
By Sobolev embeddings with power weights (see e.g.\ \cite[Proposition 2.7]{AV19_QSEE_1} or \cite[Corollary 1.4]{MV12}), the above holds provided we find $\theta\in (0,1)$ such that 
\begin{equation}
\label{eq:condition_theta_d_h_1_deterministic}
\theta-
\frac{1+\a_p}{p}=\theta-\frac{1}{2}> -\frac{1}{q+h-1}, \quad \text{ and }\quad 
1-2\theta-\frac{d}{q}\geq -\frac{d}{q+h-1}.
\end{equation}
In case $p>q+h-1$, then in the first condition in \eqref{eq:condition_theta_d_h_1_deterministic} the equality is \emph{not} allowed and in that case one also needs to use \cite[Proposition 2.1(3)]{AV19_QSEE_2} in combination with the above mentioned Sobolev embeddings with power weights.
To check \eqref{eq:condition_theta_d_h_1_deterministic}, we can argue as follows. The second condition in \eqref{eq:condition_theta_d_h_1_deterministic} is satisfied with $\theta=\frac{1}{2}\big[1-\frac{d(h-1)}{q(q+h-1)}\big]$. Note that $\theta\in (0,1)$ since $h>1$, $q\geq 2$ and $q>\frac{d(h-1)}{2}\geq \frac{d(h-1)}{q+h-1}$ by assumption.
With the above choice of $\theta$, one can check that also the first condition in \eqref{eq:condition_theta_d_h_1_deterministic} is satisfied since $q>\frac{d}{2}(h-1)$.
Therefore \eqref{eq:embedding_L_q_h_1_detereministic} holds.

It remains to show \eqref{eq:Step_1_claim_high_reaction_rate}. Due to \eqref{eq:Step_1_claim_high_reaction_rate_s}, to prove \eqref{eq:Step_1_claim_high_reaction_rate} it suffices to show \eqref{eq:Step_1_claim_high_reaction_rate} for $t\in (0,T_0]$ where $T_0>0$ is as before \eqref{eq:continuity_dependence_deterministic} and $c_T$ independent of $T_0$.
The advantage is that, for $t\in (0,T_0]$, we can use the continuous dependence on the initial data \eqref{eq:continuity_dependence_deterministic} and prove the claimed estimate by approximation.

Let $(\varepsilon_0,T_0)$ be as before \eqref{eq:continuity_dependence_deterministic}.
Take a sequence $(v_0^{(n)})_{n\geq 1}\subseteq C^{\infty}$ such that $v_0^{(n)}\to v_0$ in $L^q$ and $\|v^{(n)}-v_0\|_{L^q}\leq \varepsilon_0$ for all $n\geq 1$.  Fix $r\in (q,\infty)$. As in  \eqref{eq:continuity_dependence_deterministic}, there exists a $(r,r)$--solution $(v^{(n)},\tau^{(n)})$ to \eqref{eq:reaction_diffusion_deterministic} with data $v_0^{(n)}$ satisfying $\inf_{n\geq 1}\tau^{(n)}\geq T_0$ and
$$
v^{(n)}\in W^{1,r}_{\loc}([0,\tau^{(n)});W^{-1,r})\cap L^{r}_{\loc}([0,\tau^{(n)});W^{1,r}) \embed C([0,\tau^{(n)});B^{1-\frac{2}{r}}_{r,r})
\ \ \text{ for all }n\geq 1.
$$
Since $v^{(n)}$ is sufficiently smooth and $B^{1-2/r}_{r,r}\embed L^q$ for $r>q$, \eqref{eq:Step_1_claim_high_reaction_rate} with $(v,\tau)$ replaced by $(v^{(n)},\tau^{(n)})$ can be proven by letting $s\downarrow 0$ in \eqref{eq:Step_1_claim_high_reaction_rate_s}. 
Moreover, by \eqref{eq:continuity_dependence_deterministic},
\begin{equation*}
v^{(n)}\to v \  \text{ in } \ 
W^{1,p}(0,T_0;w_{\a_{p}};W^{-1,q})\cap L^p(0,T_0;w_{\a_p};W^{1,q}).
\end{equation*}
In particular, there exists a (not relabeled) subsequence such that $(v^{(n)},\nabla v^{(n)})\to (v,\nabla v)$ a.e.\ on $[0,T_0]\times \Tor^d$ and, by \eqref{eq:embedding_L_q_h_1_detereministic} and \cite[Theorem 1.2]{ALV21}, 
\begin{equation*}
v^{(n)}\to v \ \text{ in } \ L^{q+h-1}(0,T_0;L^{q+h-1})\cap C((0,T_0];L^q).
\end{equation*}
Thus, by Fatou's lemma and the above considerations, \eqref{eq:Step_1_claim_high_reaction_rate} with $t\in (0,T_0]$ and $c_T$ independent of $T_0$ follows by letting $n\to \infty$ in the corresponding estimate for $v^{(n)}$ using also that $\inf_{n\geq 1} \tau^{(n)}\geq T_0$.

To prove the last assertion in Step 1, note that, if  $\m_0=0$ and $\m_1<0$, then \eqref{eq:mass_conservation_estimate} yields $\|v_i\|_{L^1(0,\tau;L^1)}\lesssim \|v_0\|_{L^1}$ and therefore all the constants in the starting estimate \eqref{eq:Step_1_claim_high_reaction_rate_s} can be chosen independently of $T$.

\emph{Step 2: Recall that $\g=\frac{d}{d+2}$, cf.\ \eqref{eq:interpolation_inequality_mean_zero}. Then, for all $0\leq t<  \tau\wedge T$,}
\begin{equation}
\label{eq:step_2_deterministic_proof}
\|v\|_{L^{q+h-1}(0,t;L^{q+h-1})}^q\leq c_T\Big(1+ \|v_0\|_{L^q}^q +
 \mu^{-\g} \|v\|_{L^{q+h-1}(0,t;L^{q+h-1})}^{q+h-1}\Big),
\end{equation}
\emph{Finally, $c_T$ can be chosen independently of $T$ if Assumption \ref{ass:f_polynomial_growth}\eqref{it:mass} holds with $\m_0=0$ and $\m_1<0$}. 

In this step we use the interpolation inequality \eqref{eq:interpolation_inequality_mean_zero} in a similar way as we did in the proof of Lemma \ref{l:interpolation_L_eta} with \eqref{eq:interp_step_1_proof}. However, here we need an explicit dependence on the diffusivity $\mu$ and therefore we use the homogeneous estimate \eqref{eq:interpolation_inequality_mean_zero}. 
Let us fix $i\in\{1,\dots,\ell\}$. Since $q>\frac{d(h-1)}{2}$, we have $q+h-1< \frac{q}{\g}$. By interpolation, we have, for all $0\leq t<\tau\wedge T$,
\begin{equation}
\label{eq:high_diffusion_estimate_q_h_1}
\|v_i\|_{L^{q+h-1}(0,t;L^{q+h-1})}^q \lesssim \|v_i\|_{L^{1}(0,t;L^1)}^q+  \|v_i\|_{L^{q/\g}(0,t;L^{q/\g})}^{q}
\stackrel{\eqref{eq:mass_conservation_estimate}}{\leq} c_T+ \|v_i\|_{L^{q/\g}(0,t;L^{q/\g})}^{q}.
\end{equation}
Next we estimate the second term on the RHS\eqref{eq:high_diffusion_estimate_q_h_1}.
Note that, for all $0\leq t<\tau\wedge T$,
\begin{align}
\label{eq:L_q_h_1_splitting}
\|v_i\|_{L^{q/\g}(0,t;L^{q/\g})}^q
&=
\Big\||v_i|^{q/2}\Big\|_{L^{2/\g}(0,t;L^{2/\g})}^2\\
\nonumber
&\leq   \underbrace{
\Big\||v_i|^{q/2}-\int_{\Tor^d} |v_i|^{q/2}\,\dd x \Big\|_{L^{2/\g}(0,t;L^{2/\g})}^2}_{I_{1,i}\stackrel{{\rm def}}{=}}+ 
\underbrace{\Big\| \int_{\Tor^d} |v_i|^{q/2}\,\dd x\Big\|_{L^{2/\g}(0,t)}^2}_{I_{2,i}\stackrel{{\rm def}}{=}}
\end{align}
Next we estimate $I_{1,i}$ and $I_{2,i}$, separately. We begin with $I_{1,i}$. By  \eqref{eq:interpolation_inequality_mean_zero} with $u=|v_i|^{q/2}$, 
\begin{align*}
I_{1,i}
&  \lesssim
\Big\||v_i|^{q/2}\Big\|_{L^{\infty}(0,t;L^2)}^{2(1-\g)}
\Big\|\nabla \big[|v_i|^{q/2}\big]\Big\|_{L^{2}(0,t;L^2)}^{2\g}\\
& = \mu^{-\g} \Big(\|v_i\|_{L^{\infty}(0,t;L^{q})}^{q}\Big)^{1-\g} 
\Big(\mu\int_0^t \int_{\Tor^d}|v_i|^{q-2} |\nabla v_i|^2\,\dd x\, \dd s\Big)^{\g}\\
& \stackrel{\eqref{eq:Step_1_claim_high_reaction_rate}}{\lesssim_T} 
\mu^{-\g} \big(1+\|v_0\|^q_{L^q}+\|v\|_{L^{d+h-1}(0,t;L^{q+h-1})}^{q+h-1}  \big)\\
& \stackrel{(i)}{\lesssim} 1+\|v_0\|^q_{L^q} + \mu^{-\g}\|v\|_{L^{d+h-1}(0,t;L^{q+h-1})}^{q+h-1} 
\end{align*}
where in $(i)$ we used that $\mu\geq 1$. Next we look at $I_{2,i}$. Recall that $q>2$ and let $\varphi(q,h,d)\in (0,1)$ be such that $1-\varphi+\frac{\varphi \g}{q}= \frac{2}{q}$. Again, by interpolation,
\begin{align*}
I_{i,2}
=\|v_i\|_{L^{q/\g}(0,t;L^{q/2})}^{q}
&\leq \| v_i\|_{L^{q/\g}(0,t;L^1)}^{q(1-\varphi)}\|v_i\|_{L^{q/\g}(0,t;L^{q/\g})}^{q\varphi}\\
&\stackrel{\eqref{eq:mass_conservation_estimate}}{\leq} 
c_T(1+\| v_{0}\|_{L^1}^q)^{(1-\varphi)}\|v_i\|_{L^{q/\g}(0,t;L^{q/\g})}^{q\varphi}\\
&\leq 
c_{T} (1+\|v_0\|_{L^{q}}^{q})
+ \frac{1}{2}\|v_i\|_{L^{q/\g}(0,t;L^{q/\g})}^q.
\end{align*}
Using the estimates for $I_{i,1}$ and $I_{i,2}$ in \eqref{eq:L_q_h_1_splitting}, we have, for all $0\leq t<\tau\wedge T$,
\begin{equation}
\label{eq:step_2_deterministic_proof_i} 
\|v_i\|_{L^{q/\g}(0,t;L^{q/\g})}^q
\leq c_{T} \big(1+\|v_0\|_{L^{q}}^{q}\big)+ c_T \mu^{-\g}\| v\|_{L^{q+h-1}(0,t;L^{q+h-1})}^{q+h-1}
\end{equation}
where we have absorbed the term $2^{-1}\|v_i\|_{L^{q/\g}(0,t;L^{q/\g})}^q$ appearing in the estimate of $I_{i,2}$ in the LHS of \eqref{eq:step_2_deterministic_proof_i}. This is possible since $\|v_i\|_{L^{q/\g}(0,t;L^{q/\g})}<\infty$ for all $0\leq t<\tau\wedge T$, as it follows from the estimates of $I_{1,i}$, the fact that $I_{2,i}\lesssim_T\sup_{r\in (0,t)}\|v(r)\|_{L^q}^q<\infty$ for $0\leq t<\tau\wedge T$, \eqref{eq:Step_1_claim_high_reaction_rate} and \eqref{eq:embedding_L_q_h_1_detereministic}. 

By using \eqref{eq:step_2_deterministic_proof_i}  in \eqref{eq:high_diffusion_estimate_q_h_1}, we obtain 
\begin{equation*}
\|v_i\|_{L^{q+h-1}(0,t;L^{q+h-1})}^q\leq c_T\Big(1+ \|v_0\|^q_{L^q} +
 \mu^{-\g} \|v\|_{L^{q+h-1}(0,t;L^{q+h-1})}^{q+h-1}\Big).
\end{equation*}
The claimed estimate \eqref{eq:step_2_deterministic_proof} follows by taking the sum over $i\in \{1,\dots,\ell\}$.

The last assertion in Step 2 follows by using that $c_T$ in \eqref{eq:Step_1_claim_high_reaction_rate} can be chosen to be independent of $T$ and,  in the estimate of $I_{2,i}$, that $v_i\in L^{q/\g}(0,\tau;L^1)$ by \eqref{eq:mass_conservation_estimate} with $\m_0=0$ and $\m_1<0$.

\emph{Step 3: Fix $N\geq 1$ and let $v_0$ be as in \eqref{eq:data_v_0_high_diffusion}. Then there exist $\mu_0,K_0>0$ depending only on $(N,q,p,d,T,h,\m_i)$ such that the $(p,q)$--solution $(v,\tau)$ to \eqref{eq:reaction_diffusion_deterministic}  satisfies}
\begin{equation}
\label{eq:claim_step_3_deterministic_bound}
 \|v\|_{L^{q+h-1}(0,\tau\wedge T;L^{q+h-1})}\leq K_0\ \ \text{\emph{ provided }} \ \  \mu\geq \mu_0. 
\end{equation}  
The proof requires some preparation. Recall that $\|v_0\|_{L^q}\leq N$ and $v\in L^{q+h-1}_{\loc}([0,\tau);L^{q+h-1})$ by \eqref{eq:data_v_0_high_diffusion} and \eqref{eq:embedding_L_q_h_1_detereministic} respectively. Thus the estimate of Step 2 implies:
\begin{equation}
\label{eq:psi_bound_ellip}
\psi_{\mu,R}\big(\|v\|_{L^{q+h-1}(0,t;L^{q+h-1})}^q\big)
\leq 1+N^q \ \ \  \text{ for all }t\in [0,\tau\wedge T),
\end{equation}
where $\psi_{\mu,R}(x)=R \,x- \mu^{-\g} x^{1+\theta}$ with $x\in [0,\infty)$ and $\theta\stackrel{{\rm def}}{=}\frac{h-1}{q}>0$,   $R\stackrel{{\rm def}}{=}c^{-1}_T$ are independent of $(t,v_0,\mu)$.
It is routine to check that $\psi_{\mu,R}$ has a unique maximum on $[0,\infty)$ and 
\begin{equation}
\label{eq:M_m_ellip}
m_{\mu,R}\stackrel{{\rm def}}{=}\argmax_{\R_+} \psi_{\mu,R}=\Big(\frac{R \mu^{\g}}{1+\theta}\Big)^{1/\theta} \quad \text{ and } \quad M_{\mu,R}\stackrel{{\rm def}}{=} \max_{\R_+} \psi_{\mu,R}=
\frac{R\theta}{1+\theta} \Big(\frac{R \mu^{\g}}{1+\theta}\Big)^{1/\theta}.
\end{equation}

The idea is to choose $\mu_0(R,\theta,d)$ large enough so that $M_{\mu_0,R}>1+ N^q$, cf.\ Figure \ref{fig:1}. This eventually leads to a contradiction with \eqref{eq:psi_bound_ellip}. To this end, let us begin by
noticing that $[0,\infty)\ni \mu\mapsto M_{\mu,R}$ is increasing. Thus there exists $\mu_0(N,q,p,d,T,h)>0$ such that 
\begin{equation}
\label{eq:choice_ellip_high}
\mu\geq \mu_0 \qquad \Longrightarrow \qquad M_{\mu,R}\geq 2+N^q .
\end{equation}

Now suppose that $\mu\geq \mu_0$. 
Assume by contradiction that 
\begin{equation}
\label{eq:assumption_proof_by_contradiction}
\sup_{t\in [0,\tau\wedge T)}\|v\|_{L^{q+h-1}(0,t;L^{q+h-1})}=\infty.
\end{equation}
Next, note that the mapping
$$
\x:[0,\tau\wedge T)\to [0,\infty)\ \  \text{ defined as }\ \  t\mapsto \x(t)\stackrel{{\rm def}}{=}\|v\|_{L^{q+h-1}(0,t;L^{q+h-1})}^{q}
$$ 
is continuous, non-decreasing and satisfies $\x(0)=0$. Thus \eqref{eq:assumption_proof_by_contradiction} implies the existence of $\tau_0>0$ such that $\x(\tau_0)=m_{\mu,R}<\infty$. Note that \eqref{eq:choice_ellip_high} imply $\psi_{\mu,R}(\x(\tau_0))=M_{\mu,R}>1+N^q$. This leads to a contradiction with \eqref{eq:psi_bound_ellip}. 
The same argument also yields 
\begin{equation*}
\x(t)\leq m_{\mu,R} \ \  \text{ for all } 0\leq t<\tau\wedge T,
\end{equation*} 
where $m_{\mu,R}$ is as in \eqref{eq:M_m_ellip}. Combining the above with \eqref{eq:psi_bound_ellip}, we get, for all $0\leq t<\tau\wedge T$,
\begin{align*}
R\, \x(t)\leq 1+ N^q + \mu^{-\g} (\x(t))^{\theta} \x(t)\stackrel{\eqref{eq:M_m_ellip}}{\leq} 
1+N^q +\frac{R\, \x(t)}{1+\theta}.
\end{align*}
Therefore $\x(t)\leq \frac{1+\theta}{R\theta}(1+N^q)$ for all $0\leq t<\tau\wedge T$.
Whence the estimate \eqref{eq:claim_step_3_deterministic_bound} with $K_0\stackrel{{\rm def}}{=}\frac{1+\theta}{R\theta}(1+N^q)$ follows from the definition of $\x(t)$ and the Fatou lemma.

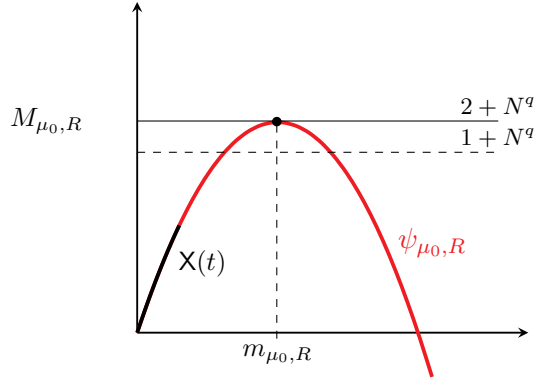
\begin{figure}
\begin{tikzpicture}[scale=0.8]
\draw[-stealth, line width=0.25mm] (0, 0) -- (6.5, 0);
\draw[-stealth, line width=0.25mm] (0, 0) -- (0, 5.5);
\draw[scale=0.7, domain=0:7, smooth, variable=\x, Red, line width=0.5mm] plot ({\x}, {\x*3-0.05*\x*\x*9});
\draw[scale=0.7, domain=0:1, smooth, variable=\x, black, line width=0.5mm] plot ({\x}, {\x*3-0.05*\x*\x*9}) node[right, xshift=-0.15cm, yshift=-0.5cm] {$\x(t)$};
\draw[dashed] (0,3) -- (6,3) node[above, yshift=-0.06cm]{\small$1+N^{q}$};
\draw (0,3.52) -- (6,3.52) node[left, xshift=-5.4cm]{$M_{\mu_0,R}$}  node[above, yshift=-0.06cm]{\small$2+N^{q}$};

\filldraw[black] (2.32,3.51) circle (2pt); 
\node[Red] at (4.9,1.5) {\large$\psi_{\mu_{0},R}$};
\draw[dashed] (2.32,-0.1) node[left, xshift=0.6cm, yshift=-0.2cm]{$m_{\mu_0,R}$} -- (2.32,3.52);
\end{tikzpicture}
\caption{Intuitive picture of the argument in Step 3 of Proposition \ref{prop:global_high_viscosity}.}
\label{fig:1}
\end{figure}

\emph{Step 4: Let $\mu_0$ be as in Step 3 and assume that $\mu\geq \mu_0$. Then $\tau\geq T$}. 
Combining the estimates of Steps 1 and 3 we have
\begin{equation}
\label{eq:estimate_steps_3_4_together}
\sup_{t\in (0,\tau\wedge T)}\|v_i(t)\|_{L^q}^q+\int_0^{\tau\wedge T} \int_{\Tor^d} |v_i|^{q-2}|\nabla v_i|^2\,\dd x\, \dd s
\leq C_0(T,N,q,p,d,h).
\end{equation}
To conclude the proof it remains to show that $\tau\geq T$. To this end, we apply the blow-up criterion of \cite[Corollary 2.3$(ii)$]{CriticalQuasilinear}, which ensures that 
\begin{equation}
\label{eq:blow_up_criterium}
\tau<T\qquad \Longrightarrow\qquad 
\sup_{t\in (0,\tau\wedge T)} \|v(t)\|_{B^0_{q,p}}=\infty.
\end{equation}
Here we also used that $B^0_{q,p}=(W^{-1,q},W^{1,q})_{1-\frac{1+\a_{p}}{p},p}$ and $\a_{p}=\frac{p}{2}-1$. 
Let us note that even if \cite{CriticalQuasilinear} deals with bilinear nonlinearities, the blow-up criterion of \cite[Corollary 2.3$(ii)$]{CriticalQuasilinear} still holds. Indeed, one can reproduce the argument in the proof of \cite[Theorem 4.10(3)]{AV19_QSEE_2} where  we recall that the weight $\a=\frac{p}{2}-1$ is allowed in case of deterministic equations.

We prove $\tau\geq T$ by contradiction. Assume that $\tau<T$. Then \eqref{eq:estimate_steps_3_4_together} and the embedding $L^q\embed B^0_{q,p}$ yield 
$
\sup_{t\in [0,\tau)}\|v(t)\|_{B^0_{q,p}}<\infty
$ 
which contradicts \eqref{eq:blow_up_criterium}. Hence $\tau\geq T$.

\emph{Step 5: \eqref{it:sol_operator_continuous} holds}. Recall that $\tau\geq T$ by Step 4. Let
\begin{equation}
\label{eq:T_star_def}
T_*\stackrel{{\rm def}}{=}\sup\Big\{t\in [0,T]\,:\, \lim_{u_0\to v_0}\|v-u\|_{\MRD(t)}=0\Big\},
\end{equation}
 where the limit is taken in the $L^q$-norm, $u$ is the solution of \eqref{eq:reaction_diffusion_deterministic} with initial data $u_0\in L^q$ and $\MRD(t)\stackrel{{\rm def}}{=}W^{1,p}(0,t,w_{\a_p};W^{-1,q})\cap L^p(0,t,w_{\a_p};W^{1,q})$ for all $t\in\R_+$. 
Note that $T_*>0$ by \eqref{eq:continuity_dependence_deterministic}. 

To prove \eqref{it:sol_operator_continuous} it is enough to show that  $T=T_*$ and that the supremum in \eqref{eq:T_star_def} is reached. To this end, one can argue by contradiction, we leave the details to the reader. In the argument it is convenient to use that $v([s,T])\subseteq B^{1-2/p}_{q,p}$ is compact for all $s>0$ by \eqref{eq:trace_regularity_v} and the local continuous dependence of solutions to \eqref{eq:reaction_diffusion_deterministic} on the initial data taken from the compact set $v([T_0,T])\subseteq B^{1-2/p}_{q,p}$ (see e.g.\ \cite[Theorem 1.2]{CriticalQuasilinear}).
\end{proof}

In the proof of Theorem \ref{t:delayed_blow_up_t_infty} we need uniform estimates on the half-line $(0,\infty)$. 
In case of exponentially decreasing mass, we obtain them by slightly modifying the proof of Proposition \ref{prop:global_high_viscosity}.

\begin{lemma}
\label{l:global_high_viscosity_infty}
Suppose that $(F,f)$ satisfy Assumption \ref{ass:f_polynomial_growth}\eqref{it:f_polynomial_growth_1}--\eqref{it:mass}. Assume that 
Assumption \ref{ass:f_polynomial_growth}\eqref{it:mass} holds with $\m_0=0$ and $\m_1<0$. 
Fix $N\geq 1$. Let $\frac{d(h-1)}{2}\vee2 < q<\infty$ and $p\in [q,\infty)$.
Then there exists 
$
\mu_0>0
$ 
such that if
$$
\min_{1\leq i\leq \ell} \mu_i\geq \mu_0,
$$ 
then the following assertion holds: 

For each $v_0\in L^q(\Tor^d;\R^{\ell})$ such that $v_0\geq 0$ a.e.\ on $\Tor^d$ and $\|v_0\|_{L^q}\leq N$, there exists a (unique) $(p,q)$--solution $v$ to \eqref{eq:reaction_diffusion_deterministic} on $[0,\infty)$ such that, for all $q_0\in [1,q)$,
\begin{equation}
\label{eq:decay_v_L_p_estimate_deterministic}
\|v(t)\|_{L^{q_0}}
\leq C(N,q,p,d,h,\alpha_i,\m_j)e^{-c_0 t} \ \ \text{ for all } t\geq 0,
\end{equation}
where $c_0>0$ depends only on $(\m_i,q_0,q)$. 
\end{lemma}

\begin{proof}
Since 
Assumption \ref{ass:f_polynomial_growth}\eqref{it:mass} holds with $\m_0=0$ and $\m_1<0$, by \eqref{eq:mass_conservation_estimate} we have
\begin{equation}
\label{eq:expentially_decreasing_mass}
\int_ {\Tor^d} |v|\,\dd x \lesssim e^{- |\m_1| t}\|v_0\|_{L^1}\leq e^{- |\m_1 | t} N.
\end{equation}

As in the proof of Proposition \ref{prop:global_high_viscosity} the existence of a $(p,q)$--solution $(v,\tau)$ to \eqref{eq:reaction_diffusion_deterministic} follows from \cite[Theorem 1.2]{CriticalQuasilinear}. It remains to prove $\tau=\infty$.
Arguing as in the Step 4 of Proposition \ref{prop:global_high_viscosity}, it is enough to show that, for some $\mu_0(N,q,p,d,h,\alpha_i,\m_j)>0$, one has
\begin{equation}
\label{eq:uniform_estimates_half_line}
\sup_{t\in [0,\tau)}\|v(t)\|^q_{L^{q}}\leq C(N,q,p,d,h,\alpha_i,\m_j).
\end{equation}
Indeed, if \eqref{eq:uniform_estimates_half_line} holds, then \eqref{eq:decay_v_L_p_estimate_deterministic} follows by interpolating \eqref{eq:expentially_decreasing_mass} and \eqref{eq:uniform_estimates_half_line}.

To prove \eqref{eq:uniform_estimates_half_line}, one can repeat the arguments in Step 3 of Proposition \ref{prop:global_high_viscosity}. Indeed, due to Step 2 of the same proof, the constant $c_T$ in \eqref{eq:step_2_deterministic_proof} can be made independent of $T$ since we are assuming $\m_0=0$ and $\m_1<0$.
\end{proof}


\subsection{Uniqueness for weak solutions to reaction-diffusion equations}
\label{ss:uniqueness}
In this subsection we prove uniqueness results for weak solutions to deterministic reaction-diffusion equations. 
Such results will be needed in the proof of Theorem \ref{t:delayed_blow_up}. In particular, the class of maps considered in the following result is the one used in Lemma \ref{l:compactness} below. 
We begin by proving the following uniqueness result for \eqref{eq:reaction_diffusion_deterministic}.

\begin{proposition}
\label{prop:uniqueness}
Let Assumption \ref{ass:f_polynomial_growth}\eqref{it:f_polynomial_growth_1}--\eqref{it:mass} be satisfied. 
Let 
$\frac{d(h-1)}{2}\vee 2<q<\infty$ and $v_0\in L^q(\R^d;\R^{\ell})$. Let either $\xi=\frac{dq}{d-2}$ and $d\geq 3$ or $\xi\in [\xi_0,\infty)$ for some sufficiently large $\xi_0(q,h,d)\in (q,\infty)$ and $d=2$. For $\g\in (0,1)$,
set
\begin{equation}
\label{eq:X_uniqueness}
\X\stackrel{{\rm def}}{=}L^2(0,T;H^{1-\g})\cap C([0,T];H^{-\g})\cap L^{\infty}(0,T;L^{q})\cap L^q(0,T;L^{\xi}).
\end{equation}
Let $
v^{(1)},v^{(2)}\in  \X
$
be weak solutions to \eqref{eq:reaction_diffusion_deterministic} in the following sense: 

For all $j\in \{1,2\}$, $\eta\in C^{\infty}(\Tor^d;\R^{\ell})$ and $t\in [0,T]$,
\begin{equation}
\label{eq:weak_formulation_deterministic_equation}
\begin{aligned}
\l v^{(j)}(t), \eta\r 
&= \int_{\Tor^d} v_{0}\cdot \eta\,\dd x \\
& +\sum_{1\leq i\leq \ell }\int_0^t\int_{\Tor^d}\Big(\mu_i\, v_i^{(j)} \Delta\eta_i +  f_i(\cdot,v^{(j)})\eta_i -F_i(\cdot,v^{(j)})\cdot \nabla \eta_i  \Big)\,\dd x\, \dd s .
\end{aligned}
\end{equation}
Then $v^{(1)}\equiv v^{(2)}$.
\end{proposition}

In \eqref{eq:weak_formulation_deterministic_equation}, $\l \cdot,\cdot\r$ denotes the pairing in the duality $H^{-\g}\times H^{\g}$.
Step 1 in the proof below shows that $f(\cdot,v^{(j)}), F(\cdot,v^{(j)})\in L^1(0,T;L^1)$ for $j\in \{1,2\}$. Thus all the terms on the RHS\eqref{eq:weak_formulation_deterministic_equation} are well-defined.
In the case $d=2$, the proof below provides a description of $\xi_0$. More precisely, $\xi_0=\xi_*\vee \xi_{**}$ where $\xi_{*}$ and $ \xi_{**}$ are as in Step 1  in the proof below and \eqref{eq:case_2_d_X_embedding}, respectively. 

The result of Proposition \ref{prop:uniqueness} is not really surprising since $\X\subseteq L^{\infty}(0,T;L^{q})\cap L^q(0,T;L^{\xi})$ and therefore 
the class of solutions considered there are somehow close to the strong ones. 

It will prove convenient later to see that $(p,q)$--solutions of Proposition \ref{prop:global_high_viscosity} belongs to $\X$. In particular, they are in the class of weak solutions considered in Proposition \ref{prop:uniqueness}.

\begin{remark}
\label{r:regularity_class_uniqueness}
Here we prove that the $(p,q)$--solutions to \eqref{eq:reaction_diffusion_deterministic} provided by Proposition \ref{prop:global_high_viscosity} satisfies $v\in \X$ where $\X$ is as in \eqref{eq:X_uniqueness}. Fix $T\in (0,\infty)$, $q,p\in (2,\infty)$ and let $\a_{p}=\frac{p}{2}-1$. By Proposition \ref{prop:global_high_viscosity}\eqref{it:global_high_viscosity_2} and \eqref{eq:conseguence_of_estimate_deterministic},  it suffices to show that
\begin{multline}
\label{eq:MR_embed_X}
\MRD(T)\stackrel{{\rm def}}{=}W^{1,p}(0,T,w_{\a_p};W^{-1,q})\cap L^p(0,T,w_{\a_p};W^{1,q})\\
\embed 
L^2(0,T;H^{1-\g})\cap C([0,T];H^{-\g})  \ \text{ for all }\g\in (0,1).
\end{multline} 
By mixed-derivative embeddings (see e.g.\ \cite[Proposition 2.8]{AV19_QSEE_1}), for all $\theta\in (0,1)$,
\begin{equation}
\label{eq:mixed_derivative_W_1_p_deterministic}
\MRD(T)\embed H^{\theta,p}(0,T,w_{\a_p};H^{1-2\theta,q}).
\end{equation}
Letting $\theta=\frac{1+\g}{2}$, then the RHS\eqref{eq:mixed_derivative_W_1_p_deterministic} coincides with $H^{\frac{1+\g}{2},p}(0,T;w_{\a_p};H^{-\g,q})\embed C([0,T];H^{-\g,q})$. While, letting $\theta=\frac{\g}{2}$ in the RHS\eqref{eq:mixed_derivative_W_1_p_deterministic} we have, for some $p_0\in (p,\infty)$,
\begin{align*}
H^{\g/2,p} (0,T,w_{\a_p};H^{1-\g,q})
\stackrel{(i)}{\embed}
L^{p_0} (0,T,w_{\a_p};H^{1-\g,q})
\stackrel{(ii)}{\embed}
L^{2} (0,T;H^{1-\g,q})
\end{align*}
where in $(i)$ we used Sobolev embeddings \cite[Proposition 2.7]{AV19_QSEE_1} and in $(ii)$ follows from the H\"{o}lder inequality and $\frac{1+\a_p}{p_0}<\frac{1+\a_p}{p}=\frac{1}{2}$ (see e.g.\ \cite[Proposition 2.1(3)]{AV19_QSEE_2}).
Thus \eqref{eq:MR_embed_X} follows by collecting the previous embeddings as well as by $H^{s,q}\embed H^s$ since $q\geq 2$.
\end{remark}

\begin{proof}[Proof of Proposition \ref{prop:uniqueness}]
In the following proof, for $ j\in \{1,2\}$, $v^{(j)}$ denotes a map from $\X$ (see \eqref{eq:X_uniqueness}) and solves \eqref{eq:weak_formulation_deterministic_equation} for all $\eta\in C^{\infty}(\Tor^d;\R^{\ell})$.
We break the proof into several steps.

\emph{Step 1: Let either $\xi=\frac{dq}{d-2}$ and $d\geq 3$ or $\xi\in [\xi_*,\infty)$ for some sufficiently large $\xi_*(q,h,d)\in (q,\infty)$ and $d=2$ (see the comments at the end of sub-steps 2a and 2b). Then }
\begin{equation}
\label{eq:L_2_integrability_weak_solution_f_F}
f(\cdot,v^{(j)})\in L^2(0,T;H^{-1}) \qquad \text{ and }\qquad F_i(\cdot,v^{(j)})\in L^2(0,T;L^2),
\end{equation}
\emph{for all  $i\in \{1,\dots,\ell\}$. In particular
$v^{(j)}\in L^2(0,T;H^{1})\cap H^{1}(0,T;H^{-1})\subseteq C([0,T];L^2).$}
The last claim of Step 1 follows from \eqref{eq:L_2_integrability_weak_solution_f_F} and the uniqueness of the heat equation in the $L^2$--setting. 

\emph{Sub-step 1a: $f(\cdot,v)\in L^2(0,T;H^{-1})$ for all $v\in \X$}. 
Here we mainly consider $d\geq 3$. We provide some comments at the end of this sub-step for the case $d=2$.  Recall that $L^{\zeta}\embed H^{-1}$ where $\zeta=\frac{2d}{d+2}$ (note that $\zeta>1$ since $d\geq 3$). Set $h_0\stackrel{{\rm def}}{=} 1+\frac{2q}{d}$ and note that $q=\frac{d}{2}(h_0-1)$ as well as $h_0\geq h\vee (1+\frac{4}{d})$.  By Assumption \ref{ass:f_polynomial_growth}\eqref{it:f_polynomial_growth_1},
\begin{equation}
\label{eq:f_estimate_det_zeta}
\|f(\cdot,v)\|_{ L^2(0,T;H^{-1})} \lesssim 
\|f(\cdot,v)\|_{ L^2(0,T;L^{\zeta})}\lesssim
1+\|v\|_{L^{2h_0}(0,T;L^{h_0\zeta})}^{h_0}.
\end{equation}
It remains to check that 
\begin{equation}
\label{eq:embedding_X_L_2h_substep_1a}
L^{\infty}(0,T;L^q)\cap L^q(0,T;L^{\frac{dq}{d-2}})\embed L^{2h_0}(0,T;L^{h_0\zeta}).
\end{equation}
Without loss of generality we assume $q<2h_0$, otherwise if $q\geq 2h_0$, then the above embedding follows from $L^q(0,T;L^{\frac{dq}{d-2}})\embed L^{2h_0}(0,T;L^{h_0\zeta})$ as $\frac{qd}{d-2}\geq 2h_0 > \frac{2dh_0}{d+2}= \zeta h_0$.
Thus, assuming that $q<2h_0$, by standard interpolation theory, \eqref{eq:embedding_X_L_2h_substep_1a} holds provided, for some $\varphi\in (0,1)$,
\begin{equation}
\label{eq:varphi_condition_f_v_X}
\frac{\varphi}{q} \leq \frac{1}{2h_0}
\quad \text{ and }\quad 
\frac{1-\varphi}{q}+\frac{\varphi(d-2)}{qd} \leq \frac{1}{h_0\zeta}.
\end{equation}
The first inequality in \eqref{eq:varphi_condition_f_v_X} is verified for $\varphi=\frac{q}{2h_0}\in (0,1)$. With the latter choice, one can readily check that the second inequality in \eqref{eq:varphi_condition_f_v_X} is equivalent to $q\geq \frac{2dh_0}{d+4}$. The latter condition holds with the strict inequality as $q=\frac{d}{2}(h_0-1)$ and $h_0> 1+\frac{4}{d}$.

If $d=2$, then the above argument works similarly. However, we have to choose $\zeta\in (1,2)$ for the embedding $L^{\zeta}\embed H^{-1}$ in \eqref{eq:f_estimate_det_zeta} as the sharp case $\zeta=1$ is not true in general.  Indeed if $q\geq 2h_0$, then one can choose $\xi_*\geq 2h_0$. While,  if $q<2h_0$, then one can choose $\xi_*\in (1,\infty)$ large and $\zeta\in (1,\infty)$ small such that 
\begin{equation}
\label{eq:choice_xi_h_0_d_2}
\frac{1-\varphi}{q}+\frac{\varphi}{\xi_*} < \frac{1}{h_0\zeta} \ \ \text{ where } \ \ \varphi=\frac{q}{2h_0}.
\end{equation}
To see that such a choice is possible, one can argue as follows. By a continuity argument, it is enough to check \eqref{eq:choice_xi_h_0_d_2} with $(\xi_*,\zeta)$ replaced by its $(\infty,1)$. The first in \eqref{eq:choice_xi_h_0_d_2} is equivalent to $q>\frac{2h_0}{3}$ which is satisfied since $d=2$, $q=h_0-1$ and $h_0>3$ by construction.

\emph{Substep 1b: $F(\cdot,v)\in L^2(0,T;L^2)$ for all $v\in \X$}. As in Substep 1a, we set $h_0=1+\frac{2q}{d}$.
By Assumption \ref{ass:f_polynomial_growth}\eqref{it:f_polynomial_growth_1},
$$
\|F(\cdot,v)\|_{L^2(0,T;L^2)}\lesssim 1+ \|v\|_{L^{h_0+1}(0,T;L^{h_0+1})}^{(h_0+1)/2}.
$$
As above, we first consider the case $d\geq 3$. 
 Thus, it remains to check that
\begin{equation}
\label{eq:embedding_X_L_2h_substep_1b}
L^{\infty}(0,T;L^q)\cap L^q(0,T;L^{\frac{dq}{d-2}})\embed L^{h_0+1}(0,T;L^{h_0+1}).
\end{equation} 
Without loss of generality, we assume that $q<h_0+1$. Indeed, if $q\geq h_0+1$, then the above embedding follows from $L^{q}(0,T;L^{\frac{dq}{d-2}})\embed L^{h_0+1}(0,T;L^{h_0+1})$ as $\frac{dq}{d-2}>q\geq h_0+1$. Next we consider $q<h_0+1$. In this case, by interpolation, \eqref{eq:embedding_X_L_2h_substep_1b} follows provided 
\begin{equation}
\label{eq:varphi_condition_F_v_X}
\frac{\varphi}{q} \leq \frac{1}{h_0+1}
\quad \text{ and }\quad 
\frac{1-\varphi}{q}+\frac{\varphi(d-2)}{qd} \leq \frac{1}{h_0+1}.
\end{equation}
The first inequality in \eqref{eq:varphi_condition_F_v_X} is verified for $\varphi=\frac{q}{h_0+1}\in (0,1)$. With the latter choice, one can readily check that the second inequality  in \eqref{eq:varphi_condition_F_v_X} is equivalent to $q\geq \frac{d(h_0+1)}{d+2}$. As above, the latter condition is satisfied with the strict inequality since $q=\frac{d}{2}(h_0-1)$ and $h_0> 1+\frac{4}{d}$. The case $d=2$ works in the same way as in Substep 1a. We omit the details. 

Before going into the next step we collect some facts. Step 1 shows that $v^{(j)}$ solves \eqref{eq:reaction_diffusion_deterministic_cut_off} in its differential form where the equality is understood in $H^{-1}$.
For exposition convenience, in Step 2 we prove the claim of Proposition \ref{prop:uniqueness} assuming that (in case $d=2$ we choose $\xi_0$ large enough)
\begin{equation}
\label{eq:X_embedding_h_type_space}
\X\embed L^{\psi}((0,T)\times \Tor^d;\R^{\ell}) \ \  \text{ for some }  \ \ \psi> (h-1)\Big(1+\frac{d}{2}\Big),
\end{equation}
where $\psi$ depends only on $(h,d,q)$. Step 3 is devoted to the proof of \eqref{eq:X_embedding_h_type_space}. 

\emph{Step 2: $v^{(1)}\equiv v^{(2)}$}. 
By a standard iteration argument, to prove the claim of Step 2 it suffices to show the existence of $\delta_*>0$ such that, for all $s\in [0,T]$, 
\begin{equation}
\label{eq:induction_uniqueness}
v^{(1)}(s)=
v^{(2)}(s)\text{ a.e.\ on }\Tor^d 
 \quad \Longrightarrow \quad 
v^{(1)}=
v^{(2)}
\text{ a.e.\ on }[s,s+\delta_*]\times \Tor^d.
\end{equation} 
Note that the evaluation at $s$ in first condition of \eqref{eq:induction_uniqueness} is well defined since $v^{(j)}\in C([0,T];L^2)$ by Step 1.
The remaining part of this step is devoted to the proof of \eqref{eq:induction_uniqueness}. Let $\varepsilon>0$ be fixed later. The embedding \eqref{eq:X_embedding_h_type_space} and the H\"{o}lder inequality show the existence of $\delta(\varepsilon)>0$, independent of $s\in [0,T]$, such that 
\begin{equation}
\label{eq:choice_varepsilon_v_h}
\int_{s}^{s+\delta} \int_{\Tor^d} |v^{(j)}|^{(h-1)(1+\frac{d}{2})}\,\dd x\, \dd s\leq \varepsilon.
\end{equation}
Up to replace $\delta$ by $\delta\wedge \varepsilon$, we can assume that $\delta(\varepsilon)\leq \varepsilon$. Next we prove the existence of $\varepsilon_*>0$ such that \eqref{eq:induction_uniqueness} holds for $\delta_*=\delta(\varepsilon_*)$.

Fix $s\in [0,T]$ and assume that $v^{(1)}(s)=
v^{(2)}(s)$.  Let $\ellip_0\stackrel{{\rm def}}{=}\min_{1\leq i\leq \ell} \ellip_i$.
Recall that $v^{(j)}$ solves \eqref{eq:reaction_diffusion_deterministic_cut_off} in $H^{-1}$ on $[s,T]$, and that the claim of Step 1 holds. Computing $\frac{\dd}{\dd t}\|v^{(1)}-v^{(2)}\|_{L^2}^2$ one obtains, for all $t\in [s,T]$,
\begin{align*}
\|v^{(1)}(t)- v^{(2)}(t) \|_{L^2}^2
&+2\ellip_0\int_{s}^{t} \int_{\Tor^d}\Big|\nabla \big[v^{(1)}-v^{(2)}\big]\Big|^2\,\dd x\, \dd s\\
&\leq 2\underbrace{\int_{s}^{t}\int_{\Tor^d} (f(\cdot,v^{(1)})-f(\cdot,v^{(2)}))\cdot(v^{(1)}-v^{(2)})\,\dd x\, \dd s}_{ I_{f}(t)\stackrel{{\rm def}}{=}}\\
&+ 2\sum_{1\leq i\leq \ell}\underbrace{\int_{s}^{t}\int_{\Tor^d} (F_i(\cdot,v^{(1)})-F_i(\cdot,v^{(2)}))\cdot\nabla [v^{(1)}_i-v^{(2)}_i]\,\dd x\, \dd s}_{I_F(t)\stackrel{{\rm def}}{=}}.
\end{align*}
Next we estimate the terms $I_f$ and $I_F$ separately. 
We begin by considering $I_f$. Let $\g=\frac{d}{d+2}$.  By Assumption \ref{ass:f_polynomial_growth}\eqref{it:f_polynomial_growth_1} we have, for all $t\in [s,s+\delta]$,
\begin{align*} 
|I_f(t)|
&\lesssim  \int_{s}^t\int_{\Tor^d} \big(1+|v^{(1)}|^{h-1}+ |v^{(2)}|^{h-1}\big)\big|v^{(1)}-v^{(2)}\big|^2\,\dd x\,\dd s\\
&\stackrel{(i)}{\leq} \Big[t-s+ \|v^{(1)}\|_{L^{\frac{h-1}{1-\g}}(s,t;L^{\frac{h-1}{1-\g}} )}^{h-1} 
+\|v^{(2)}\|_{L^{\frac{h-1}{1-\g}}(s,t;L^{\frac{h-1}{1-\g}})}^{h-1} \Big] 
\|v^{(1)}-v^{(2)}\|_{L^{\frac{2}{\g}}(s,t; L^{\frac{2}{\g}})}^2\\	
&\stackrel{(ii)}{\lesssim} 
\big[\varepsilon+2\varepsilon^{h-1}\big] \|v^{(1)}-v^{(2)}\|_{L^{\frac{2}{\g}}(s,t; L^{\frac{2}{\g}})}^2\\
&\stackrel{(iii)}{\leq} 
c_{\varepsilon}
\|v^{(1)}-v^{(2)}\|_{L^{\infty}(s,t;L^{2})}^2
+ \frac{\ellip_0}{2} \Big\|\nabla [v^{(1)}-v^{(2)}] \Big\|_{L^2(s,t;L^2)}^2
\end{align*}
where in $(i)$ we used the H\"{o}lder inequality with exponents $(\frac{1}{1-\g},\frac{1}{\g})$  and in $(ii)$  we used \eqref{eq:choice_varepsilon_v_h}, $\frac{1}{1-\g}=1+\frac{d}{2}$ and $\delta\leq \varepsilon$ by construction. In $(iii)$ we used \eqref{eq:interpolation_inequality_mean_zero}, $c_{\varepsilon}$ depend only on $(\|f(\cdot,0)\|_{L^{\infty}},\ellip_0)$ and satisfies $\lim_{\varepsilon\downarrow 0} c_{\varepsilon}=0$. 

We estimate $I_F$ in a similar way. To begin, note that for all $t\in [s,s+\delta]$,
\begin{align*}
|I_F(t)|
&\leq \frac{\ellip_0}{2} \int_{s}^t \int_{\Tor^d} \big|\nabla [v^{(1)}-v^{(2)}]\big|^2\,\dd x\, \dd s\\
&+ C(\ellip_0) \sum_{1\leq i\leq \ell} 
\int_{s}^t \big| F_i(\cdot,v^{(1)})-F_i(\cdot,v^{(2)})\big|^2\,\dd x\, \dd s.
\end{align*}
Again, by Assumption \ref{ass:f_polynomial_growth}\eqref{it:f_polynomial_growth_1},
\begin{align*}
&\int_{s}^t\int_{\Tor^d} \big| F_i(\cdot,v^{(1)})-F_i(\cdot,v^{(2)})\big|^2\,\dd x\, \dd s \\
&\lesssim \int_{s}^t\int_{\Tor^d}\big(1+|v^{(1)}|^{h-1}+ |v^{(2)}|^{h-1}\big)\big|v^{(1)}-v^{(2)}\big|^2\,\dd x\, \dd s \\
&\leq c_{\varepsilon}
\|v^{(1)}-v^{(2)}\|_{L^{\infty}(s,t;L^{2})}^2
+ \frac{\ellip_0}{2} \Big\|\nabla [v^{(1)}-v^{(2)}] \Big\|_{L^2(s,t;L^2)}^2
\end{align*}
where the last inequality follows by noticing that the the second line in the above estimate coincides with the LHS in the first line in the estimate of $I_f$.

Using the above estimates, we get
\begin{align*}
\|v^{(1)}- v^{(2)} \|_{L^{\infty}(s,t;L^2)}^2
&+2\ellip_0\int_{s}^{t} \int_{\Tor^d}\Big|\nabla \big[v^{(1)}-v^{(2)}\big]\Big|^2\,\dd x\, \dd s\\ 
&\leq 
2c_{\varepsilon}
\|v^{(1)}-v^{(2)}\|_{L^{\infty}(s,t;L^{2})}^2
+ \frac{3}{2}\ellip_0 \Big\|\nabla [v^{(1)}-v^{(2)}] \Big\|_{L^2(s,t;L^2)}^2.
\end{align*}
By choosing $\varepsilon_*>0$ so that $c_{\varepsilon_*}<\frac{1}{2}$, the above yields \eqref{eq:induction_uniqueness} with $\delta_*=\delta(\varepsilon_*)$, as desired.

\emph{Step 3: \eqref{eq:X_embedding_h_type_space} holds}. 
As above we let $\g=\frac{d}{d+2}$. Here we use again an interpolation argument. Note that, by \eqref{eq:X_uniqueness}, 
\begin{equation}
\label{eq:X_embed_uniqueness_proof}
\X\embed L^{\infty}(0,T;L^q)\cap L^q(0,T;L^{\xi})\stackrel{(i)}{\embed} L^{q/\g}(0,T;L^{\eta}),
\end{equation}
where $\eta$ is uniquely determine by the relation 
$
\frac{1-\g}{q}+\frac{\g}{\xi}=\frac{1}{\eta}.
$
Note that $\frac{q}{\g}>(h-1)(1+\frac{d}{2})$ is equivalent to $q>\frac{d}{2}(h-1)$ which holds by assumption. It remains to prove $
\eta>(h-1)(1+\frac{d}{2}).
$
Since $\xi=\frac{dq}{d-2}$ in case $d\geq 3$, the previous follows again from $q>\frac{d}{2}(h-1)$. Finally we consider the case $d=2$. In the latter situation $\g=\frac{1}{2}$ and arguing as in \eqref{eq:X_embed_uniqueness_proof}, by interpolation, we have to choose $\xi_{**}\in (1,\infty)$ such that
\begin{equation}
\label{eq:case_2_d_X_embedding}
\frac{1}{2q}+\frac{1}{2\xi_{**}}<\frac{1}{2(h-1)}.
\end{equation} 
To see that \eqref{eq:case_2_d_X_embedding} is solvable, it is enough to let $\xi_{**}\to \infty$ and note that it reduces to $\frac{2}{q}<\frac{2}{h-1} $, i.e.\ $q>h-1=\frac{d}{2}(h-1)$. In particular, there exists $\xi_{**}(q,h)\in (1,\infty)$ for which  \eqref{eq:case_2_d_X_embedding} holds.
\end{proof}

As a by product of Proposition \ref{prop:uniqueness} we can establish a ``weak-strong'' uniqueness result for deterministic reaction-diffusion with cut-off: 
\begin{equation}
\label{eq:reaction_diffusion_deterministic_cut_off}
\left\{
\begin{aligned}
\partial_t v_i &=\mu_i\Delta v_i +\phi_{R,r}(\cdot,v)\big[\div(F_i(\cdot,v)) + f_i(\cdot,v)\big], &\text{ on }&\Tor^d,\\
v_i(0)&=v_{0,i}, &\text{ on }&\Tor^d,
\end{aligned}
\right.
\end{equation}
where $i\in \{1,\dots,\ell\}$. As before, here $\phi_{R,r}(\cdot,v)$ are as in \eqref{eq:def_cut_off} with $R,r\in (1,\infty)$.
The following result  will play a role in the scaling limit result of Theorem \ref{t:weak_convergence}. 

\begin{corollary}[Weak--strong uniqueness for \eqref{eq:weak_formulation_deterministic_equation}]
\label{cor:uniqueness}
Let Assumption \ref{ass:f_polynomial_growth}\eqref{it:f_polynomial_growth_1}--\eqref{it:mass} be satisfied.
Let $\frac{d(h-1)}{2}\vee 2<q<\infty$. Fix $R,r\in (1,\infty)$ and $v_0\in L^q(\Tor^d;\R^{\ell})$. Let $\xi$ and $\X$ be as in Proposition \ref{prop:uniqueness}. 
Assume that there exists a solution $v^{(1)}\in \X$ of \eqref{eq:reaction_diffusion_deterministic} in the weak formulation of \eqref{eq:weak_formulation_deterministic_equation} satisfying 
\begin{equation}
\label{eq:R_bound_v_1}
\|v^{(1)}\|_{L^r(0,T;L^q)}\leq R-1.
\end{equation}
Let $v^{(2)}\in  \X
$
be a weak solution to \eqref{eq:reaction_diffusion_deterministic_cut_off} in the following sense: 

For all $\eta\in C^{\infty}(\Tor^d;\R^{\ell})$ and $t\in [0,T]$,
\begin{equation*}
\begin{aligned}
&\l v^{(2)}(t), \eta\r = \int_{\Tor^d} v_{0}\cdot \eta\,\dd x \\
&\quad +\sum_{1\leq i\leq \ell} \int_0^t\int_{\Tor^d}\Big(\mu_i\, v^{(2)}_i\Delta\eta_i + \phi_{R,r}(\cdot,v^{(2)})\big[ f_i(\cdot,v^{(2)})\eta_i - F_i(\cdot,v^{(2)})\cdot \nabla \eta_i \big] \Big)\,\dd x\,\dd s .
\end{aligned}
\end{equation*}
Then $v^{(1)}\equiv v^{(2)}$.
\end{corollary}

Due to \eqref{eq:R_bound_v_1}, $v^{(1)}$ is also a weak solution to the problem \eqref{eq:reaction_diffusion_deterministic_cut_off} with cut-off. In the proof of Proposition \ref{prop:delayed_blow_up_smooth}, we check \eqref{eq:R_bound_v_1} by taking the (strong) $(p,q)$--solution to \eqref{eq:reaction_diffusion_deterministic}. Hence, to some extend, Corollary \ref{cor:uniqueness} shows that weak solutions coincide with strong ones to \eqref{eq:reaction_diffusion_deterministic_cut_off} (if there are any) and that \eqref{eq:R_bound_v_1} is a regularity assumption. 
This explains the name of Corollary \ref{cor:uniqueness}.

\begin{proof}
The idea is to reduce to the case analyzed in Proposition \ref{prop:uniqueness} by mimicking a stopping time argument.  To this end, let us set
$$
e\stackrel{{\rm def}}{=}\inf\big\{t\in [0,T]\,:\, \|v^{(2)}\|_{L^r(0,t;L^q)}\geq  R\big\}, \quad \text{ where } \quad \inf\emptyset\stackrel{{\rm def}}{=}T.
$$
It remains to prove that $e=T$. Indeed, if the latter holds,  then $\phi_{R,r}(\cdot,v^{(2)}) \equiv 1$ and therefore $v^{(2)}$ is also a weak solution to \eqref{eq:reaction_diffusion_deterministic} (i.e.\ it satisfies \eqref{eq:weak_formulation_deterministic_equation} for all $\eta\in C^{\infty}(\Tor^d;\R^{\ell})$). Hence, applying Proposition \ref{prop:uniqueness}, we eventually have $v^{(1)}\equiv v^{(2)}$.

We prove $e=T$ by contradiction. Assume that $e<T$. Then
\begin{equation}
\label{eq:R_bound_v_2}
\|v^{(2)}\|_{L^r(0,e;L^q)}= R \quad \Longrightarrow \quad \phi_{R,r}(s,v^{(2)})=1 \ \text{ for all }s\in [0,e].
\end{equation}
Therefore
$v^{(2)}|_{[0,e]}$ is a weak solution to \eqref{eq:reaction_diffusion_deterministic} in the sense of \eqref{eq:weak_formulation_deterministic_equation}. Hence, 
by Proposition \ref{prop:uniqueness},
\begin{equation}
\label{eq:equality_on_e_u}
v^{(1)}=v^{(2)} \text{ a.e.\ on }[0,e]\times \Tor^d
\end{equation}
Combining \eqref{eq:equality_on_e_u} and \eqref{eq:R_bound_v_1}, one has  $
\|v^{(2)}\|_{L^r(0,e;L^q)}\leq R-1$. This fact contradicts \eqref{eq:R_bound_v_2} and therefore $e=T$, as desired.
\end{proof}

\section{Proofs of Theorems \ref{t:delayed_blow_up} and \ref{t:delayed_blow_up_t_infty}}
\label{s:proofs}
In this section we prove Theorems \ref{t:delayed_blow_up} and \ref{t:delayed_blow_up_t_infty}.
To prove both results we can now argue as in \cite{FL19,FGL21}. In particular, as a central step we prove a scaling limit result for stochastic reaction-diffusion equations with cut-off \eqref{eq:reaction_diffusion_system_truncation}, see Subsection \ref{ss:scaling_limit_cut_off}.
Theorems \ref{t:delayed_blow_up} and \ref{t:delayed_blow_up_t_infty} will be proved in Subsections \ref{ss:proof_global_T} and \ref{ss:proof_global_infty}, respectively.


\subsection{The scaling limit for reaction-diffusion equations with cut-off}
\label{ss:scaling_limit_cut_off}
In this subsection we continue our investigation of reaction-diffusions with cut-off 
initiated in Section \ref{s:global_cut_off}.
Recall that the cut-off equation reads as follows:
\begin{equation}
\label{eq:reaction_diffusion_system_truncation_2}
\left\{
\begin{aligned}
\dd v_i -\ellip_i\Delta v_i \,\dd t&= \phi_{R,r} (\cdot,v)\Big[\div (F(\cdot,v))+f_{i}(\cdot, v)\Big]\,\dd t \\
&+ \sqrt{c_d\ellip} \sum_{k,\alpha} \theta_k (\sigma_{k,\alpha}\cdot \nabla) v_i\circ \dd w_t^{k,\alpha}, \qquad & \text{ on }&\Tor^d,\\
v_i(0)&=v_{i,0}, \qquad  & \text{ on }&\Tor^d,
\end{aligned}\right.
\end{equation}
where $\phi_{R,r}$ is as in \eqref{eq:def_cut_off} for $R>0$, $r\in [r_0,\infty)$  and $r_0$ is as in Theorem \ref{t:global_cut_off}. 
The aim of this subsection is to prove the following scaling limit result. It can be seen as a version of \cite[Theorem 1.4]{FL19} or
\cite[Proposition 3.7]{FGL21} in our setting 
and it is of independent interest. 

Recall that weak solutions to \eqref{eq:reaction_diffusion_system_truncation_2} are understood as in Corollary \ref{cor:uniqueness}.

\begin{theorem}[Scaling limit]
\label{t:weak_convergence}
Let Assumption \ref{ass:f_polynomial_growth} be satisfied.
Fix $T\in (0,\infty)$ and $v_0\in L^q(\Tor^d;\R^{\ell})$. Assume that $q>\frac{d(h-1)}{2}$. Let $\xi\in [\xi_0,\infty)$ and $r\in [r_0,\infty)$ where $\xi_0$ and $r_0$ are as in Proposition \ref{prop:uniqueness} and Theorem \ref{t:global_cut_off}, respectively. Suppose that the following hold.
\begin{enumerate}[{\rm(1)}]
\item\label{it:weak_convergence_L_q} Let $(v_{0}^{(n)})_{n\geq 1}$ be a sequence such that 
\begin{equation*}
v_0^{(n)}\in B^{1-2\frac{1+\a}{p}}_{q,p}(\Tor^d;\R^{\ell})\  \text{ for all } n\geq 1,
 \  \ \text{ and }\ \ \ 
v_0^{(n)} \rightharpoonup v_0 \text{ in }L^q(\Tor^d;\R^{\ell}).
\end{equation*}
\item\label{it:theta_goes_to_zero} Let $(\theta^{(n)})_{n\geq 1}\subseteq \ell^2(\Z^d_0)$ be a sequence such that  
$ \#\{k\,:\, \theta_k^{(n)}\neq 0\}<\infty $ and
\eqref{eq:theta_normalized_symmetric} with $\theta=\theta^{(n)}$ hold for all $n\geq 1$, and
$$
\lim_{n\to \infty}\|\theta^{(n)}\|_{\ell^{\infty}}=0.
$$
\item\label{it:determinstic_limit_cut_off} 
For some $\g\in (0,1)$, there exists a unique weak solution
$$
v=(v_i)_{i=1}^{\ell}
\in L^2(0,T;H^{1-\g})\cap C([0,T];H^{-\g})
\cap L^{\infty}(0,T;L^{q})
\cap L^{q}(0,T;L^{\xi}) 
$$  
to the following {\rm deterministic} system of reaction-diffusion equation with cut-off:
\begin{align*}
\left\{
\begin{aligned}
\partial_t v_i &=(\ellip_i + \ellip)\Delta v_i +  \phi_{R,r}(\cdot,v)\big[\div(F_i(\cdot,v))+ f_i (\cdot,v)\big] & \text{ on }&\Tor^d,\\ 
v_i(0)&=v_{0,i} & \text{ on }&\Tor^d.
\end{aligned}
\right.
\end{align*}
\end{enumerate}
Denote by $v^{(n)}$ the $(p,\a,1,q)$--strong solution to \eqref{eq:reaction_diffusion_system_truncation_2} with data $v_0^{(n)}$ (see Theorem \ref{t:global_cut_off}) and let $v$ be as in \eqref{it:determinstic_limit_cut_off}.
Then
\begin{equation}
\label{eq:claim_scaling_limit}
\lim_{n\to \infty}\P\big( \|v^{(n)}- v\|_{L^r(0,T;L^{q}(\Tor^d;\R^{\ell}))}>\varepsilon\big)=0 \ \  \text{ for all } \ \varepsilon>0.
\end{equation}  
\end{theorem}

Eq.\ \eqref{eq:claim_scaling_limit} shows the (weak) enhanced diffusive effect of the transport noise in \eqref{eq:reaction_diffusion_system_truncation_2}. Note that the increased diffusivity depends on the strength of the noise through the parameter $\ellip$. 
 The proof of
Theorem \ref{t:weak_convergence} actually gives a stronger result. More precisely, we show that \eqref{eq:claim_scaling_limit} also holds in case the $L^r(0,T;L^{q})$-norm is replaced by $L^2(0,T;H^{1-\g})\cap C([0,T];H^{-\g})
\cap L^r(0,T;L^q)$ where $\g>0$ is arbitrary (this is needed to obtain the assertions of Remark \ref{r:refined_enhanced_dissipation}).

The proof of Theorem \ref{t:weak_convergence} requires some preparation and it will be given at the end of this subsection. 
We begin with a compactness result. 

\begin{lemma}
\label{l:compactness}
Fix $T \in (0,\infty)$. Let  $\g_0,\g_1,\g\in (0,\infty)$, $q,r\in (1,\infty)$ and $\xi\in (q,\infty)$. Set
\begin{align*}
\Y&\stackrel{{\rm def}}{=} L^2(0,T;H^1) \cap L^{\infty}(0,T;L^{q})\cap  C^{\g_0}(0,T;H^{-\g_1})\cap L^q(0,T;L^{\xi}),\\
\X&\stackrel{{\rm def}}{=}L^2(0,T;H^{1-\g})\cap C([0,T];H^{-\g})\cap L^r(0,T;L^{q}).
\end{align*}
Then $\Y\embed \X$ compactly. Moreover, for any $K\in (0,\infty)$, the set 
\begin{align*}
\Big\{u\in \X\,
:\,
\sup_{t\in [0,T]}\|u(t)\|_{L^q}+
\|u\|_{L^q(0,T;L^{\xi})} \leq K \Big\}\  \text{ is closed in }\X.
\end{align*}
\end{lemma}

\begin{proof}
The proof is similar to the one of \cite[Lemma 3.3]{FGL21}. For the reader's convenience, we include some details. Below $(\g_0,\g_1,\g,r,\xi)$ are as in the statement of Lemma \ref{l:compactness}.
Firstly we show the compactness of the embedding $\Y\embed \X$. 
Let $(u_n)_{n\geq 1}$ be a sequence in $\Y$ such that $\|u_n\|_{\Y}\leq 1$. It remains to show that there exists a subsequence (not relabeled for simplicity) such that $u_n\to u$ in $\X$. 
To begin, note that, by Ascoli-Arzel\`{a}  theorem, there exists a (not relabeled) subsequence such that $u_n\to u$ in $C([0,T];H^{-\g_1-\varepsilon})$ for all $\varepsilon>0$. 
Next we show that $u_n\to u$ in $\X$. 
Combining the uniform bound of $(u_n)_{n\geq 1}$ in $L^{\infty}(0,T;L^2)\subseteq
L^{\infty}(0,T;L^{q})$, one has
\begin{equation}
\label{eq:L_infty_H_varepsilon}
u_n\to u \text{ in }\ C([0,T];H^{-\varepsilon}) \ \ \ \text{ for all } \ \varepsilon\in (0,1).
\end{equation}
Note that $\|g\|_{L^{2(1+ \varepsilon)/{\varepsilon}}(0,T;L^2)}
\lesssim \|g\|_{L^{\infty}(0,T;H^{-\varepsilon})}^{1/(1+ \varepsilon)}
\|g\|_{L^2(0,T;H^1)}^{\varepsilon/(1+ \varepsilon)}$ for all $\varepsilon\in (0,1)$ by interpolation. 
Whence, choosing $\varepsilon>0$ small, the above and the uniform bound in $L^2(0,T;H^1)$ yield
\begin{equation*}
u_n\to u \text{ in }L^{r_0}(0,T;L^2) \ \ \text{ for all }r_0\in (1,\infty).
\end{equation*}
Similarly, interpolating the above with the uniform bound in $L^{\infty}(0,T;L^q)$, we get
\begin{equation}
\label{eq:convergence_L_r_0_L_q_0_proof_compactness}
u_n\to u \text{ in }L^{r_0}(0,T;L^{q_0}) \ \ \text{ for all }r_0\in (1,\infty)\text{ and }q_0\in (1,q).
\end{equation}
We claim that there exist $r_0\in (1,\infty)$, $q_0\in (1,q)$ and $\theta_0\in (0,1)$ such that 
\begin{equation}
\label{eq:mix_integrability_L_r_L_q}
\|g\|_{L^{r}(0,T;L^q)}\lesssim \|g\|_{L^{r_0}(0,T;L^{q_0})}^{1-\theta_0}\|g\|_{ L^{q}(0,T;L^{\xi})}^{\theta_0} .
\end{equation}
To see the above one can argue as follows. Fix $\theta_0\in (0,1)$ such that $\frac{\theta_0}{q}< \frac{1}{r}$. Note that $\frac{1-\theta_0}{q}+ \frac{\theta_0}{\xi}< \frac{1}{q}$ since $\xi>q$. Hence there exist $r_0\in (2,\infty)$ and $q_0\in (1,q)$ such that 
$$
\frac{1-\theta_0}{r_0}+\frac{\theta_0}{q}\leq \frac{1}{r}
\qquad \text{ and }\qquad \frac{1-\theta_0}{q_0}+\frac{\theta_0}{\xi}\leq \frac{1}{q}.
$$
In particular, \eqref{eq:mix_integrability_L_r_L_q} follows with the above choice of $(r_0,q_0,\theta_0)$ and standard interpolation theory. 
Thus, \eqref{eq:convergence_L_r_0_L_q_0_proof_compactness} and the uniform bound in $L^q(0,T;L^{\xi})$,
we obtain $
u_n\to u \text{ in }L^{r}(0,T;L^{q}) 
$.
Combining this with  \eqref{eq:L_infty_H_varepsilon} for some $\varepsilon\in (0, \g]$, to conclude the proof it remains to note that $u_n\to u$ in $L^{2}(0,T;H^{1-\g})$ due to \eqref{eq:L_infty_H_varepsilon} and the uniform bound in $L^{2}(0,T;H^1)$.

The last claim follows from the Fatou lemma.
\end{proof}

To apply Lemma \ref{l:compactness} we have to investigate further  regularity estimate of solutions to reaction-diffusion equations with cut-off \eqref{eq:reaction_diffusion_system_truncation_2}. The following complements Theorem \ref{t:global_cut_off}\eqref{it:global_cut_off_2}.

\begin{lemma}[Time-regularity estimates]
\label{l:time_regularity}
Let Assumption \ref{ass:f_polynomial_growth} be satisfied. Fix $T\in (0,\infty)$, $R\geq 1$ and $\parameter\in (1,\infty)$. Assume that $v_0\in L^q$ and $q>\frac{d}{2}(h-1)$. Let $r\in [r_0,\infty)$ where $r_0$ is as in Theorem \ref{t:global_cut_off}.  Suppose that $\theta\in \ell^2$ satisfies \eqref{eq:theta_normalized_symmetric}. 
Let $v=(v_i)_{i=1}^{\ell}$ be $(p,\a,1,q)$--solution to \eqref{eq:reaction_diffusion_system_truncation_2} provided by Theorem \ref{t:global_cut_off} and set
$$
\Mart_i(t)\stackrel{{\rm def}}{=}\sqrt{c_d \ellip}\sum_{k,\alpha} \theta_k \int_{0}^t (\sigma_{k,\alpha}\cdot \nabla) v_i\,\dd w_s^{k,\alpha}.
$$
Then there exist $\g_0,\g_1,C_0>0$ independent of $(v_0,\theta)$ such that for all $i\in \{1,\dots,\ell\}$
\begin{align}
\label{eq:time_regularity_martingale}
\E\big[\|\mathcal{M}_i\|_{C^{\g_0}(0,T;H^{-\g_1})}^{2 \parameter}\big]
&\leq  C_0\|\theta\|^{2 \parameter}_{\ell^{\infty}}(1+\|v_0\|^{ \parameter q}_{L^q}),\\
\label{eq:time_regularity_v}
\E\big[ \|v_i\|_{C^{\g_0}(0,T;H^{-\g_1})}^{2 \parameter} \big]
&\leq C_0(1+\|v_0\|^{ 2\parameter q}_{L^q}).
\end{align}
\end{lemma}

The key point is that on the RHS\eqref{eq:time_regularity_martingale} we have the $\ell^{\infty}$--norm of $\theta$.

\begin{proof}
For notational convenience, we fix $i\in \{1,\dots,\ell\}$ and we drop it from the notation if no confusion seems likely.
The proof of \eqref{eq:time_regularity_martingale} follows almost the one of \cite{FGL21}, see p.\ 1779. Since the argument exploits
several basic properties of the noise, we include some details. 

Set $e_j(x)=e^{2\pi i j\cdot x}$ for $j\in \Z^d$ and $x\in \Tor^d$. Let $\g_1\in (0,\infty)$ be decided later. 
The It\^{o} isomorphism yields, for any $0\leq s\leq t\leq T$,
\begin{align*}
\E\big[\|\Mart(t)-\Mart(s)\|_{H^{-\g_1}}^{2 \parameter}\big]
&\eqsim_{\parameter}
 \E \Big[\sum_{k,\alpha}\theta_k^2\int_s^t \|(\sigma_{k,\alpha}\cdot\nabla) v_i\|_{H^{-\g_1}}^2\,\dd r \Big]^{\parameter}\\
&\eqsim \E \Big[\sum_{k,\alpha}\sum_{j\in \Z^d}\frac{\theta_k^2}{ (1+|j|^2)^{\g_1}} \int_s^t\big| \l e_j, (\sigma_{k,\alpha}\cdot\nabla) v_i \r\big|^2\,\dd r \Big]^{\parameter}\\
&\leq \|\theta\|_{\ell^{\infty}}^{2 \parameter} \E \Big[\sum_{k,\alpha}\sum_{j\in \Z^d} (1+|j|^2)^{-\g_1} \int_s^t\big| \l e_j, (\sigma_{k,\alpha}\cdot\nabla) v_i \r\big|^2\,\dd r \Big]^{\parameter}
\end{align*}
where $\l f,g\r =\int_{\Tor^d} f \cdot\overline{g}\,\dd x$.
Since $\div\,\sigma_{k,\alpha}=0$, we have
$$
 \l e_j, (\sigma_{k,\alpha}\cdot\nabla ) v_i\r
=
\int_{\Tor^d} e_j \,\div(\sigma_{k,\alpha} v_i)\,\dd x
=- 2\pi i j\cdot \int_{\Tor^d}  e_j v_i   \sigma_{k,\alpha}\,\dd x.
$$
Recall that 
 $(\sigma_{k,\alpha})_{k,\alpha}$ is an (incomplete) orthonormal basis of $L^2(\Tor^d;\R^{d})$. Therefore, for all $j\in \Z^d$ and a.e.\ on $[0,T]\times \O$, the Parseval identity yields
$$
\sum_{k,\alpha}\big| \l e_j, (\sigma_{k,\alpha}\cdot\nabla) v_i \r\big|^2\leq 
|j|^2 \|e_j v_i\|_{L^2}^2 \lesssim 
|j|^2 \|v_i\|_{L^2}^2\lesssim |j|^2(1+\|v_0\|_{L^q}^q)
$$
where in the last inequality we used Theorem \ref{t:global_cut_off}\eqref{it:global_cut_off_2}. 
Therefore
\begin{equation}
\label{eq:mart_difference_Lp_estimate}
\E\big[\|\Mart(t)-\Mart(s)\|_{H^{-\g_1}}^{2 \parameter} \big]
\lesssim \|\theta\|_{\ell^{\infty}}^{2 \parameter}(1+\|v_0\|_{L^q}^{q \parameter}) |t-s|^{ \parameter} 
\Big[\sum_{j\in \Z^d} \frac{ |j|^2 }{(1+|j|^2)^{\g_1}}\Big]^{\parameter}.
\end{equation}
Note that the sum on the RHS\eqref{eq:mart_difference_Lp_estimate} is finite provided $\g_1>(d+2)/2$. Combining \eqref{eq:mart_difference_Lp_estimate} and the Kolmogorov continuity modification theorem, one gets \eqref{eq:time_regularity_martingale} for all $\g_0\in (0,\frac{\parameter-1}{2\parameter})$.

Next we prove \eqref{eq:time_regularity_v}.
Recall that $v=(v_i)_{i=1}^{\ell}$ is a $(p,\a,1,q)$--solution to \eqref{eq:reaction_diffusion_system_truncation_2}. 
Thus 
$$
v=v_0+\Det(t)+\Mart(t)\ \ \  \text{ a.e.\ on }[0,T]\times \O,
$$
where $\Det\stackrel{{\rm def}}{=}\Det_{\Delta}+ \Det_f+\Det_{F}$ and for $t\in [0,T]$
\begin{align*}
\Det_{\Delta}(t)&=\int_{0}^t\Big( (\ellip_i+\ellip) \Delta v_i \Big)_{i=1}^{\ell}\,\dd s, \\ 
\Det_{f}(t)&=\int_{0}^t\Big( \phi_{R,r}(\cdot,v) f_i(\cdot,v)\Big)_{i=1}^{\ell}\,\dd s,\\
\Det_{F}(t)&=\int_{0}^t \Big(\phi_{R,r}(\cdot,v) \div(F_i(\cdot,v))\Big)_{i=1}^{\ell}\,\dd s.
\end{align*}
Since \eqref{eq:time_regularity_martingale} has been already proved  and $\|\theta\|_{\ell^{\infty}}\leq \|\theta\|_{\ell^2}=1$ by assumption, to prove \eqref{eq:time_regularity_v}
 it is sufficient to estimate $\Det$. 
By Theorem \ref{t:global_cut_off}\eqref{it:global_cut_off_2} and $q\geq 2$, we have a.s.\
\begin{equation*}
\|\mathcal{D}_{\Delta}\|_{H^{1}(0,T;H^{-1})}^{2}\lesssim_T \|v_i\|_{L^2(0,T;H^1)}^{2}
\lesssim 1+\|v_0\|_{L^q}^q .
\end{equation*}

Recall that $r\in [r_0,\infty)$ where $r_0$ is as in Theorem \ref{t:global_cut_off}. Next fix $h_0\in [h,\infty)$ such that  $1+\frac{4}{d}< h_0<1+\frac{2q}{d}$. Note that $q>\frac{d}{2}(h_0-1)>2$ by construction. Finally, set $\zeta\stackrel{{\rm def}}{=}\frac{q+h_0-1}{h_0}$. One can check that $\zeta\in (1,\infty)$ since $q> 2$ by assumption. 
Hence, a.s.,
\begin{align*}
\|\mathcal{D}_{f}\|_{W^{1, \zeta}(0,T;L^{\zeta})}^{\zeta}
&\lesssim \max_{1\leq i\leq \ell}\int_0^T\int_{\Tor^d}(\phi_{R,r}(\cdot,v))^{\zeta} |f_i(\cdot,v)|^{\zeta}\,\dd x\,\dd s\\
&\lesssim_{R,T} 1+ \int_0^T \int_{\Tor^d} (\phi_{R,r}(\cdot,v))^{\zeta}(1+|v|^{h_0 \zeta})\,\dd x\,\dd s	\\
&= 1+ \int_0^T \int_{\Tor^d}(\phi_{R,r}(\cdot,v))^{\zeta} |v|^{q+h_0-1}\,\dd x\,\dd s\\
&\stackrel{(i)}{\lesssim}_R  1+\max_{1\leq i\leq \ell}\Big( \int_0^T \int_{\Tor^d}|v_i|^{q-2}|\nabla v_i|^2 \,\dd x\,\dd s\Big)^{\beta}
\stackrel{(ii)}{\lesssim} 1+\|v_0\|_{L^q}^{q\beta}
\end{align*}
where in $(i)$ we used Lemma \ref{l:interpolation_L_eta} and that $\phi_{R,r}(t,v)=0$ for all $t\in[0,T]$ such that $\|v\|_{L^r(0,t;L^q)}\geq R$ (cf.\ Step 2 in the proof of Theorem \ref{t:global_cut_off}\eqref{it:global_cut_off_2} for similar compuations). Inspecting the proof of Lemma \ref{l:interpolation_L_eta} , we also have $\beta=\frac{\theta \psi}{2}$ where $\theta<\frac{d}{d+2}$ and $\psi=\frac{2}{q}(q+h_0-1)$. Since $h_0>1+\frac{2}{d}$ and $q>\frac{d}{2}(h_0-1)$, we have $\beta<\zeta$ and therefore
$$
\|\mathcal{D}_{f}\|_{W^{1, \zeta}(0,T;L^{\zeta})}\lesssim 
1+\|v_0\|_{L^q}^{q} \text{ a.s.\ }
$$
Using the above argument, the fact that $q+h_0-1\geq h_0+1$ and Assumption \ref{ass:f_polynomial_growth}\eqref{it:f_polynomial_growth_1}, we have 
\begin{equation*}
\|\mathcal{D}_{F}\|_{H^{1}(0,T;H^{-1})}\lesssim \|\phi_{R,r}(\cdot,v) |v|^{(h_0+1)/2}\|_{L^2(0,T;L^2)}
\lesssim 1+\|v_0\|_{L^q}^{q} \text{ a.s.}
\end{equation*}
To conclude, note that, by Sobolev embeddings, $H^{1}(0,T;H^{-1})\embed C^{1/2}(0,T;H^{-1})$ and 
$$ 
W^{1,\zeta}(0,T;L^{\zeta})\embed C^{(\zeta-1)/\zeta}(0,T;L^{\zeta})\embed C^{(\zeta-1)/\zeta}(0,T;H^{-k})
$$ 
where $k \geq 1$ is large. Thus the conclusion follows \eqref{eq:time_regularity_v} by collecting the previous estimates.
\end{proof}

We are ready to prove Theorem \ref{t:weak_convergence}. Here we follow \cite[Proposition 3.7]{FGL21}. 

\begin{proof}[Proof of Theorem \ref{t:weak_convergence}]
By \eqref{it:weak_convergence_L_q}, we have $\sup_{n\geq 1} \|v_0^{(n)}\|_{L^q}<\infty$. Let $N\geq 1$ be an integer such that $N\geq \sup_{n\geq 1} \|v_0^{(n)}\|_{L^q}$. Fix $\g\in (0,1)$.
For $n\geq 1$, let $v^{(n)}$ be the global $(p,\a,1,q)$--solution to \eqref{eq:reaction_diffusion_system_truncation_2} provided by Theorem \ref{t:global_cut_off}. Let $\mu^{(n)}$ be the law of $v^{(n)}:\O\to \X$ where
$$
\X\stackrel{{\rm def}}{=}L^2(0,T;H^{1-\g})\cap C([0,T];H^{-\g})\cap  L^r(0,T;L^{q}).
$$
By
Theorem \ref{t:global_cut_off}\eqref{it:global_cut_off_2} and \eqref{eq:conseguence_of_estimate}, there exists a constant $K(N)\in (0,\infty)$, independent of $n\geq 1$, such that $v^{(n)}\in \X_K$ a.s.\ for all $n\geq 1$ where
$$
\X_K\stackrel{{\rm def}}{=}\Big\{u\in \X\,:\, 
\sup_{t\in [0,T]}\|u(t)\|_{L^q}+ \|u\|_{L^{q}(0,T;L^{\xi})}\leq K\Big\}.
$$ 
Recall that $\xi=\frac{dq}{d-2}$ if $d\geq 3$ and $\xi\geq \xi_0$ where $\xi_0$ is sufficiently large otherwise (cf.\ Proposition \ref{prop:uniqueness}).
By Theorem \ref{t:global_cut_off}\eqref{it:global_cut_off_2}  and Lemma \ref{l:time_regularity}, there exists $\g_0,\g_1>0$ such that
$$
\sup_{n\geq 1} \E \Big[\sup_{t\in [0,T]}\|v^{(n)}(t)\|_{L^q}^2 + \|v^{(n)}\|_{L^2(0,T;H^1)}^2 
+ \|v^{(n)}\|_{C^{\g_0}(0,T;H^{-\g_1})}^2 +\|v^{(n)}\|_{L^{q}(0,T;L^{\xi})}^2  \Big]<\infty.
$$ 
By Prokhorov's theorem and Lemma \ref{l:compactness}, there exists a probability measure $\mu$ on $\X$ such that $\mu^{(n)} \rightharpoonup \mu$ (up to take a non-relabeled subsequence). Note that $\supp\,\mu\subseteq \X_K$ as $\supp\,\mu^{(n)}\subseteq \X_K$ for all $n\geq 1$.
We now divide the proof into two steps.

\emph{Step 1: Consider the truncated reaction-diffusion with cut-off as in \eqref{it:determinstic_limit_cut_off}:}
\begin{equation}
\label{eq:reaction_diffusion_cut_off_step_1}
\left\{
\begin{aligned}
\partial_t v_i &=(\ellip_i + \ellip)\Delta v_i +  \phi_{R,r}(\cdot,v)\big[\div(F_i(\cdot,v))+ f_i (\cdot,v)\big] ,& \text{ on }&\Tor^d,\\ 
v_i(0)&=v_{0,i}, & \text{ on }&\Tor^d,
\end{aligned}
\right.
\end{equation}
\emph{where $i\in \{1,\dots,\ell\}$. Then }
\begin{equation}
\label{eq:mu_concentrates_on_weak_solutions}
\mu\big(u=(u_i)_{i=1}^{\ell}\in\X_K\,:\, u\text{ is a weak solution to }\eqref{eq:reaction_diffusion_cut_off_step_1}\big)=1.
\end{equation}

Recall that weak solution to \eqref{eq:reaction_diffusion_cut_off_step_1} in the class $\X_K$ are defined in Corollary \ref{cor:uniqueness}.

Fix $\fun\in C^{\infty}(\Tor^d;\R^{\ell})$. Let $\T_{\fun}: \X_K\to C([0,T])$ be given by 
\begin{align*}
&[\T_{\fun}(u)](t)
\stackrel{{\rm def}}{=}\l \fun,u(t)\r- \int_{\Tor^d} v_{0}\cdot \fun\,\dd x \\
&\ \ 
-\sum_{1\leq i\leq \ell} \int_0^t\int_{\Tor^d}\Big( (\ellip+\ellip_i)\, u_i \Delta \fun_i +\phi_{R,r}(\cdot,\cdot) \big[ f_i(\cdot,u) \fun_i -F_i(\cdot,u)\cdot \nabla \fun_i \big] \Big)\,\dd x\,\dd s ,
\end{align*}
where $u\in \X_K$, $t\in [0,T]$ and $\l \cdot,\cdot\r$ denotes the pairing in the duality $(H^{\g},H^{-\g})$.
In the following we prove the continuity of the map $\T_{\fun,f}: \X_K\to C([0,T])$ defined as
\begin{align*}
[\T_{\fun,f} (u)](t)
\stackrel{{\rm def}}{=}\sum_{1\leq i\leq \ell} \int_0^t\int_{\Tor^d}\phi_{R,r}(\cdot,u)  f_i(\cdot,u) \fun_i\,\dd x\,\dd s .
\end{align*}
The remaining terms in $\T_{\fun}$ can be treated analogously using also that $\X_K\subseteq L^2(0,T;H^{1-\g})$. 
By Lebesgue domination theorem, we have, for all $u^{(1)},u^{(2)}\in \X_K$,
\begin{align*}
&\|\T_{\fun,f} (u^{(1)})-\T_{\fun,f}(u^{(2)})\|_{C([0,T])}\\
& \lesssim_{\fun} \int_0^T \int_{\Tor^d}| f(\cdot,u^{(1)})| \big|\phi_{R,r}(\cdot,u^{(1)})-\phi_{R,r}(\cdot,u^{(2)}) \big|\,\dd x\,\dd s
+ \int_0^T \int_{\Tor^d} |f(\cdot,u^{(1)})-f(\cdot,u^{(2)})|\,\dd x\,\dd s\\
& \lesssim \sup_{t\in [0,T]}|\phi_{R,r}(t,u^{(1)})-\phi_{R,r}(t,u^{(2)}) |  \int_0^T \int_{\Tor^d}| f(\cdot,u^{(1)})| \,\dd x\,\dd s
+ \int_0^T \int_{\Tor^d} |f(\cdot,u^{(1)})-f(\cdot,u^{(2)})|\,\dd x\,\dd s\\
& 
\lesssim_R \|u^{(1)}-u^{(2)}\|_{L^r(0,T;L^q)}\Big(1+ \|u^{(1)}\|_{L^h(0,T;L^h)}^h \Big)\\
&\qquad \qquad\qquad\qquad \ \
+ \Big(1+\|u^{(1)}\|_{L^h(0,T;L^h)}^{h-1}+ \|u^{(2)}\|_{L^h(0,T;L^h)}^{h-1}\Big)\|u^{(1)}-u^{(2)}\|_{L^h(0,T;L^h)},
\end{align*}
where we used Assumption \ref{ass:f_polynomial_growth}\eqref{it:f_polynomial_growth_1} and that $\phi$ is bounded and Lipschitz continuous. 
By Remark \ref{r:interpolation_L_eta_2} and $q\geq 2$ we have, for some $\alpha,\beta>0$ and all $u\in \X$,
$$
\| u\|_{L^{h}(0,T;L^h)}\lesssim_{q,h}
\|u\|_{L^{r}(0,T;L^{q})}^{\alpha}\|u\|_{L^q(0,T;L^{\xi})}^{\beta} .
$$
Thus the continuity of $\T_{\fun,f}$ on $\X_K$ follows from by combining the above estimates and using that 
$\|u^{(1)}\|_{L^q(0,T;L^{\xi})},\|u^{(2)}\|_{L^q(0,T;L^{\xi})}\leq K$ a.s.
Since $\T_{\fun}$ is continuous, we may define the pushforward measures of $\mu^{(n)}$ and $\mu$ under the map $\T_{\fun}$, respectively:
$$
\mu_{\fun,\#}^{(n)}\stackrel{{\rm def}}{=}
\mu^{(n)}(\T_{\fun}^{-1}\cdot)\quad \text{ and }\quad \mu_{\fun,\#}\stackrel{{\rm def}}{=}
\mu(\T_{\fun}^{-1}\cdot).
$$ 
Observe that $
\mu_{\fun,\#}^{(n)} \rightharpoonup  \mu_{\fun,\#}$ as $\mu^{(n)} \rightharpoonup \mu$ and that $\mu_{\psi,\#}^{(n)}$ is the law with $\T_{\fun} v^{(n)}$. Moreover $\T_{\fun} v^{(n)}$ satisfies
\begin{equation}
\label{eq:T_psi_n_decomposition}
\T_{\fun} v^{(n)} = \l \fun,v_0^{(n)}- v_0\r + \l \fun,\Mart^{(n)} \r
\end{equation}
where 
$$
\Mart_i^{(n)}(t)
\stackrel{{\rm def}}{=}
\sqrt{c_d \ellip}\sum_{k,\alpha} \theta_k \int_{0}^t (\sigma_{k,\alpha}\cdot \nabla) v_i^{(n)}\,\dd w_s^{k,\alpha}.
$$
By \eqref{eq:time_regularity_martingale} in Lemma \ref{l:time_regularity} and $\|\theta^{(n)}\|_{\ell^{\infty}}\to 0$ (see assumption \eqref{it:theta_goes_to_zero}) we have, for all $\parameter \in (1,\infty)$,
$$
\E\sup_{t\in [0,T]}|\l \fun,\Mart^{(n)}(t)\r|^{2\parameter} \lesssim \|\theta^{(n)}\|_{\ell^{\infty}}^{2\parameter} \to 0 \ \ \text{ as }n\to \infty.
$$
Using the above and assumption \eqref{it:weak_convergence_L_q} in \eqref{eq:T_psi_n_decomposition}, one can check that  
$
\supp\mu_{\fun,\#}= \{0\}.
$
The conclusion follows from the separability of $H^{-\g}$ and the density of the embedding $C^{\infty}\embed H^{-\g}$ (cf.\ the last part of the proof of \cite[Proposition 3.7]{FGL21}).

\emph{Step 2: Let $v$ be as in \eqref{it:determinstic_limit_cut_off}. Then $v^{(n)}\to v$ in probability in $\X$. In particular \eqref{eq:claim_scaling_limit} holds}. It suffices to show that
\begin{equation}
\label{eq:mu_equal_to_v}
\mu= \delta_v,
\end{equation}
where $\delta_v$ is the Dirac measure at $v\in \X$.
To see this recall that $v$ is independent of $\om\in \O$. Hence, 
\begin{align*}
\limsup_{n\to \infty} \P(\|v^{(n)}-v\|_{\X}\geq \varepsilon)&=
\limsup_{n\to \infty}\P\big(v^{(n)}\in \complement B_{\varepsilon}(v)\big)\\
&=
\limsup_{n\to \infty} \mu^{(n)}\big(\complement B_{\varepsilon}(v)\big)
 \stackrel{(i)}{\leq} \delta_v\big(\complement B_{\varepsilon}(v)\big)= 0
\end{align*}
where $\complement B_{\varepsilon}(v)\stackrel{{\rm def}}{=}\{u\in \X\,:\,\|v-u\|_{\X}\geq \varepsilon\}$ and $(i)$ we used $\mu^{(n)} \rightharpoonup \delta_v$ due to \eqref{eq:mu_equal_to_v}.

It remains to prove \eqref{eq:mu_equal_to_v}. By \eqref{it:determinstic_limit_cut_off}, $v$ is the \emph{unique} weak solution in $\X\subseteq \X_K$ to the reaction-diffusion equation with cut-off \eqref{eq:reaction_diffusion_cut_off_step_1} and therefore
\begin{equation}
\label{eq:weak_strong_uniqueness_sets}
\big\{u=(u_i)_{i=1}^{\ell}\in\X_K\,:\, u\text{ is a weak solution to }\eqref{eq:reaction_diffusion_cut_off_step_1}\big\}= \{v\}.
\end{equation}
Hence \eqref{eq:mu_equal_to_v} follows by combining the above with \eqref{eq:mu_concentrates_on_weak_solutions}.
\end{proof}

The arguments of Theorem \ref{t:weak_convergence} also yield a suitable continuity 
of weak solutions for system of deterministic reaction-diffusion equations with cut-off:
\begin{equation}
\label{eq:reaction_diffusion_deterministic_cut_off_weak_limit}
\left\{
\begin{aligned}
\partial_t v_i &=\mu_i\Delta v_i +\phi_{R,r}(\cdot,v)\big[\div(F_i(\cdot,v)) + f_i(\cdot,v)\big], &\text{ on }&\Tor^d,\\
v_i(0)&=v_{0,i},&\text{ on }&\Tor^d.
\end{aligned}
\right.
\end{equation}
As it will be needed in the proof of Theorem \ref{t:delayed_blow_up}, we formulate it in the next result.
Recall that weak solutions to \eqref{eq:reaction_diffusion_deterministic_cut_off_weak_limit} in $\X$ are defined in Corollary \ref{cor:uniqueness}.

\begin{proposition}
\label{prop:weak_to_strong_convergence_deterministic_problem}
Let Assumption \ref{ass:f_polynomial_growth}\eqref{it:f_polynomial_growth_1}--\eqref{it:mass} be satisfied.
Fix $T\in (0,\infty)$ and $R\geq 1$.
Assume that $\mu_i>0$ for all $i\in \{1,\dots,\ell\}$.
Let $\frac{d(h-1)}{2}\vee 2 < q<\infty$. 
Fix $v_0\in L^q$. Let $\xi$ and $\X$ be as in Proposition \ref{prop:uniqueness}. 
\begin{enumerate}[{\rm(1)}]
\item\label{it:weak_to_strong_convergence_deterministic_problem_1} Let $(v_0^{(n)})_{n\geq 1}\subseteq L^q$ be a sequence such that $v_0^{(n)} \rightharpoonup v_0$ in $L^q$.
\item\label{it:weak_to_strong_convergence_deterministic_problem_2}
Suppose that there exists a \emph{unique} weak solution $v\in \X$ to \eqref{eq:reaction_diffusion_deterministic_cut_off_weak_limit} such that, for some $\g_0,\g_1>0$,
$$
v\in \Y\stackrel{{\rm def}}{=}L^2(0,T;H^1) \cap L^{\infty}(0,T;L^{q})\cap  C^{\g_0}(0,T;H^{-\g_1})\cap L^q(0,T;L^{\xi}).
$$
Moreover, for all $n\geq 1$, there exists a weak solution $\vd^{(n)}\in \X$ to \eqref{eq:reaction_diffusion_deterministic_cut_off_weak_limit} with initial data $v_0^{(n)}$ such that 
$
\sup_{n\geq 1}
\|\vd^{(n)}\|_{\Y}<\infty.
$
\end{enumerate}
Then $\vd^{(n)}\to \vd$ in $\X$.
\end{proposition}

In applications \eqref{it:weak_to_strong_convergence_deterministic_problem_2} will be checked using Proposition \ref{prop:global_high_viscosity} and Corollary \ref{cor:uniqueness}. 

\begin{proof}
It is enough to show that for each subsequence of $(\vd^{(n)})_{n\geq1 }$, we may find a subsequence such that $\vd^{(n)}\to \vd$ in $\X$. As above, to economize the notation, we do not relabel subsequences. 
 
By Lemma \ref{l:compactness} and the bound in \eqref{it:weak_to_strong_convergence_deterministic_problem_2}, there exists in $u\in \X$ such that $\vd^{(n)}\to u$ in $\X$. By \eqref{it:weak_to_strong_convergence_deterministic_problem_1} and arguing as in the Step 1 of Theorem \ref{t:weak_convergence} we may pass to the limit in the weak formulation of \eqref{eq:reaction_diffusion_deterministic_cut_off_weak_limit} (cf.\ Corollary \ref{cor:uniqueness}). 
Hence $u\in \X$ is a weak solution to \eqref{eq:reaction_diffusion_deterministic_cut_off_weak_limit}. The uniqueness of $\vd$ (see assumption \eqref{it:weak_to_strong_convergence_deterministic_problem_2}) forces $u=\vd$. 
\end{proof}

\subsection{Proof of Theorem \ref{t:delayed_blow_up}}
\label{ss:proof_global_T}
As a preparatory step for Theorem \ref{t:delayed_blow_up}, we prove the following version of it with sufficiently smooth initial data $v_0$ where $(\theta,\ellip)$ depend only the $L^q$-norm of $v_0$. 
Once this is proved, Theorem \ref{t:delayed_blow_up} follows from such result and a standard density argument. 
Recall that the existence and uniqueness for \eqref{eq:reaction_diffusion} is ensured by Theorem \ref{t:local}.

\begin{proposition}[Delayed blow-up and weak enhanced diffusion -- Smooth data]
\label{prop:delayed_blow_up_smooth}
Let Assumption \ref{ass:f_polynomial_growth} be satisfied. 
Fix $N\geq 1$, $\varepsilon\in (0,1)$, $T,\ellip_0\in (0,\infty)$ and $r\in (1,\infty)$. 
Then there exist
\begin{equation}
\begin{aligned}
\ellip\geq \ellip_0, \quad   R>0 , \quad
\theta\in \ell^2(\Z_0^d) \ \ \text{ with } \ \   \#\{k\,:\,\theta_k\neq 0\}<\infty
\end{aligned}
\end{equation}
such that, for all initial data $v_0\in B^{1-2\frac{1+\a}{p}}_{q,p}(\Tor^d;\R^{\ell})$ satisfying $v_0\geq 0 $ (component-wise)   on $\Tor^d$ and 
$\|v_0\|_{L^{q}(\Tor^d;\R^{\ell})}\leq N$, 
the unique $(p,\a,1,q)$--solution $(v,\tau)$ to \eqref{eq:reaction_diffusion} with $(\ellip,\theta)$ as above satisfies  the assertions \eqref{it:delayed_blow_up}-\eqref{it:enhanced_dissipation} of Theorem \ref{t:delayed_blow_up} and
\begin{equation}
\label{eq:uniform_estimate_smooth_case_L_r_q}
\P\big(\tau\geq T,\, \|v\|_{L^r(0,T,L^q)}\leq R\big)>1-\varepsilon.
\end{equation}
Finally, there exists $K_0>0 $, independent of $v_0$ (but depending on $N\geq 1$), such that 
\begin{equation}
\label{eq:uniform_estimate_smooth_case_H_s_q}
\P\big(\tau\geq T,\, \|v\|_{L^p(0,T,w_{\a_{p,\s}};H^{2-\s,q})}\leq K_0\big)>1-2\varepsilon \ \text{ where } \ \a_{p,\s}\stackrel{{\rm def}}{=}p\big(1-\tfrac{\s}{2}\big)-1.
\end{equation} 
\end{proposition}

Recall that $(p,q,\a)$ in the above result are fixed in Assumption \ref{ass:f_polynomial_growth}\eqref{it:integrability_exponent_main_assumption}. In particular $\a\in [0,\frac{p}{2}-1)$ and therefore the initial data $v_0$ considered Proposition \ref{prop:delayed_blow_up_smooth} has \emph{positive} smoothness.
As explained below the statement of Theorem \ref{t:delayed_blow_up} the presence of $\s>1$ in \eqref{eq:uniform_estimate_smooth_case_H_s_q} is necessary to obtain $K_0$ independent of $v_0$ (indeed, the Sobolev index of $L^p(0,T,w_{\a_{p,\s}};H^{2-\s,q})$ is equal to the one of  $L^q$).

The above result can be proven following the proof of \cite[Theorem 1.4]{FGL21}. As our setting (slightly) differs from the one of \cite{FGL21}, we include some details.

\begin{proof}[Proof of Proposition \ref{prop:delayed_blow_up_smooth}]
Throughout this proof we let $(N,\varepsilon,T,\ellip_0,r)$ be as in the statement of Proposition \ref{prop:delayed_blow_up_smooth}. Without loss of generality we assume $r\geq r_0$ where $r_0$ is as in Theorem \ref{t:global_cut_off}.
Moreover, to make the argument below more transparent, we display the dependence on the initial data for the equation considered. For instance, the $(p,\a,1,q)$--solution to \eqref{eq:reaction_diffusion} with data $v_0$ will be denoted by $(v(v_0),\tau(v_0))$.

We begin by collecting some useful facts. Set 
\begin{align}
\label{eq:KN_definition}
\KN&\stackrel{{\rm def}}{=}\Big\{v_0\in B^{1-2\frac{1+\a}{p}}_{q,p}(\Tor^d;\R^{\ell})\,:\, v_0\geq 0 \text{ on }\Tor^d \text{ and }
\|v_0\|_{L^{q}(\Tor^d;\R^{\ell})}\leq N\Big\},\\
\label{eq:LN_definition}
\LN&\stackrel{{\rm def}}{=}\big\{v_0\in L^q(\Tor^d;\R^{\ell})\,:\, v_0\geq 0 \text{ on }\Tor^d \text{ and }
\|v_0\|_{L^{q}(\Tor^d;\R^{\ell})}\leq N\big\}.
\end{align}
Note that $\KN\subseteq \LN$.
Proposition \ref{prop:global_high_viscosity} ensures the existence of positive constants $\ellip\geq \ellip_0$ and $R>1$,
both independent of $v_0\in \LN$, for which the deterministic reaction-diffusion equations \eqref{eq:reaction_diffusion_deterministic} with $\mu_i=\ellip_i+ \ellip$ have a $(p,q)$--solution $\vd(v_0)$ on $[0,T]$ for all initial data $v_0\in \LN$ and 
\begin{equation}
\label{eq:R_bound_det_solution}
 \|\vd(v_0)\|_{L^r(0,T;L^q)}\leq R-1.
\end{equation}
Due to \eqref{eq:R_bound_det_solution} and \eqref{eq:def_cut_off}, $\vd(v_0)$ is a $(p,q)$--solution on $[0,T]$ to the deterministic problem with cut-off \eqref{eq:reaction_diffusion_deterministic_cut_off_weak_limit} where $\mu_i=\ellip_i+\ellip $, $R$ as above and initial data $v_0\in \LN$. 

Finally, 
Let $(\theta^{(n)})_{n\geq 1}$ be the sequence defined in \eqref{eq:choice_theta}. For any $n\geq 1$, 
Theorem \ref{t:global_cut_off} provides a unique strong solution $\vcn(v_0)$ to the reaction-diffusion equations with cut-off \eqref{eq:reaction_diffusion_system_truncation} for all initial data $v_0\in \KN$, $R$ is as in \eqref{eq:R_bound_det_solution} and $\theta=\theta^{(n)}$.

The key idea now is to prove that, for all $\varepsilon\in (0,1)$,
\begin{equation}
\label{eq:claim_convergence_v_n_theta_v_star}
\lim_{n\to \infty}\sup_{v_0\in \KN} \P\big( \|\vcn(v_0)-\vd(v_0)\|_{L^r(0,T;L^q)}\geq \varepsilon\big)=0.
\end{equation}

We break the proof of \eqref{eq:claim_convergence_v_n_theta_v_star} in several steps. The proof of \eqref{eq:claim_convergence_v_n_theta_v_star} is postponed to Step 4. In Step 1 we prove that \eqref{eq:claim_convergence_v_n_theta_v_star} implies the assertions \eqref{it:delayed_blow_up}-\eqref{it:enhanced_dissipation} of Theorem \ref{t:delayed_blow_up} and \eqref{eq:claim_convergence_v_n_theta_v_star}.
In Steps 2 we prove additional interpolation estimates, which complements the one in Lemma \ref{lem:interpolation_strong_setting}, and leads to the proof of \eqref{eq:uniform_estimate_smooth_case_H_s_q} given in Step 3.

\emph{Step 1: If \eqref{eq:claim_convergence_v_n_theta_v_star} holds, then there exist $(\ellip,\theta,R)$ independent of $v_0\in \KN$ for which the assertions \eqref{it:delayed_blow_up}-\eqref{it:enhanced_dissipation} of Theorem \ref{t:delayed_blow_up} and \eqref{eq:uniform_estimate_smooth_case_L_r_q} hold.} 

By \eqref{eq:claim_convergence_v_n_theta_v_star}, we can choose $n_*\geq 1$, independent of $v_0\in \KN$, such that 
\begin{equation}
\label{eq:convergence_v_n_theta_v_star_proof_step_1}
 \P\big( \|\vcns(v_0)-\vd(v_0)\|_{L^r(0,T;L^q)}\leq  \varepsilon\big)>  1-\varepsilon .
\end{equation}
Combining \eqref{eq:R_bound_det_solution},  \eqref{eq:convergence_v_n_theta_v_star_proof_step_1} and $\varepsilon<1$, for all $v_0\in \KN$,
\begin{equation}
\label{eq:choice_n_star}
\P\big( \|\vcns(v_0)\|_{L^r(0,T;L^q)}< R\big)>1-\varepsilon.
\end{equation}
Next fix $v_0\in \KN$.
Let $\tau_*$ be the stopping time given by
\begin{equation*}
\tau_*\stackrel{{\rm def}}{=}\inf\big\{t\in [0,T]\,:\, \|\vcns(v_0)\|_{L^r(0,T;L^q)}\geq  R\big\}, \ \  \text{ where }\ \ \inf\emptyset\stackrel{{\rm def}}{=}T.
\end{equation*} 
Note that, due to \eqref{eq:choice_n_star} and the definition of $\tau_*$, we have 
\begin{equation}
\label{eq:tau_star_good}
\P(\tau_*=T)>1-\varepsilon, \ \ \text{ and }\ \ 
\phi_{R,r}(\cdot,\vcns)=1 \text{ on }[0,\tau_*]\times \O.
\end{equation}
By using the second condition in \eqref{eq:tau_star_good}, one can readily check that  
$(\vcns|_{[0,\tau_*]\times \O},\tau_*)$ is a local $(p,\a,1,q)$--solution to  the original problem \eqref{eq:reaction_diffusion} in the sense of Definition \ref{def:solution}. By maximality of $(v,\tau)$ (see the last item of Definition \ref{def:solution}), we have
\begin{equation}
\label{eq:relation_tau_star_tau}
\tau_*\leq \tau \ \text{ a.s.}, \quad \text{and }\quad  \vcns|_{[0,\tau_*]\times \O}=v\  \text{ a.e.\ on }[0,\tau)\times \O.
\end{equation}
Thus the assertions \eqref{it:delayed_blow_up}--\eqref{it:enhanced_dissipation} of Theorem \ref{t:delayed_blow_up} follows by combining \eqref{eq:convergence_v_n_theta_v_star_proof_step_1} and \eqref{eq:tau_star_good}--\eqref{eq:relation_tau_star_tau}. Finally, \eqref{eq:uniform_estimate_smooth_case_L_r_q} follows from \eqref{eq:choice_n_star}--\eqref{eq:relation_tau_star_tau}.

\emph{Step 2: There exist $\alpha_1,\alpha_2>0$, $\beta_1\in (0,h)$ and $\beta_2\in (0,\frac{h+1}{2})$, depending only on $(h,q,\s,d)$, such that, for all $u\in H^{2-\s,q}$,}
\begin{align}
\label{eq:sublinear_estimate_f_H_s}
\|f(\cdot,u)\|_{H^{-\s,q}}&\lesssim 1+ \|u\|_{L^q}^{\alpha_1}\|u\|_{H^{2-\s,q}}^{\beta_1 h},\\
\label{eq:sublinear_estimate_F_H_s}
\|\div(F(\cdot,u))\|_{H^{-\s,q}}&\lesssim 1+ \|u\|_{L^q}^{\alpha_2}\|u\|_{H^{2-\s,q}}^{\beta_2 \frac{h+1}{2}}.
\end{align}
The proof follows as the one of Lemma \ref{lem:interpolation_strong_setting}. However, for the reader's convenience, we include a proof of \eqref{eq:sublinear_estimate_f_H_s}. Recall that $q>\frac{d(h-1)}{2}\vee \frac{d}{d-\s}$ by assumption. By Assumption \ref{ass:f_polynomial_growth}\eqref{it:f_polynomial_growth_1}, 
\begin{align*}
\|f(\cdot,u)\|_{H^{-\s,q}}
\stackrel{(i)}{\lesssim} \|f(\cdot,u)\|_{L^{\zeta}}
\lesssim 1+\|u\|_{L^{h\zeta}}^{h}.
\end{align*}
where in $(i)$ we used the Sobolev embedding $L^{\zeta}\embed H^{-\s,q}$ and $\zeta=\frac{dq}{\s q+d}>1$ (as $q>\frac{d}{d-\s}$). 

Now, if $h\zeta\leq q$, then \eqref{eq:sublinear_estimate_f_H_s} follows with $\alpha_1=h$ and $\beta_1=0$. Next, it remains to discuss the case $h\zeta >q$. In the latter case, we employ Sobolev embeddings once more. Note that 
$$
H^{\varphi,q}\embed L^{h \zeta} \quad \Longleftrightarrow \quad 
\varphi-\frac{d}{q}= -\frac{d}{h\zeta}= -\frac{1}{h}\Big(\s+\frac{d}{q}\Big).
$$
Thus $\varphi=-\frac{\s}{h}+\frac{d}{q}(1-\frac{1}{h})$. Note that $\varphi>0$ since $h\zeta >q$. Moreover $\varphi<2-\s$. To see the latter, note that it is equivalent to $\s+\frac{d}{q}<\frac{2h}{h-1}$ and it is satisfied since $\s<2$ and $q>\frac{d(h-1)}{2}$. Since $[L^q,H^{2-\s,q}]_{\psi}=H^{\varphi,q}$ for $\beta_1=\frac{\varphi}{2-\s} \in (0,1)$, collecting the previous observations we have
\begin{equation*}
\|f(\cdot,u)\|_{H^{-\s,q}}\lesssim 1+ \|u\|_{L^q}^{(1-\beta_1 )h}\|u\|^{\beta_1 h}_{H^{2-\s,q}}.
\end{equation*}
Hence \eqref{eq:sublinear_estimate_f_H_s} follows from the above as $\beta_1 h<1$ is equivalent $q>\frac{d(h-1)}{2}$.

\emph{Step 3: Proof of \eqref{eq:uniform_estimate_smooth_case_H_s_q}}. The claim of this step follows the arguments used in Step 2 of Theorem \ref{t:global_cut_off}\eqref{it:global_cut_off_1}. 
Recall that \eqref{eq:uniform_estimate_smooth_case_L_r_q} was proven in Step 2. Fix $v_0\in \KN$ and set
$$
\g\stackrel{{\rm def}}{=}\inf\{t\in [0,\tau(v_0))\,:\,\|v(v_0)\|_{L^r(0,t;L^q)}\geq R \}\wedge T \quad \text{ and }\quad
 \inf\emptyset\stackrel{{\rm def}}{=} \tau\wedge T,
$$
where $(v(v_0),\tau(v_0))$ is the $(p,\a,1,q)$--solution to \eqref{eq:reaction_diffusion}.
Note that $\g$ is a stopping time due to Remark \ref{r:regularity_paths}\eqref{it:regularity_paths_2}, and $\P(\g=T)>1-\varepsilon$ by \eqref{eq:uniform_estimate_smooth_case_L_r_q}.
In virtue of Step 2, up to enlarge $r_0$ if needed, one can repeat the arguments in Step 2 of Theorem \ref{t:global_cut_off}\eqref{it:global_cut_off_1} with the spaces $(H^{-1,q}H^{1,q})$ and the stochastic interval $[0,\tau\wedge T)\times \O$ are replaced by $(H^{-\s,q},H^{2-\s,q})$ and $[0,\g)\times \O$, respectively. In particular, by using the stochastic maximal $L^p$--regularity estimates (see e.g.\ \cite[Theorem 1.2]{AV21_SMR_torus}), one obtains the analogue of the estimate  \eqref{eq:claim_step_2_estimate_a_priori_estimate_besov_space} in the current situation: 
\begin{equation}
\label{eq:estimate_g_R}
\E\|v\|_{L^p(0,\g,w_{\a_{p,\s}};H^{2-\s,q})}^p
\stackrel{(i)}{\lesssim}_{\theta,\ellip,R} 1+\|v_0\|_{B^0_{q,p}}^p\stackrel{(ii)}{\lesssim}_{\theta,\ellip,R,p,q} 1+\|v_0\|_{L^q}^p.
\end{equation}
Here in $(i)$ we used that $\frac{1+\a_{p,\s}}{p}=1-\frac{\s}{2}$ and that the space for the initial data is 
$(H^{-\s,q}H^{2-\s,q})_{\frac{\s}{2},p}=B^0_{q,p}$ and in $(ii)$ that $L^q\embed B^0_{q,p}$  as $p\geq q$. 
The implicit constants in \eqref{eq:estimate_g_R} depends on $(\theta,\ellip)$ which has been fixed so that \eqref{eq:uniform_estimate_smooth_case_L_r_q} holds. In particular they are independent of $v_0\in \KN$.
Hence, the estimate \eqref{eq:uniform_estimate_smooth_case_H_s_q} follows from \eqref{eq:estimate_g_R}, the Chebyshev inequality and the fact that 
$\P(\g= T)>1-\varepsilon$.

\emph{Step 4: Proof of \eqref{eq:claim_convergence_v_n_theta_v_star}}. Fix $\varepsilon\in (0,1)$.
By contradiction, assume that \eqref{eq:claim_convergence_v_n_theta_v_star} does not hold, i.e.\
\begin{equation*}
\limsup_{n\to \infty}\sup_{v_0\in \KN} \P\big( \| \vcn(v_0)-\vd(v_0)\|_{L^r(0,T;L^q)}\geq \varepsilon\big)>0.
\end{equation*}
Thus there exists a (not-relabeled) subsequence of data $(v_0^{(n)})_{n\geq 1}\subseteq \KN$ such that 
\begin{equation}
\label{eq:contradiction_star}
\lim_{n\to \infty} \P\big( \| \vcn(v_0^{(n)})-\vd(v_0^{(n)})\|_{L^r(0,T;L^q)}\geq \varepsilon\big)>0.
\end{equation}
Moreover, up to extract a further subsequence, we can assume that, as $n\to \infty$,
\begin{equation}
\label{eq:weak_convergence_contradiction}
v_0^{(n)} \rightharpoonup v_0 \text{ in }L^q,\  \text{ for some }\ \  \ v_0\in L^q \ \text{ such that } \ \|v_0\|_{L^q}\leq N.
\end{equation}
Note that  $v_0\in \LN$ as $\KN \ni v^{(n)}_0\geq 0$ on $\Tor^d$ for all $n\geq 1$, see 
\eqref{eq:KN_definition}-\eqref{eq:LN_definition}. The choice of $\ellip$ and the comments below \eqref{eq:R_bound_det_solution} show that there exists a $(p,q)$--solution $\vd^{(n)}(v_0)$ to \eqref{eq:reaction_diffusion_deterministic_cut_off_weak_limit} on $[0,T]$ such that $\sup_{n\geq 1}\|\vd(v_0^{(n)})\|_{\Y}<\infty$ where $\Y$ is as in Proposition \ref{prop:weak_to_strong_convergence_deterministic_problem}. 
Recall that, due to \eqref{eq:R_bound_det_solution}, $\vd(v_0^{(n)})$ are actually $(p,q)$--solutions to \eqref{eq:reaction_diffusion_deterministic} with $\mu_i=\ellip_i+\ellip$ provided by Proposition \ref{prop:global_high_viscosity} and therefore in the class considered in Proposition \ref{prop:weak_to_strong_convergence_deterministic_problem}. By Corollary \ref{cor:uniqueness} and \eqref{eq:R_bound_det_solution}, we also have that $\vd(v_0)\in \X$ is also unique in the class of weak solutions, where $\X$ is as in Proposition \ref{prop:weak_to_strong_convergence_deterministic_problem}.
Hence, the latter result ensures that 
$$
\vd(v^{(n)}_0)\to \vd(v_0) \ \text{ in }L^r(0,T;L^q) \ \ \   \text{  as }\ n\to \infty.
$$
The above and \eqref{eq:contradiction_star} yield
\begin{equation}
\label{eq:contradiction_star_2}
\limsup_{n\to \infty} \P\Big( \| \vcn(v_0^{(n)})-\vd(v_0)\|_{L^r(0,T;L^q)}\geq \frac{\varepsilon}{2}\Big)>0.
\end{equation}
Next we derive a contradiction with Theorem \ref{t:weak_convergence}. To this end we first check its assumptions \eqref{it:weak_convergence_L_q}-\eqref{it:determinstic_limit_cut_off} of Theorem \ref{t:weak_convergence}. Note that 
\eqref{it:weak_convergence_L_q} follows from \eqref{eq:weak_convergence_contradiction} and $v_0^{(n)}\in \KN$ for all $n \geq 1$.
\eqref{it:theta_goes_to_zero} follows from the above choice of $\theta^{(n)}$ as in \eqref{eq:choice_theta}. Finally, \eqref{it:determinstic_limit_cut_off}  follows from $v_0\in \LN$ and the comments below \eqref{eq:R_bound_det_solution}. Let us stress that the uniqueness part of the assumption \eqref{it:determinstic_limit_cut_off} in Theorem \ref{t:weak_convergence}  follows from Corollary \ref{cor:uniqueness} and \eqref{eq:R_bound_det_solution}. Hence Theorem \ref{t:weak_convergence} is applicable and it yields \eqref{eq:claim_scaling_limit} with $v^{(n)}=\vcn(v_0^{(n)})$  and $v=\vd(v_0)$. 
The latter gives a contradiction with \eqref{eq:contradiction_star_2} and completes Step 4.
\end{proof}

To prove 
Theorem \ref{t:delayed_blow_up} we use a density argument and the fact that the conditions in Proposition \ref{prop:delayed_blow_up_smooth} are uniformly w.r.t.\ $\|v_0\|_{L^q}$.
To set up a convenient density argument we need an additional estimate for stochastic reaction diffusion equations with a \emph{modified} cut-off. The choice of the cut-off is now inspired by the estimates \eqref{eq:uniform_estimate_smooth_case_L_r_q}-\eqref{eq:uniform_estimate_smooth_case_H_s_q}.

Fix $K>0$, $\s\in (1,2)$ and $\eta>0$. As in Proposition \ref{prop:delayed_blow_up_smooth}, we set $\a_{p,\s}=p(1-\frac{\s}{2})-1$.
Let $\phi\in C^{\infty}(\R)$ be such that $\phi|_{[0,1]}=1$ and $\phi|_{[2,\infty)}=0$. Finally, set
\begin{equation}
\label{eq:cut_off_psi}
\Phi_{K,r,\s,\eta}(t,v)\stackrel{{\rm def}}{=}\phi\big(K^{-1}\|v\|_{L^r(0,t;L^q)})\cdot \phi\big(K^{-1} \|v\|_{L^p(0,t,w_{\a_{p,\s}};H^{2-\s-\eta,q})}\big).
\end{equation}
Consider the following stochastic reaction equations with (a modified) cut-off:
\begin{equation}
\label{eq:reaction_diffusion_system_truncation_double}
\left\{
\begin{aligned}
\dd v_i -\ellip_i\Delta v_i \,\dd t
&= \Phi_{K,r,\s,\eta} (\cdot,v)\Big[\div (F(\cdot,v))+f_{i}(\cdot, v)\Big]\,\dd t \\
&+ \sqrt{c_d\ellip} \sum_{k,\alpha} \theta_k (\sigma_{k,\alpha}\cdot \nabla) v_i\circ \dd w_t^{k,\alpha}, &  \text{ on }&\Tor^d,\\
v_i(0)&=v_{i,0},  &  \text{ on }&\Tor^d.
\end{aligned}
\right.
\end{equation}
The notion of
$(p,\a,\s,q)$--solutions to \eqref{eq:reaction_diffusion_system_truncation_double} can be given as in Definition \ref{def:solution}.

The main difference of \eqref{eq:reaction_diffusion_system_truncation_double} compared to \eqref{eq:reaction_diffusion_system_truncation} analyzed in Section \ref{s:global_cut_off} is that the action of the cut-off $\Phi_{K,r,\s,\eta}(\cdot,v)$ is stronger than the one used in \eqref{eq:reaction_diffusion_system_truncation}, i.e.\ \eqref{eq:def_cut_off}. 
Let us note that the truncation chosen in \eqref{eq:reaction_diffusion_system_truncation_double} is too strong to run the arguments of Section \ref{s:global_cut_off}. On the other hand, the one in \eqref{eq:def_cut_off} seems not enough to obtain the stability estimate of Lemma \ref{l:stability_estimate_cut_off} below (cf.\ Remark \ref{r:blow_up_criteria_necessity}). Such estimate is the last ingredient in the proof of Theorem \ref{t:delayed_blow_up}. To this end, we need the following estimates.

\begin{lemma}
\label{l:interp_inequality_final_one}
Let Assumption \ref{ass:f_polynomial_growth}\eqref{it:integrability_exponent_main_assumption}--\eqref{it:f_polynomial_growth_1} be satisfied.  
Assume that $
q>\frac{d(h-1)}{2}\vee \frac{d}{d-\s} $ for some $\s\in (1,2)$. There there exist $r_1\in (1,\infty)$ and $\eta_1>0$, depending only on $(h,d,q,\s)$, such that the following estimate holds for all $r\in [r_1,\infty)$ and $\eta\in (0,\eta_1]$:
\begin{multline*}
\|f(\cdot,u^{(1)})-f(\cdot,u^{(2)})\|_{L^p(0,T,w_{\a_{p,\s}},H^{-\s,q})}
+\|\div(F(\cdot,u^{(1)}))-\div(F(\cdot,u^{(2)}))\|_{L^p(0,T,w_{\a_{p,\s}},H^{-\s,q})} \\
\lesssim (1+\|u^{(1)}\|_{\Sp}^{h-1}+\|u^{(2)}\|_{\Sp}^{h-1})\|u^{(1)}-u^{(2)}\|_{\Sp},
\end{multline*}
for all  $u^{(1)},u^{(2)}\in \Sp\stackrel{{\rm def}}{=}L^r(0,T;L^q)\cap L^p(0,T,w_{\a_{p,\s}};H^{2-\s-\eta,q})$.
\end{lemma}

\begin{proof}
The proof follows the argument in Step 2 of Proposition \ref{prop:delayed_blow_up_smooth}. We content ourself to prove the estimate for $f(\cdot,u^{(1)})-f(\cdot,u^{(2)})$ as the other one is similar. To economize the notation, in the proof below, we write $\a$ instead of $\a_{p,\s}=p(1-\frac{\s}{2})-1$ if no confusion seems likely. 

By \cite[Lemma 3.2]{AV22}, there exists $\beta\in (1-\frac{1+\a}{p},1-\frac{h}{h-1}\frac{1+\a}{p})$ such that 
$$
\|f(\cdot,u^{(1)})-f(\cdot,u^{(2)})\|_{H^{-\s,q}}\lesssim (1+\|u^{(1)}\|_{H^{-\s+2\beta,q}}^{h-1}+\|u^{(2)}\|_{H^{-\s+2\beta,q}}^{h-1})
\|u^{(1)}-u^{(2)}\|_{H^{-\s+2\beta,q}}.
$$
Let us recall that $\beta<1-\frac{h}{h-1}\frac{1+\a}{p}$ is equivalent to the \emph{subcriticality} of the $(p,\a,\s,q)$-setting. 
Since 
$
\big\||g_1|^{h-1} |g_2|\big\|_{L^p(0,T,w_{\a})}\leq \|g_1\|_{L^{ph}(0,T,w_{\a})}^{h-1} \|g_2\|_{L^{ph}(0,T,w_{\a})}
$ by H\"{o}lder inequality, it remains to show the existence of some $r\in (1,\infty)$ and $\eta>0$ such that
\begin{equation}
\label{eq:embedding_Z}
L^r(0,T;L^q)\cap L^p(0,T,w_{\a};H^{2-\s-\eta,q}) \embed L^{ph}(0,T,w_{\a};H^{-\s+2\beta,q}).
\end{equation}
To prove \eqref{eq:embedding_Z} one can argue as follows. By interpolation, for all $\g\in (0,1)$, 
\begin{equation}
\label{eq:interpolation_z_intermediate_step}
L^r(0,T;L^q)\cap L^p(0,T,w_{\a};H^{2-\s-\eta,q})\embed
L^{r_{\g}}(0,T,w_{\a_{\g}};H^{\g(2-\s-\eta),q})
\end{equation}
where 
$$
\frac{1}{r_{\g}}= \frac{1-\g}{r}+ \frac{\g}{p} \qquad \text{ and }\qquad
\frac{\a_{\g}}{r_{\g}}=\frac{\g \a}{p}.
$$
Without loss of generality we may assume $2\beta-\s<2-\s-\eta$ as $\beta<1$. Hence, we can choose $\g\stackrel{{\rm def}}{=}\frac{2\beta-\s}{2-\s-\eta}\in (0,1)$ in the above. Note that $H^{\g(2-\s-\eta),q}=H^{-\s+2\beta,q}$. By \eqref{eq:interpolation_z_intermediate_step} and \cite[Proposition 2.1(3)]{AV19_QSEE_2}, \eqref{eq:embedding_Z} follows provided 
\begin{equation}
\label{eq:r_theta_condition_embedding_zeta}
r_{\g}>ph \qquad \text{ and }\qquad \frac{1+\a_{\g}}{r_{\g}}<\frac{1+\a}{p}.
\end{equation}
By continuity, \eqref{eq:r_theta_condition_embedding_zeta} holds provided it holds for $\eta=0$ and $r=\infty$. In the latter case $\a_{\g}=\a$, $r_{\g}=\frac{p}{\g}$ and $\g=\frac{2\beta-\s}{2-\s}<1$. Thus, in that case, the second in \eqref{eq:r_theta_condition_embedding_zeta} is automatically satisfied. It remains to check that $r_{\g}>ph$. Since $\a=p(1-\frac{\s}{2})-1$, we have $\g=\frac{\beta-1+\frac{1+\a}{p}}{\frac{1+\a}{p}}$ and therefore
$$
r_{\g}>ph
\qquad \Longleftrightarrow\qquad
\frac{1}{r_{\g}}= \frac{\g}{p}= \frac{\beta-1+\frac{1+\a}{p}}{1+\a}< \frac{1}{ph} .
$$
The latter condition holds as it is equivalent to $\beta< 1-\frac{h}{h-1}\frac{1+\a}{p}$ which holds by construction.
\end{proof}

The next result is the last ingredient we need to prove Theorem \ref{t:delayed_blow_up}.

\begin{lemma}[Stability estimate for \eqref{eq:reaction_diffusion_system_truncation_double}]
\label{l:stability_estimate_cut_off}
Fix $T\in (0,\infty)$. Suppose that Assumption \ref{ass:f_polynomial_growth}. 
Let $K>0$ and $\s\in (1,2)$ (where $h$ is as in Assumption \ref{ass:f_polynomial_growth}) and assume that
$$
q>\frac{d(h-1)}{2} \vee  \frac{d}{d-\s} \qquad \text{ and }\qquad p\geq \frac{2}{2-\s}\vee q.
$$
Let $(r_1,\eta_1)$ be as in Lemma \ref{l:interp_inequality_final_one} and fix $r\in [r_1,\infty)$, $\eta\in (0,\eta_1)$.
Then for each $v_0\in L^q$, there exists a (unique) global $(p,\a_{p,\s},\s,q)$--solution $v_{(K,r,\s,\eta)}(v_0)$ to \eqref{eq:reaction_diffusion_system_truncation_double} on $[0,T]$.
Moreover, there exists a constant $C_0(p,q,K,r,\s,\theta,\eta,T)>0$ such that, for all $v_{0}^{(1)},v_{0}^{(2)}\in L^q$,
\begin{equation}
\label{eq:stability_estimate_K_delta}
\E\|v^{(1)}-v^{(2)}\|_{L^p(0,T,w_{\a_{p,\s}};H^{2-\s,q}) \cap L^r(0,T;L^q)}^p
\leq C_0\|v_{0}^{(1)}-v_{0}^{(2)}\|_{L^q}^p,
\end{equation}
where $v^{(j)}\stackrel{{\rm def}}{=}v_{(K,r,\s,\eta)}(v_0^{(j)})$ is the solution to \eqref{eq:reaction_diffusion_system_truncation_double} with data $v_0^{(j)}$.
\end{lemma}

\begin{proof}
The existence of a (unique) global $(p,\a_{p,\s},\s,q)$--solution to \eqref{eq:reaction_diffusion_system_truncation_double} follows as in the proof of Theorem \ref{t:global_cut_off} with minor modifications. To avoid repetitions, we only give the proof of  \eqref{eq:stability_estimate_K_delta}. 
To economize the notation, in this proof, we write $\Phi(t,v)$ instead of $\Phi_{K,r,\s,\eta}(t,v)$ and we let
\begin{align}
\label{eq:X_space_stability_proof}
\Sp(t)&\stackrel{{\rm def}}{=}L^p(0,t,w_{\a_{p,\s}};H^{2-\s-\eta,q})\cap L^r(0,t;L^q)& \text{ for }& t>0,\\
\label{eq:non_stability_proof}
\non(\cdot,v)&\stackrel{{\rm def}}{=} \Phi (\cdot,v)\big[\div (F(\cdot,v))+f_{i}(\cdot, v)\big]& \text{ for }& v\in 	\Sp(T).
\end{align}

\emph{Step 1: There exists $C(h,d,q,\s,p)>0$ such that, for all $u^{(1)},u^{(2)} \in \Sp(T)$,}
\begin{align*}
\|\non(\cdot,u^{(1)})-\non(\cdot,u^{(2)})\|_{L^p(0,T,w_{\a_{p,\s}};H^{-\s,q})}
&\lesssim \|u^{(1)}-u^{(2)}\|_{\Sp(T)}.
\end{align*}

The proof of Step 1 follows as the one for Step 2 in Theorem \ref{t:global_cut_off}\eqref{it:global_cut_off_1}. 
For the reader's convenience we give a sketch.
For $j\in \{1,2\}$, fix $u^{(j)} \in\Sp(T)$ and set
$$
e^{(j)} \stackrel{{\rm def}}{=}\inf\big\{t\in [0,T]\,:\,\|u^{(j)}\|_{\Sp(t)}\geq K\big\} \quad \text{ where }\quad \inf\emptyset \stackrel{{\rm def}}{=}T.
$$
Without loss of generality we assume that $e_2\leq e_1$.  
Note that 
\begin{align*}
\non(\cdot,u^{(1)})
-
\non(\cdot,u^{(2)})
&=\underbrace{\Big(\Phi (\cdot ,u^{(1)})- \Phi (\cdot ,u^{(2)})\Big) \big[\div (F(\cdot,u^{(1)}))+f_{i}(\cdot, u^{(1)})\big]}_{I_1\stackrel{{\rm def}}{=}}\\
&+\underbrace{\Phi (\cdot ,u^{(2)})\Big(\div (F(\cdot,u^{(1)}))+f_{i}(\cdot, u^{(1)})-\div (F(\cdot,u^{(2)}))-f_{i}(\cdot, u^{(2)})\Big)}_{I_2\stackrel{{\rm def}}{=} }.
\end{align*}
Note that $\Phi (s,u^{(1)})- \Phi (s,u^{(2)})=0$ for all $s\geq e_1$ since $e_2\leq e_1$. The definition of $e_1$ and Lemma \ref{l:interp_inequality_final_one} yield
\begin{align*}
I_1
& \leq \Big(\sup_{t\in [0,T]}|\Phi (\cdot,u^{(1)})- \Phi (\cdot,u^{(2)})|\Big)\big\|\div (F(\cdot,u^{(1)}))+f_{i}(\cdot, u^{(1)})\big\|_{L^p(0,e_1,w_{\a};H^{-\s,q})}\\
&\lesssim_{K} \|u^{(1)}-u^{(2)}\|_{\Sp(T)} .
\end{align*}
Similarly, by Lemma \ref{l:interp_inequality_final_one} we have 
\begin{align*}
I_2 
&\lesssim \big\|F(\cdot,u^{(1)})-F(\cdot,u^{(2)})\big\|_{L^p(0,e_2,w_{\a};H^{1-\s,q})}
+\big\|f_{i}(\cdot, u^{(1)})-f_{i}(\cdot, u^{(2)})\big\|_{L^p(0,e_2,w_{\a};H^{-\s,q})}\\
&\lesssim_K \|u^{(1)}-u^{(2)}\|_{\Sp(T)}.
\end{align*}
The claim of Step 1 follows by collecting the estimates for $I_1$ and $I_2$.

\emph{Step 2: There exists $N(p,q,K,r,\s,\eta,\theta,T)>0$ such that}
\begin{equation}
\begin{aligned}
\label{eq:claim_Step_2_stability_estimate}
&\E\|v^{(1)}-v^{(2)}\|_{L^p(0,T,w_{\a_{p,\s}};H^{2-\s,q})}^p+
\E\|
\non(\cdot,u^{(1)})
-
\non(\cdot,u^{(2)})\|_{L^p(0,T,w_{\a_{p,\s}};H^{-\s,q})}^p\\
&\qquad \qquad \qquad \qquad \qquad  
\leq N\|v_0^{(1)}-v_0^{(2)}\|_{L^q}^p +
N\E\|v^{(1)}-v^{(2)}\|_{L^p(0,T,w_{\a_{p,\s}};H^{-\s,q})}^p.
\end{aligned}
\end{equation}

The point in \eqref{eq:claim_Step_2_stability_estimate} is that we are able to bound the maximal regularity $L^p(w_{\a_{p,\s}};H^{2-\s,q})$--norm of the difference  $v^{(1)}-v^{(2)}$ in term of the weaker $L^p(w_{\a_{p,\s}};H^{-\s,q})$-one.

First we estimate $\E\|v^{(1)}-v^{(2)}\|_{L^p(0,T,w_{\a_{p,\s}};H^{2-\s,q})}^p$. To this end, set $\vdiff\stackrel{{\rm def}}{=}v^{(1)}-v^{(2)}$. Note that, for $i\in \{1,\dots,\ell\}$, $V$ is a $(p,\a_{p,\s},\s,q)$--solution to 
\begin{equation}
\label{eq:truncated_psi_stability_proof_equation_difference}
\begin{aligned}
\dd V_i -(\ellip_i+\ellip)\Delta V_i \,\dd t
&=\Big[\non(\cdot,v^{(1)})-\non(\cdot,v^{(2)})\Big]\,\dd t \\
& + \sqrt{c_d\ellip} \sum_{k,\alpha=1} \theta_n (\sigma_{k,\alpha}\cdot \nabla) V_i\, \dd w_t^{k,\alpha}  \qquad   \text{on }\Tor^d.
\end{aligned}
\end{equation}
The above problem is complemented with the initial condition 
\begin{equation}
\label{eq:V_initial_data}
V(0)=V_0 \ \ \text{ where } \ \ V_0\stackrel{{\rm def}}{=}v_0^{(1)}-v_0^{(2)}.
\end{equation}
The stochastic maximal $L^p$-regularity estimate of \cite[Theorem 1.2]{AV21_SMR_torus} yields, for all $\g\in [0,\frac{1}{2})$,
\begin{align}
\label{eq:smr_applied_to_V}
\E\|V\|_{H^{\g,p}(0,T;w_{\a_{p,\s}};H^{2-\s-2\g,q})}^p
&\lesssim \|V_0\|_{L^q}^p+
\E\|\non(\cdot,u^{(1)})-\non(\cdot,u^{(2)})\|_{L^{p}(0,T;w_{\a_{p,\s}};H^{-\s,q})}^p\\
\nonumber
&\lesssim \|V_0\|_{L^q}^p+\E \|V\|_{\Sp(T)}^p
\end{align}
where in the last inequality we apply Step 1. Recall that $\Sp$ is as in \eqref{eq:X_space_stability_proof}.
 
Now the idea is to prove that $\Sp(T)$ is lower order compared to the maximal regularity norms. More precisely we prove the existence of $\g\in (0,\frac{1}{2})$ such that, for each $\varepsilon\in (0,1)$,
\begin{equation}
\label{eq:Sp_is_lower_order}
\|u\|_{Z_{r,\eta}(T)}\leq \varepsilon \|u\|_{H^{\g,p}(0,T;w_{\a_{p,\s}};H^{2-\s-2\g,q})\cap L^{p}(0,T;w_{\a_{p,\s}};H^{2-\s,q})}+
C_{\varepsilon} \|u\|_{L^{p}(0,T;w_{\a_{p,\s}};H^{-\s,q})}.
\end{equation}
If on the LHS\eqref{eq:Sp_is_lower_order} we replace $\Sp(T)$ by $L^p(0,T,w_{\a_{p,\s}};H^{2-\s-\eta,q})$, then the above fact is clear from standard interpolation inequalities and the fact that $\eta>0$. By \eqref{eq:X_space_stability_proof} it remains to prove the estimate \eqref{eq:Sp_is_lower_order} with $\Sp(T)$ 
replaced by $L^r(0,T;L^q)$. To this end, fix $\zeta\in (r\vee p\vee \frac{2}{2-\s},\infty)$ and set $\g\stackrel{{\rm def}}{=}1-\frac{\s}{2}-\frac{1}{\zeta}\in(0,\frac{1}{2})$. Note that $\g-\frac{1+\a_{p,\s}}{p}=-\frac{1}{\zeta}$ and $2-\s-2\g=\frac{2}{\zeta}$. Sobolev embeddings with power weights show (see e.g.\ \cite[Proposition 2.7]{AV19_QSEE_1}) 
\begin{equation}
\label{eq:sob_embedding_Lr_0}
H^{\g,p}(0,T,w_{\a_{p,\s}};H^{2-\s-2\g,q})=
H^{\g,p}(0,T,w_{\a_{p,\s}};H^{2/\zeta,q})\embed L^{\zeta} (0,T;H^{2/\zeta,q}). 
\end{equation}
By standard interpolation inequality, one has, for all $\varphi\in (0,1)$
$$
\|u\|_{L^{p_{\varphi}}(0,T,w_{\a_{\varphi}};H^{s_{\varphi},q})}\lesssim \|u\|_{L^{p}(0,T,w_{\a_{p,\s}};H^{-\s,q})}^{1-\varphi}
\|u\|_{L^{\zeta} (0,T;H^{2/\zeta,q}) }^{\varphi} 
$$
 where $\frac{1}{p_{\varphi}}=\frac{1-\varphi}{p}+\frac{1}{\zeta}$, 
$\frac{\a_{\varphi}}{p_{\varphi}}=\frac{(1-\varphi)\a}{p}$ and $s_{\varphi}=-\s(1-\varphi)+ \frac{2 \varphi}{\zeta}$. 
Note that $H^{s_{\varphi},q}\embed L^q$ for all $\varphi\in [\frac{\s\zeta}{2+\s\zeta},1)$. By continuity, one sees that there exists $\varphi_0(\zeta,\a,p)\in (\frac{\s\zeta}{2+\s\zeta},1)$ such that $\frac{1+\a_{\varphi_0}}{p_{\varphi_0}}<\frac{1}{r}$. Indeed, by letting $\varphi\uparrow 1$ the previous condition is equivalent to $\frac{1}{\zeta}<\frac{1}{r}$ which holds since $\zeta>r$. Hence,  the H\"{o}lder inequality yields
$
L^{p_{\varphi_0}}(0,T,w_{\a_{\varphi_0}};L^q)
\embed L^r(0,T;L^q),
$
cf.\ \cite[Proposition 2.1(3)]{AV19_QSEE_2}.
Collecting the previous observations, one sees that \eqref{eq:Sp_is_lower_order} with $\Sp(T)$ 
replaced by $L^r(0,T;L^q)$ follows by $\varphi_0<1$, Young inequality and \eqref{eq:sob_embedding_Lr_0}.

By \eqref{eq:smr_applied_to_V}--\eqref{eq:Sp_is_lower_order} with $\varepsilon>0$ small enough, one gets, for $\g=1-\frac{\s}{2}-\frac{1}{\zeta}$,
\begin{align*}
\E\|v^{(1)}-v^{(2)}\|_{H^{\g,p}(0,T,w_{\a_{p,\s}};H^{2-\s-2\g,q})\cap L^p(0,T,w_{\a_{p,\s}};H^{2-\s,q})}^p&\\
\leq N\big[\|v_0^{(1)}-v_0^{(2)}\|_{L^q}^p &+
\E\|v^{(1)}-v^{(2)}\|_{L^p(0,T,w_{\a_{p,\s}};H^{-\s,q})}^p\big]
\end{align*}
for some $N(p,q,K,r,\s,\eta,\theta,T)>0$.
The estimate for the second term on the LHS\eqref{eq:claim_Step_2_stability_estimate} 
follows by combining Step 1, \eqref{eq:Sp_is_lower_order} and the previous estimate. 
This concludes the proof of  \eqref{eq:claim_Step_2_stability_estimate}.

\emph{Step 3: Conclusion}. Let $(V,V_0)$ be as in Step 2, see \eqref{eq:truncated_psi_stability_proof_equation_difference}--\eqref{eq:V_initial_data}. By \eqref{eq:truncated_psi_stability_proof_equation_difference}, the Burkholder-Davis-Gundy inequality yields, for some constant $c_0(p,q,K,r,\s,\theta,\eta)>0$ and for all $t\in [0,T]$ (see \cite[Theorem 4.15]{AV19_QSEE_2} for similar computations)
\begin{align*}
&\E\|V\|_{C([0,t];H^{-\s,q})}^p\\
&\leq c_0\Big[\|V_0\|_{L^q}^p+ \E\|V\|_{L^p(0,t,w_{\a_{p,\s}};H^{2-\s,q})}^p+\E\|\non(\cdot,u^{(1)})-\non(\cdot,u^{(2)})\|_{L^p(0,T,w_{\a_{p,\s}};H^{-\s,q})}^p \Big]\\
&\leq c_0 N \Big[\|V_0\|_{L^q}^p + 
\E\|V\|_{L^p(0,t,w_{\a_{p,\s}};H^{-\s,q})}^p\Big]
\end{align*}
where in the last estimate we used Step 2.  

Setting $X_t\stackrel{{\rm def}}{=}\E \|V\|_{C([0,t];H^{-\s,q})}^p$, the above inequality yields $X_t \lesssim X_0 +  \int_0^t X_s\,\dd s$ for all $t\in [0,T]$ (recall that $\a_{p,\s}\geq 0$). Thus the Grownall inequality shows that 
$
X_T \lesssim_T X_0,
$
i.e.\
$$
\E\|V\|_{C([0,T];H^{-\s,q})}^p\lesssim \|V_0\|_{L^q}^p .
$$
The estimate \eqref{eq:stability_estimate_K_delta} follows from the above inequality, \eqref{eq:claim_Step_2_stability_estimate} and $V=v^{(1)}-v^{(2)}$.
\end{proof}

\begin{proof}[Proof of Theorem \ref{t:delayed_blow_up}]
Fix $(N,T,\varepsilon,\ellip_0,r)$.
Let $(p,\a_{p,\s},\s,q)$ and $(r_1,\eta_1)$ be as in  \eqref{eq:p_a_p_s_assumption_local} and Lemma \ref{l:stability_estimate_cut_off}, respectively. 
Without loss of generality we assume that $r\in[ r_1,\infty)$. Finally, fix $\eta\in (0,\eta_1]$.

Since $\|v_0\|_{L^q}\leq N$, for each $\g\in (0,1)$ there exists $v_0^{(\g)}$ such that 
\begin{equation}
\label{eq:decomposition_data_L_q_data}
 v_0^{(\g)}\in C^{\infty}   \qquad \text{ and }\qquad  
\|v_0-v_0^{(\g)}\|_{L^q}\leq \g.
\end{equation}
In particular $\|v_0^{(\g)}\|_{L^q}\leq N+1$. 
For all $\g\in (0,1)$, let $(v^{(\g)},\tau^{(\g)})$ be the $(p,0,1,q)$--solution to  \eqref{eq:reaction_diffusion} provided by Theorem \ref{t:local} and Remark \ref{r:regularity_paths}\eqref{it:regularity_paths_2}. 

Let $(\theta,\ellip,R,K_0)$ be as in Proposition \ref{prop:delayed_blow_up_smooth} with $\a=0$ and $(T,\ellip_0,r,p,q)$ as above and $(N,\varepsilon)$ replaced by $(N+1,\frac{\varepsilon}{12})$. Note that $(\ellip,\theta)$ are independent of $v_0$ satisfying \eqref{eq:data_L_q_N_statement}. Since $v_0^{(\g)}$ is smooth, Proposition \ref{prop:delayed_blow_up_smooth} 
applies with the above choice of $(\ellip,\theta,R)$ and it ensures that 
\begin{align}
\label{eq:consequence_proposition_eta_lives_up_to_T_smooth_0}
\P\Big(\tau^{(\g)}\geq T,\, \|v^{(\g)}-\vd^{(\g)}\|_{L^r(0,T;L^q)}\leq\frac{\varepsilon}{6}\Big)&>1-\frac{\varepsilon}{6},\\
\label{eq:consequence_proposition_eta_lives_up_to_T_smooth}
\P\Big(\tau^{(\g)}\geq T,\|v^{(\g)}\|_{L^r(0,T;L^q)\cap L^p(0,T,w_{\a_{p,\s}};H^{2-\s,q})}\leq K_0+R\Big)
&<1-\frac{\varepsilon}{6},
\end{align}
where $\vd^{(\g)}$ is the $(p,q)$--solution to \eqref{eq:reaction_diffusion_deterministic} on $[0,T]$ with $\mu_i=\ellip+\ellip_i$ and initial data $v_0^{(\g)}$. Note that the existence of $\vd^{(\g)}$ is also part of the result of Proposition \ref{prop:delayed_blow_up_smooth}.

\emph{Step 1: Theorem \ref{t:delayed_blow_up}\eqref{it:delayed_blow_up} holds and there exists $\g_0\in (0,1)$ such that }
\begin{equation}
\label{eq:step_1_additional_claim_proof_main_result}
\P\Big(\tau\wedge \tau^{(\g)}\geq T,\, \|v-v^{(\g)}\|_{L^r(0,T;L^q)}\leq\frac{\varepsilon}{3}\Big)>1-\frac{\varepsilon}{3} \ \ \ \text{ for all }\g\in (0,\g_0).
\end{equation}
Consider the truncated problem \eqref{eq:reaction_diffusion_system_truncation_double} with 
$K\stackrel{{\rm def}}{=}K_0+R+1$ and $(r,\eta)$ as the beginning of the current proof. 
Let us denote by $v_{(K,r,\s,\eta)}$ and $v_{(K,r,\s,\eta)}^{(\g)}$ the $(p,\a_{p,\s},\s,q)$--solution \eqref{eq:reaction_diffusion_system_truncation_double} with initial data $v_0$ and $v_0^{(\g)}$, respectively. 
By Lemma \ref{l:stability_estimate_cut_off} and Chebyshev's inequality, we can find $\g_0(N,\varepsilon,T,\ellip_0,r,p,q)\in (0,1)$ such that, for all $\g\in (0,\g_0)$, 
\begin{align*}
\P(\V)>1-\frac{\varepsilon}{6} \ \  \text{ where }\ \ 
\V\stackrel{{\rm def}}{=}\Big\{\|v_{(K,r,\s,\eta)}-v_{(K,r,\s,\eta)}^{(\g)}\|_{L^r(0,T;L^q)\cap L^p(0,T,w_{\a_{p,\s}};H^{2-\s,q})}\leq \frac{\varepsilon}{3}\Big\}.
\end{align*}

Since $K\geq K_0+R$, the uniqueness of $v_{(K,r,\s,\eta)}^{(\g)}$ yields
\begin{equation}
\begin{aligned}
\label{eq:v_K_s_equal_v}
\tau\wedge T&\geq  T\  \text{ on }\V_0\quad \text{ and }\quad v_{(K,r,\s,\eta)}^{(\g)}=v^{(\g)} \  \text{ a.e.\ on }[0,\tau^{(\g)}\wedge T]\times \V_0,\\
\V_0&\stackrel{{\rm def}}{=}\Big\{\tau^{(\g)}\geq T,\|v^{(\g)}\|_{L^r(0,T;L^q)}+ \|v^{(\g)}\|_{L^p(0,T,w_{\a_{p,\s}};H^{2-\s,q})}\leq  K_0+R\Big\}.
\end{aligned}
\end{equation}
To see \eqref{eq:v_K_s_equal_v} one can argue as Step 1 of Proposition \ref{prop:delayed_blow_up_smooth}. Indeed,  let
$$
\mu\stackrel{{\rm def}}{=}\inf\{t\in [0,\tau)\,:\, \|v^{(\g)}\|_{L^r(0,t;L^q)}+ \|v^{(\g)}\|_{L^p(0,t,w_{\a_{p,\s}};H^{2-\s,q})} \geq K_0+R\}\wedge T,
$$
where $\inf\emptyset\stackrel{{\rm def}}{=}\tau^{(\g)}\wedge T$. Note that $\mu=\tau^{(\g)}\wedge T=T$ on $\V_0$.
Then $\Phi_{K,r,\s,\eta}(\cdot,v_{(K,r,\s,\eta)}^{(\g)})=1$ a.e.\ on $[0,\mu)\times \O$, see \eqref{eq:cut_off_psi}. Therefore $(v^{(\g)},\mu)$ is a $(p,\a_{p,\s},\s,q)$--solution to \eqref{eq:reaction_diffusion_system_truncation_double} with data $v_0^{(\eta)}$. Combining the uniqueness of $v_{(K,r,\s,\eta)}$ and the fact that $\{\mu=\tau^{(\g)}\wedge T\}\supseteq \V_0$, one obtains \eqref{eq:v_K_s_equal_v}.

Next, note that, by \eqref{eq:consequence_proposition_eta_lives_up_to_T_smooth} and \eqref{eq:v_K_s_equal_v}, we have $\P(\V_0)>1-6^{-1}\varepsilon$.  
Therefore
\begin{equation}
\label{eq:lower_bound_prob_W}
\P(\W)>1-\frac{\varepsilon}{3}, \quad \text{ where }\quad \W\stackrel{{\rm def}}{=}\V \cap \V_0.
\end{equation}
Recall that $K=K_0+R+1$.
The triangular inequality, \eqref{eq:v_K_s_equal_v} and the definitions of $\V_0$ yield
\begin{align*}
\|v_{(K,r,\s,\eta)}\|_{L^r(0,T;L^q)\cap L^p(0,T,w_{\a_{p,\s}};H^{2-\s,q})}
\leq K\ \ \  \text{ a.s.\ on } \W.
\end{align*}

Arguing as below \eqref{eq:v_K_s_equal_v}, the above and a stopping time argument readily yields 
\begin{equation}
\label{eq:v_v_K_s_equality}
\tau\geq T \ \text{ a.s.\ on } \W\quad \text{ and } \quad v_{(K,r,\s,\eta)}=v \ \text{ a.e.\ on }[0,\tau\wedge T)\times \W.
\end{equation} 
The first in \eqref{eq:v_v_K_s_equality} and \eqref{eq:lower_bound_prob_W} prove assertion \eqref{it:delayed_blow_up} of  Theorem \ref{t:delayed_blow_up}. 

Finally, to prove \eqref{eq:step_1_additional_claim_proof_main_result}, note that, the definition of $\V\supseteq \W$, the fact that $\tau^{(\g)}\geq T$ on $\V_0$, \eqref{eq:v_K_s_equal_v} and \eqref{eq:v_v_K_s_equality} imply
$$
\Big\{
\tau\wedge \tau^{(\g)}\geq T,\, \|v-v^{(\g)}\|_{L^r(0,T;L^q)}\leq\frac{\varepsilon}{3}\Big\}\supseteq \W.
$$
Thus \eqref{eq:step_1_additional_claim_proof_main_result} 
follows from \eqref{eq:lower_bound_prob_W}.

\emph{Step 2: Theorem \ref{t:delayed_blow_up}\eqref{it:enhanced_dissipation} holds}. 
Let $\vd$ be as described below \eqref{eq:consequence_proposition_eta_lives_up_to_T_smooth}. 
By Proposition \ref{prop:global_high_viscosity}\eqref{it:sol_operator_continuous} there exists $\g_1(v_0,T,\varepsilon
,q,p)\in (0,1)$ such that, for all $\g\in (0,\g_1)$,
\begin{equation}
\label{eq:continuity_eta_det}
\|\vd-\vd^{(\g)}\|_{L^r(0,T;L^q)}\leq \frac{\varepsilon}{3}.
\end{equation}
Next fix $\g\in (0,\g_0\wedge \g_1)$. 
The triangular inequality shows that 
\begin{align*}
&\{\tau\geq T,\, \|v-\vd\|_{L^r(0,T;L^q)}\leq\varepsilon\}
\supseteq \Big\{\tau\wedge \tau^{(\g)}\geq T,\|v-v^{(\g)}\|_{L^r(0,T;L^q)}\leq\frac{\varepsilon}{3}\Big\}\\
&\quad \cap \Big\{\tau^{(\g)}\geq T,\, \|v^{(\g)}-\vd^{(\g)}\|_{L^r(0,T;L^q)}\leq \frac{\varepsilon}{3}\Big\}
\cap \Big\{\|\vd^{(\g)}-\vd\|_{L^r(0,T;L^q)}\leq\frac{\varepsilon}{3}\Big\}.
\end{align*}
Therefore Theorem \ref{t:delayed_blow_up}\eqref{it:enhanced_dissipation} follows by combining the latter inclusions and \eqref{eq:consequence_proposition_eta_lives_up_to_T_smooth_0}, \eqref{eq:step_1_additional_claim_proof_main_result}, \eqref{eq:continuity_eta_det}.
\end{proof}

\subsection{Proof of Theorem \ref{t:delayed_blow_up_t_infty}}
\label{ss:proof_global_infty}
Following the arguments of \cite{FL19,FGL21} we deduce Theorem \ref{t:delayed_blow_up_t_infty} from Theorem \ref{t:delayed_blow_up} and Lemma \ref{l:global_high_viscosity_infty}. As in  \cite{FL19} we need that the \emph{stochastic} problem \eqref{eq:reaction_diffusion} is globally well-posed for small initial data, see assumption b) in \cite[Theorem 1.5]{FGL21}. 
This will be the content of the following result.

\begin{proposition}[Global existence with small initial data]
\label{prop:small_global_q_0}
Let Assumption \ref{ass:f_polynomial_growth} be satisfied. Let $N\geq 1$ and let $v_0\in L^q(\Tor^d;\R^{\ell})$ be such that $\|v_0\|_{L^q}\leq N$ and $v_0\geq 0$ on $\Tor^d$ (component-wise). Suppose that Assumption \ref{ass:f_polynomial_growth}\eqref{it:mass} holds with $\m_0=0$ and $\m_1<0$. 
Assume that 
$$
\s\in (1,2), \ \ \ \  q>\frac{d(h-1)}{2}\vee \frac{d}{d-\s} \ \ \  \text{ and } \ \ \ p\geq \frac{2}{2-\s}\vee q.
$$ 
Let $(v,\tau)$ be the $(p,\a_{p,\s},q,\s)$-solution to \eqref{eq:reaction_diffusion} provided by Theorem \ref{t:local} (recall $\a_{p,\s}=p(1-\frac{\s}{2})-1$). 
Then for each 
$$
\varepsilon\in (0,1), \ \ \ r\in (2,\infty) \ \ \ \text{ and }\  \ \  \frac{d(h-1)}{2}\vee 2\leq q_0<q_1\leq  q
$$ 
there exist $S,\eta>0$, depending only on $(\m_j,\alpha_i,r,q_0,q_1,r, q,p,d,h,N,\varepsilon)$, such that for all stopping time $\g\in [S,\infty)$ a.s.
$$
\P \big(\tau>\g,\,\|v(\g)\|_{L^{q_1}}\leq \eta \big) \geq 1-\varepsilon 
\ \  \Longrightarrow  \ \
\P\big(\tau=\infty,\, \|v\|_{L^r(\g,\infty;L^{q_0})}\leq \varepsilon\big)\geq 1-\varepsilon.
$$
\end{proposition}

Proposition \ref{prop:small_global_q_0} ensures the absence blow-up with high probability provided $\|v(\g)\|_{L^{q_1}}$ is small with high probability as well.  
The smallness of the norm $\|v(\g)\|_{L^{q_0}}$ is not surprising as the mass is exponentially decreasing by Theorem \ref{t:local}\eqref{it:local_1} with $\m_0=0$ and $\m_1<0$. 

Next we prove Proposition \ref{prop:small_global_q_0} and afterwards Theorem \ref{t:delayed_blow_up_t_infty}.

\begin{proof}[Proof of Proposition \ref{prop:small_global_q_0}]
The proof follows as the one of Lemma \ref{l:global_high_viscosity_infty} (see also Proposition \ref{prop:global_high_viscosity}). Here we use the smallness of $\eta$ instead of the one of $\mu^{-1}$.
Let $(r,q_0,q_1,\varepsilon)$ be as in the statement of Proposition \ref{prop:small_global_q_0}. 
For $\eta,S>0$ and a stopping time $\g:\O\to [S,\infty)$, set 
\begin{equation}
\label{eq:v_gamma_eta}
\V_{\g,\eta}\stackrel{{\rm def}}{=}\{\tau> \g,\,\|v(\g)\|_{L^{q_0}}\leq \eta\}\in \F_{\gamma}.
\end{equation}

Below
we also write $\V$ instead of $\V_{\g,\eta}$ if no confusion seems likely. Below we assume that 
\begin{equation}
\label{eq:global_small_data_assumption_epsilon}
\P(\V)>1-\varepsilon.
\end{equation}
As in Lemma \ref{l:global_high_viscosity_infty}, below, we frequently use that the exponential decay of the mass:
\begin{equation}
\label{eq:exponential_decay_proof}
\int_{\Tor^d} |v|\,\dd x \lesssim_N e^{-|\m_1| t }\ \ \ \text{ a.s.\ for all }t\in [0,\tau). 
\end{equation}
The above follows from Theorem \ref{t:local}\eqref{it:local_1},  $\m_0=0$, $\m_1<0$ and $\|v_0\|_{L^q}\leq N$.

\emph{Step 1: Let $(S,\eta)$ be positive constants. There exist a constant $c_0(\m_i,\alpha_j,q_0,q_1,r,q,p,d,h)\geq 1$, independent of $(S,\eta)$, such that a.s.\ on $\V$ and for all $t\in [\g,\tau)$ } 
\begin{align}
\label{eq:step_1_small_data_sup_norm}
\sup_{t\in [\g,t)}\|v\|^{q_1}_{L^{q_1}}
&\leq c_0\Big(\eta^{q_1}+ e^{- |\m_1| S }+ \int_{\g}^t \int_{\Tor^d}  |v|^{q_1+h-1}\,\dd x\, \dd s \Big),\\
\label{eq:step_1_small_data_q_norm}
\|v\|_{L^{q_1+h-1}(\g,t;L^{q_1+h-1})}^{q_1}
&\leq c_0\Big(\eta^{q_1}+ e^{-|\m_1| S  }+ \int_{\g}^t \int_{\Tor^d}  |v|^{q_1+h-1}\,\dd x\, \dd s\Big).
\end{align} 
Here we follow the proof of Steps 1--2 in Proposition \ref{prop:global_high_viscosity}. Recall that $v$ is regular on $(0,\tau)\times \O$ by Theorem \ref{t:local}\eqref{it:local_2}. Thus, the It\^{o} formula yields, a.s.\ for all $t\in [\g,\tau)$ and $i\in \{1,\dots,\ell\}$,
\begin{align*}
\|v_i(t)\|_{L^{q_1}}^{q_1}
&+\ellip_i q (q-1) \int_{\g}^t\int_{\Tor^d} |v_i|^{q_1-2}|\nabla v_i|^2 \,\dd x\, \dd s\\
&=\|v_i(\g)\|_{L^{q_1}}^{q_1}
+q \int_{\g}^{t}\int_{\Tor^d} |v_i|^{q_1-2}\big[ f_i(\cdot,v)v_i - (q-1)F_i (\cdot,v)\cdot \nabla v_i \big]\,\dd x\, \dd s.
\end{align*}
As in the proof of Theorem \ref{t:global_cut_off}\eqref{it:global_cut_off_2} the martingale part in the previous identity vanishes since $\div\,\sigma_{k,\alpha}=0$. Next we estimate the RHS of the previous inequality. Thus
\begin{multline*}
\Big|\int_{\g}^{t}\int_{\Tor^d} |v_i|^{q_1-2} f_i(\cdot,v)v_i\, \dd x\, \dd s\Big|
\lesssim 
\int_{\g}^{t}\int_{\Tor^d}| v_i|+ |v_i|^{q_1+h-1} \,\dd x\, \dd s\\
\lesssim
\int_{\g}^{t}\int_{\Tor^d} e^{-|\m_1| t } + |v_i|^{q_1+h-1} \,\dd x\, \dd s
\lesssim  e^{-|\m_1| S}+\int_{\g}^{t}\int_{\Tor^d} |v_i|^{q_1+h-1} \,\dd x\, \dd s,
\end{multline*}
where  we used \eqref{eq:exponential_decay_proof} and $\g\in [S,\infty)$ a.s. 
In the above the implicit constants depends only on $(\m_i,\alpha_j,q_0,q_1,q,p,d,h,N)$.
Reasoning as in Step 1 (resp.\ 2) of Proposition \ref{prop:global_high_viscosity}, one can check that \eqref{eq:step_1_small_data_sup_norm} (resp.\  \eqref{eq:step_1_small_data_q_norm}) holds. To avoid repetitions we omit the details.

\emph{Step 2: There exists $M_0(\m_i,\alpha_j,q_0,q_1,q,p,d,h)>0$ for which the following assertion holds. Suppose that $S\geq S_0$ and $\eta\geq \eta_0$ satisfy}
\begin{equation}
\label{eq:eta_S_small_big_m_0}
\eta_0^{q_1}+ e^{-|\m_1| S_0}< M_0.
\end{equation}
\emph{Then there exists $K(\m_i,\alpha_j,q_0,q_1,q,p,d,h)>0$ such that}
\begin{equation}
\label{eq:estimate_v_stochastic_q_1}
\sup_{t\in [S,\tau)} \|v(t)\|_{L^{q_1}} \leq K\ \ \text{ a.s.\ on }\V.
\end{equation}
To prove \eqref{eq:estimate_v_stochastic_q_1} we argue as in Step 3 of Proposition \ref{prop:global_high_viscosity}. Let $\zeta\stackrel{{\rm def}}{=} \frac{h-1}{q_1}>1$ and $\psi_{\zeta,R}(x)\stackrel{{\rm def}}{=} R\, x- x^{1+\zeta}$ for $x\in [0,\infty)$ where $R=c_0^{-1}$. Here $c_0$ is as in Step 1. Then \eqref{eq:step_1_small_data_q_norm} is equivalent to  
$$
\psi_{\zeta,R}\big(\|v\|_{L^{q_1+h-1}(0,t;L^{q_1+h-1})}^{q_1}\big)\leq \eta^{q_1}+ e^{-|\m_1| S}\quad \text{ for all }0\leq t<\tau\text{ and a.s.\ on $\V$}.
$$
Note that $\psi_{\zeta,R}$ has unique maximum on $[0,\infty)$. Set $M_0(c_0,\zeta)\stackrel{{\rm def}}{=}\max_{\R_+} \psi_{\zeta,R}$ and  
$m_0(c_0,\zeta)\stackrel{{\rm def}}{=}\argmax_{\R_+} \psi_{\zeta,R}$. Next we choose $(S_0,\eta_0)$ as in \eqref{eq:eta_S_small_big_m_0} with $M_0$ as above. Then, if $S\geq S_0$ and $\eta\leq \eta_0$, then the above inequality readily yields (cf.\ Figure \ref{fig:1} for a similar situation)
$$
\|v\|_{L^{q_1+h-1}(0,t;L^{q_1+h-1})}^{q_1}\leq m_0 \ \ \ \text{ for all $0\leq t<\tau$}.
$$
The estimate \eqref{eq:estimate_v_stochastic_q_1} follows by combining the previous inequality and \eqref{eq:step_1_small_data_sup_norm}.

\emph{Step 3: Let $(S_0,\eta_0)$ be as in Step 2. If $S\geq S_0$ and $\eta\leq  \eta_0$, then $\tau=\infty$ a.s.\ on $\V$}. 
 By Theorem \ref{t:local}\eqref{it:local_3} and the fact that $q_1>\frac{d(h-1)}{2}\vee 2$ we have, for all $0<s<T<\infty$, 
\begin{align*}
\P(\{s<\tau<T\}\cap \V)
\stackrel{\eqref{eq:estimate_v_stochastic_q_1}}{\leq} \P\Big(s<\tau<T,\, \sup_{t\in [s,\tau)} \|v(t)\|_{L^{q_1}}<\infty\Big)=0.
\end{align*}
By letting $s\downarrow 0$ and $T\uparrow \infty$, we have $\P(\{\tau<\infty\}\cap \V)=0$. Hence $\tau=\infty$ a.s.\ on $\V$.

\emph{Step 4: Let $\varepsilon>0$ and assume that \eqref{eq:global_small_data_assumption_epsilon} holds. Let $(S_0,\eta_0,K)$ be as in Step 2. Then there exists $S_1>0$ depending only on $(\m_i,\alpha_j,r,q_0,q_1,q,p,d,h,N,\varepsilon)$ such that, if $S\geq S_1\vee S_0$ and $\eta\leq \eta_0$, then we have}
$$
\P(\tau=\infty,\, \|v\|_{L^r(\g,\tau;L^{q_0})}\leq \varepsilon)>1-\varepsilon.
$$
To prove the claim, by Step 3 and the fact that $\P(\V)>1-\varepsilon$, it is enough to show that 
\begin{equation}
\label{eq:W_estimate_conclusion_long_time_existence_small_data}
\|v\|_{L^r(\g,\infty;L^{q_0})} \leq \varepsilon \quad \text{ a.s.\ on }\V.
\end{equation}
Recall that $\g\in [S,\infty)$ a.s. By interpolating \eqref{eq:exponential_decay_proof} and \eqref{eq:estimate_v_stochastic_q_1}, for some $\m_0(\m_i,q_0,q_1)>0$, we obtain $\|v(t)\|_{L^{q_0}}\lesssim e^{-t |\m_1|}$ a.s.\ on $\V$ for all $t\geq S$. Where the implicit constants depends only on $(\m_i,\alpha_j,q_0,q_1,q,p,d,h,N,\varepsilon)$.
Hence \eqref{eq:W_estimate_conclusion_long_time_existence_small_data} follows by choosing $S_1$ large enough.
\end{proof}

The proof of Theorem \ref{t:delayed_blow_up_t_infty} follows by combining the Theorem \ref{t:delayed_blow_up}, Proposition \ref{prop:small_global_q_0} and the exponential decay of solution to \eqref{eq:reaction_diffusion_deterministic} shown in Lemma \ref{l:global_high_viscosity_infty}. 
For the reader's convenience, before going into the proof, we summarize the main argument. By Theorem \ref{t:delayed_blow_up} and Lemma \ref{l:global_high_viscosity_infty}, we know that $v(\g)-\vd(\g)$ and $\vd(\g)$ are small provided $\g\geq S$ is big enough (here $v$ and $\vd$ is the solution to \eqref{eq:reaction_diffusion} and \eqref{eq:reaction_diffusion_deterministic}, respectively). Thus $v(\g)$ is small as well. Hence Theorem \ref{t:delayed_blow_up_t_infty} follows from the previous observation and Proposition \ref{prop:small_global_q_0}.

\begin{proof}[Proof of Theorem \ref{t:delayed_blow_up_t_infty}]
Let $(N,\varepsilon,\ellip_0,r,q_0)$ be as in the statement of Theorem \ref{t:delayed_blow_up_t_infty}. Recall that $q_0<q$ and without loss of generality we may assume that $q_0>\frac{d(h-1)}{2}\vee \frac{d}{d-\s}$. Finally fix $q_1\in (q_0,q)$.

Next we collect some further parameters which are independent of $v_0$ satisfying \eqref{eq:data_L_q_N_statement_global}. 
Let $\mu_0>0$ be as in Lemma \ref{l:global_high_viscosity_infty} and let $(S,\eta)$ be as in Proposition \ref{prop:small_global_q_0} with $\varepsilon$ replaced by $\frac{\varepsilon}{2}$. 
Lemma \ref{l:global_high_viscosity_infty} ensures the existence of $T>0$, independent of $v_0$ satisfying \eqref{eq:data_L_q_N_statement_global}, for which the following assertion is satisfied provided $\ellip\geq \mu_0$: 
For all $v_0$ as in \eqref{eq:data_L_q_N_statement_global}, there exists a $(p,q)$-solution $\vd$ on $[0,\infty)$ to the deterministic problem \eqref{eq:reaction_diffusion_det_statements} satisfying
\begin{align}
\label{eq:decay_infinity_vd_final_proof}
\|\vd\|_{L^r(T,\infty;L^{q_1})}+\sup_{t\geq T}\|\vd(t)\|_{L^{q_1}}\leq   \frac{\varepsilon\wedge \eta}{4}.
\end{align}
Without loss of generality we may assume $S\leq T$ and $\ellip_0\leq \mu_0$.  

Let $(\ellip,\theta)$ be such that Theorem \ref{t:delayed_blow_up} holds with 
$$
(N,T,\varepsilon,\ellip_0,r)\quad \text{ replaced by }\quad \Big(N,T+1, \frac{\varepsilon\wedge \eta}{4}, \mu_0,r\Big).
$$ 
Note that $(\ellip,\theta)$ is independent of $v_0$ satisfying \eqref{eq:data_L_q_N_statement_global}  due to the independence of $(\ellip,\theta)$ on the initial data in Theorem \ref{t:delayed_blow_up} and the choice of $(S,\eta,\mu_0)$.
With the above choice of the parameters we can now complete the proof of Theorem \ref{t:delayed_blow_up_t_infty}. 
Indeed, Theorem \ref{t:delayed_blow_up}\eqref{it:enhanced_dissipation} ensures that 
\begin{equation}
\label{eq:enhanced_dissipation_T_1_final}
\P\Big(\tau\geq T+1,\, \|v-\vd\|_{L^r(0,T+1;L^q)}\leq\frac{\varepsilon\wedge \eta}{4}\Big)
>1-\frac{\varepsilon\wedge \eta}{4}
\end{equation}
where $\vd$ is $(p,q)$--solution to \eqref{eq:reaction_diffusion_det_statements} as described before \eqref{eq:decay_infinity_vd_final_proof}.

Note that \eqref{eq:enhanced_dissipation_T_1_final} together with \eqref{eq:decay_infinity_vd_final_proof} show
\begin{equation}
\label{eq:V_varepsilon_final_proof}
\P(\V_{\varepsilon})
>1-\frac{\varepsilon}{2} \quad \text{ where }\quad 
\V_{\varepsilon}\stackrel{{\rm def}}{=}\Big\{\tau\geq T+1,\, \|v\|_{L^r(T,T+1;L^{q_1})}\leq\frac{ \varepsilon\wedge \eta}{2}\Big\}.
\end{equation}
Recall that, by Theorem \ref{t:local}\eqref{it:local_2}, the paths $[s,\tau) \ni t\mapsto \|v(t)\|_{L^{q_1}}$ are  a.s.\ continuous for all $s>0$. By \eqref{eq:V_varepsilon_final_proof}, for each $\om \in \V_{\varepsilon}$ there exists $t\in (T,T+1)$ such that $\|v(t,\om)\|_{L^{q_1}}\leq  \eta$. 
Hence the stopping time 
$$ 
\g\stackrel{{\rm def}}{=} \one_{\{\tau>T\}}
\left\{
\begin{aligned}
&\inf\{t\in [T,\tau)\,:\, \|v(t)\|_{L^{q_1}}\leq \eta\}\wedge (T+1) &\text{ on }  \ &\{\|v(T)\|_{L^{q_1}}> \eta \},\\
&T& \text{ on }\ &\{\|v(T)\|_{L^{q_1}}\leq \eta \},
\end{aligned}
\right.
$$
where $\inf\emptyset\stackrel{{\rm def}}{=} \tau$,
satisfies 
$$
\P(\tau>\g\vee T,\, \|v(\g\vee T)\|_{L^{q_1}}\leq \eta)\geq \P(\V_{\varepsilon})>1-\frac{\varepsilon}{2}.
$$
The previous and 
Proposition \ref{prop:small_global_q_0} yield
\begin{equation}
\label{eq:conclusion_t_infty_final_step}
\P\Big(\tau=\infty,\, \|v\|_{L^r(\gamma\vee T,\infty;L^{q_0})}\leq \frac{\varepsilon}{2}\Big)\geq 1-\frac{\varepsilon}{2}.
\end{equation}
The above already proves Theorem \ref{t:delayed_blow_up_t_infty}\eqref{it:delayed_blow_up_global}. While \eqref{it:enhanced_dissipation_global} follows by combining \eqref{eq:decay_infinity_vd_final_proof},  \eqref{eq:enhanced_dissipation_T_1_final}, \eqref{eq:conclusion_t_infty_final_step} and the fact that $\g\vee T\in [T,T+1]$ a.s.
\end{proof}

\bigskip

\noindent {\bf Acknowledgements.}
The author thanks Lorenzo Dello Schiavo, Lucio Galeati and Mark Veraar for helpful comments.
The author acknowledges Caterina Balzotti for her support in creating the picture. The author thanks the anonymous referee for helpful comments.

\medskip

\noindent {\bf Declarations -- Data availability.} This manuscript has no associated data.

\bibliographystyle{alpha-sort}
\bibliography{literature}

\end{document}